\documentclass[10pt,a4paper,reqno]{amsart}
\usepackage{amsmath,amssymb,amsthm,graphicx,mathrsfs,url}
\usepackage[usenames,dvipsnames]{color}
\definecolor{darkred}{rgb}{0.4,0.1,0.1}

\definecolor{darkblue}{rgb}{0.1,0.1,0.7}
\usepackage[usenames,dvipsnames]{color}
\usepackage{tikz}

\newcommand\comm[1]{}
\newcommand\AB{A_{[B]}}
\numberwithin{equation}{section}
\theoremstyle{plain}% default
\newtheorem{theorem}{Theorem}[section]
\newtheorem*{theorem*}{Theorem}

\newtheorem{lemma}[theorem]{Lemma}
\newtheorem{proposition}[theorem]{Proposition}
\newtheorem{corollary}[theorem]{Corollary}

\newtheorem{definition}[theorem]{Definition}
\theoremstyle{remark}
\newtheorem{remark}[theorem]{Remark}
\theoremstyle{definition}
\newtheorem{example}[theorem]{Example}

\newcommand{\drm}{{\mathrm d}}

\newcommand{\R}{\mathbb{R}}
\newcommand{\C}{\mathbb{C}}
\newcommand{\Z}{\mathbb{Z}}

\DeclareMathOperator{\dist}{dist}
\DeclareMathOperator{\diag}{diag}

\newcommand\cB{\mathcal B}

\newcommand\cF{\mathcal F}
\newcommand\cG{\mathcal G}
\newcommand\cH{\mathcal H}

\newcommand\cL{\mathcal L}

\newcommand\cP{\mathcal P}

\newcommand\cS{\mathcal S}
\newcommand\CC{\mathbb C}
\newcommand\NN{\mathbb N}
\newcommand\RR{\mathbb R}
\newcommand\bbS{\mathbb S}
\newcommand\bbU{\mathbb U}
\newcommand\ZZ{\mathbb Z}

\newcommand\eps{\varepsilon}
\renewcommand{\phi}{\varphi}

\newcommand\ov{\overline}
\newcommand\wt{\widetilde}

\newcommand{\defeq}{\mathrel{\mathop:}=}

\newcommand\sess{\sigma_{\rm ess}}

\newcommand{\rmO}{{\rm O}}
\newcommand{\rmo}{{\rm o}}

\newcommand\AD{A_{\rm D}}
\newcommand\AN{A_{\rm N}}

\newcommand\Hlocunif{H_{\rm loc,unif}}

\DeclareMathOperator\dom{dom}
\DeclareMathOperator\ran{ran}

\DeclareMathOperator\real{Re}
\DeclareMathOperator\imag{Im}
\renewcommand\Re{\real}
\renewcommand\Im{\imag}

\newcommand\void[1]{}

\newcounter{counter_a}
\newenvironment{myenum}{\begin{list}{{\rm(\roman{counter_a})}}%
{\usecounter{counter_a} \setlength{\itemsep}{0.5ex}\setlength{\topsep}{0.7ex}
\setlength{\leftmargin}{6ex}\setlength{\labelwidth}{6ex}}}{\end{list}}

\newcounter{counter_b}
\newenvironment{myenuma}{\begin{list}{{\rm(\alph{counter_b})}}%
{\usecounter{counter_b} \setlength{\itemsep}{0.5ex}\setlength{\topsep}{0.7ex}
\setlength{\leftmargin}{6ex}\setlength{\labelwidth}{6ex}}}{\end{list}}

\def\sS{{\mathfrak S}}

      \def\dC{{\mathbb C}}

   \def\dN{{\mathbb N}}   
      \def\dR{{\mathbb R}}
      \def\dU{{\mathbb U}}
\def\dV{{\mathbb V}}   \def\dW{{\mathbb W}}   
   \def\dZ{{\mathbb Z}}

   \def\cB{{\mathcal B}}   
      \def\cF{{\mathcal F}}
\def\cG{{\mathcal G}}   \def\cH{{\mathcal H}}   
      \def\cL{{\mathcal L}}
      
\def\cP{{\mathcal P}}      
\def\cS{{\mathcal S}}   \def\cT{{\mathcal T}}

\title[Spectral enclosures for non-self-adjoint extensions]{Spectral
enclosures for non-self-adjoint extensions of symmetric operators}
\author[J.~Behrndt]{Jussi Behrndt}
\address{Technische Universit\"{a}t Graz,
Institut f\"{u}r Angewandte Mathematik, \newline
Steyrergasse 30,
8010 Graz, Austria}
\email{behrndt@tugraz.at}

\author[M.~Langer]{Matthias Langer}
\address{Department of Mathematics and Statistics,
University of Strathclyde, \newline
26 Richmond Street, Glasgow G1 1XH, United Kingdom}
\email{m.langer@strath.ac.uk}

\author[V.~Lotoreichik]{Vladimir Lotoreichik}
\address{Department of Theoretical Physics, Nuclear Physics Institute CAS, \newline
250 68 \v{R}e\v{z} near Prague, Czech Republic}
\email{lotoreichik@ujf.cas.cz}

\author[J.~Rohleder]{Jonathan Rohleder}
\address{Stockholms universitet, Matematiska institutionen, 10691 Stockholm, Sweden}
\email{jonathan.rohleder@math.su.se}

\begin{document}

\begin{abstract}
The spectral properties of non-self-adjoint extensions $A_{[B]}$ of a symmetric operator
in a Hilbert space are studied with the help of ordinary and quasi boundary triples
and the corresponding Weyl functions.  These extensions are given in terms of
abstract boundary conditions involving an (in general non-symmetric) boundary operator $B$.
In the abstract part of this paper, sufficient conditions for sectoriality
and m-sectoriality as well as sufficient conditions for $A_{[B]}$ to have a non-empty
resolvent set are provided in terms of the parameter $B$ and the Weyl function.
Special attention is paid to Weyl functions that decay along the negative real line
or inside some sector in the complex plane, and spectral enclosures for $A_{[B]}$
are proved in this situation.
The abstract results are applied to elliptic differential operators with
local and non-local Robin boundary conditions on unbounded domains,
to Schr\"{o}dinger operators with $\delta$-potentials of complex strengths
supported on unbounded hypersurfaces or infinitely many points on the real line,
and to quantum graphs with non-self-adjoint vertex couplings.
\\[1ex]
\\[1ex]
\textit{Mathematics Subject Classification (MSC2010):} 47A10;
35P05, 35J25, 81Q12, 35J10, 34L40, 81Q35.
\end{abstract}

\maketitle

%\tableofcontents

\section*{Contents}

\noindent
\begin{tabular}{rlr}
1. & Introduction & \pageref{sec:intro} \\
2. & Quasi boundary triples and their Weyl functions & \pageref{sec:QBT} \\
3. & Sectorial extensions of symmetric operators & \pageref{sec:sectorial} \\
4. & Sufficient conditions for closed extensions with non-empty resolvent set & \pageref{sec:closedoperators} \\
5. & Consequences of the decay of the Weyl function & \pageref{sec:consequences} \\
6. & Sufficient conditions for decay of the Weyl function & \pageref{sec:suffconddecay} \\
7. & Elliptic operators with non-local Robin boundary conditions & \pageref{sec:Robin} \\
8. & Schr\"odinger operators with $\delta$-interaction on hypersurfaces & \pageref{sec:delta} \\
9. & Infinitely many point interactions on the real line & \pageref{sec:pointinteractions} \\
10. & Quantum graphs with $\delta$-type vertex couplings & \pageref{sec:graphs}
\end{tabular}

% *******************************************************************
\section{Introduction}
\label{sec:intro}
% *******************************************************************

\noindent
Spectral problems for differential operators in Hilbert spaces and related
boundary value problems have attracted a lot of attention in the last decades
and have strongly influenced the development of modern functional analysis and operator theory.
For example, the classical treatment of Sturm--Liouville operators
and the corresponding Titchmarsh--Weyl theory in Hilbert spaces have led
to the abstract concept of boundary triples and their Weyl functions (see \cite{B76,DM91,GG91,K75}),
which is an efficient and well-established tool to investigate closed extensions
of symmetric operators and their spectral properties via abstract boundary maps and an
analytic function; see, e.g.\
\cite{AGW14, ABMN05,BGW09,BMNW08,BMNW17,BGP08,DHMS09,DM95,M10,MN11,S12}.
The more recent notion of quasi boundary triples and their Weyl functions are
inspired by PDE analysis in a similar way. This abstract concept from \cite{BL07,BL12}
is tailor-made for spectral problems involving elliptic partial differential operators
and the corresponding boundary value problems; the Weyl function of a quasi boundary triple
is the abstract counterpart of the Dirichlet-to-Neumann map.
For different abstract treatments of elliptic PDEs and Dirichlet-to-Neumann maps
we refer to the classical works~\cite{G68,V52} and the more recent approaches
\cite{AE11, AE15, AEKS14, BR15, DHMS12, GM08, GM11, GMZ07, GMZ09, GK14, ILP15, MPS16, P08, P16}.

To recall the notions of ordinary and quasi boundary triples in more detail,
let $S$ be a densely defined, closed, symmetric operator in a
Hilbert space $(\cH,(\cdot,\cdot)_\cH)$ and let $S^*$ denote its adjoint;
then $\{\cG,\Gamma_0,\Gamma_1\}$ is said to be an \emph{ordinary boundary triple}
for $S^*$ if $\Gamma_0,\Gamma_1:\dom S^*\rightarrow\cG$ are linear mappings from
the domain of $S^*$ into an auxiliary Hilbert space $(\cG,(\cdot,\cdot)_\cG)$
that satisfy the abstract Lagrange or Green identity
\begin{equation}\label{gijussi}
  (S^*f,g)_\cH - (f,S^*g)_\cH = (\Gamma_1 f,\Gamma_0 g)_\cG - (\Gamma_0 f,\Gamma_1 g)_\cG
  \quad \text{for all}\;\; f,g\in\dom S^*
\end{equation}
and a certain maximality condition.  The corresponding \emph{Weyl function} $M$
is an operator-valued function in $\cG$, which is defined by
\begin{equation}\label{mmjussi}
  M(\lambda)\Gamma_0 f = \Gamma_1 f,
  \qquad f\in\ker(S^*-\lambda),\; \lambda\in\rho(A_0),
\end{equation}
where $A_0=S^*\upharpoonright\ker\Gamma_0$ is a self-adjoint operator in $\cH$.
For a singular Sturm--Liouville expression $-\frac{\drm^2}{\drm x^2}+V$ in $L^2(0,\infty)$
with a real-valued potential $V\in L^\infty(0,\infty)$ the operators $S$ and $S^*$
can be chosen as the minimal and maximal operators, respectively, together
with $\cG=\dC$ and $\Gamma_0 f=f(0)$, $\Gamma_1 f= f'(0)$ for $f \in \dom S^*$;
in this case the corresponding abstract Weyl function coincides with the
classical Titchmarsh--Weyl $m$-function.

The notion of \emph{quasi boundary triples} is a natural generalization of the
concept above, inspired by, and developed for, the treatment of elliptic differential operators.
The main difference is, that the boundary maps $\Gamma_0$ and $\Gamma_1$ are
only defined on a subspace $\dom T$ of $\dom S^*$, where $T$ is an operator
in $\cH$ which satisfies $\overline T=S^*$.
The identities \eqref{gijussi} and \eqref{mmjussi} are only required to hold
for elements in $\dom T$; see Section~\ref{sec:QBT} for precise definitions.
For the Schr\"{o}dinger operator $-\Delta + V$ in $L^2(\Omega)$
with a real-valued potential $V \in L^\infty(\Omega)$
on a domain $\Omega\subset \dR^n$ with a sufficiently regular boundary $\partial\Omega$,
the operators $S$ and $S^*$
can again be taken as the minimal and maximal operator, respectively, and a
convenient choice for the domain of $T=-\Delta+V$ is $H^2(\Omega)$.
Then $\cG=L^2(\partial\Omega)$ and
$\Gamma_0 f=\partial_\nu f\vert_{\partial\Omega}$, $\Gamma_1 f=f\vert_{\partial\Omega}$
(where the latter denote the normal derivative and trace)
form a quasi boundary triple, and the corresponding Weyl function
is the energy-dependent Neumann-to-Dirichlet map.

The main focus of this paper is on non-self-adjoint extensions of $S$ that are
restrictions of $S^*$ parameterized by an ordinary or quasi boundary triple and
an (in general non-self-adjoint) boundary parameter, and to describe
their spectral properties.
For a quasi boundary triple $\{\cG,\Gamma_0,\Gamma_1\}$ and a linear operator $B$
in $\cG$ we consider the operator
\begin{equation}\label{ABjussi}
  \AB f = Tf, \qquad
  \dom \AB = \bigl\{f\in\dom T:\Gamma_0 f= B\Gamma_1 f\bigr\}
\end{equation}
in $\cH$. The principal results of this paper include (a) a sufficient condition
for $\AB$ to be m-sectorial and (b) enclosures for the numerical range
and the spectrum of the operator $\AB$ in parabola-type regions.
The latter make use of decay properties of the Weyl function $M$
along the negative half-axis or inside sectors in the complex plane;
in order to make these results easily applicable, we provide (c) an abstract
sufficient condition for the Weyl function to decay appropriately.
We point out that, to the best of our knowledge, these results are also new in
the special case of ordinary boundary triples.
While the operator $A_{[B]}$ can be regarded as a perturbation of
the self-adjoint operator $A_0$ in the resolvent sense, let us mention
that the spectra of additive non-self-adjoint perturbations of
self-adjoint operators were studied recently in, e.g.\ \cite{C17,CT16,D02,DHK09,F11}.
In the second half of the present paper, we provide applications of
these results to several classes of operators, namely to elliptic differential operators
with local and non-local Robin boundary conditions on domains with
possibly non-compact boundaries, to Schr\"odinger operators with $\delta$-interactions
of complex strength supported on hypersurfaces, to infinitely many point $\delta$-interactions
on the real line, and to quantum graphs with non-self-adjoint vertex couplings.

Let us explain in more detail the structure, methodology, and results of this paper.
After the preliminary Section~\ref{sec:QBT}, our first main result is
Theorem~\ref{thmsec}, where it is shown that, under certain assumptions on the
Weyl function and the boundary parameter $B$, the operator $\AB$ in \eqref{ABjussi} is sectorial,
and a sector containing the numerical range of $\AB$  is specified.
However, in applications it is essential to ensure that a sectorial
operator is m-sectorial; hence the next main objective is to prove that the resolvent set
of the operator $\AB$ in \eqref{ABjussi} is non-empty, which is a non-trivial
question particularly for quasi boundary triples.
This problem is treated in Section~\ref{sec:closedoperators}.
The principal result here is Theorem~\ref{mainthm},
in which we provide sufficient conditions for $\lambda_0\in\rho(\AB)$
in terms of the operator $M(\lambda_0)$ and the parameter $B$.
In this context also a Krein-type resolvent formula is obtained, and the adjoint of $\AB$
is related to a dual parameter $B'$; cf.\ \cite{BLL13IEOT, BLLR17}
for the special case of symmetric~$B$. We list various corollaries of
Theorem~\ref{mainthm} for more specialized situations.
We point out that an alternative description of sectorial and m-sectorial extensions
of a symmetric operator can be found in~\cite{A00,M06}; see also the review article~\cite{A12}
and \cite{AKT12,AP13,AP17}.
Section~\ref{sec:closedoperators} is complemented by two propositions
on Schatten--von Neumann properties for the resolvent difference of $\AB$ and $A_0$;
cf.\ \cite{BLL13IEOT, DM91} for related abstract results and,
e.g.\ \cite{BGLL15, BLL13delta, BS80, G84, LR12, M10} for applications to differential operators.
Such estimates can be used, for instance, to get bounds on the discrete spectrum of $\AB$;
cf.\ \cite{DHK09}.
In Section~\ref{sec:consequences} we consider the situation when the
Weyl function $M$ converges to $0$ in norm along the negative half-axis
or in some sector in the complex plane.
The most important result in this section is Theorem~\ref{thm:parabola}
where, under the assumption that $\|M(\lambda)\|$ decays like a power of $\frac{1}{|\lambda|}$,
the numerical range and the spectrum of $A_{[B]}$ are contained in a parabola-type region.
Spectral enclosures of this type with more restrictive assumptions on $B$ were obtained
for elliptic partial differential operators in~\cite{BF62, B65, F62};
similar enclosures for Schr\"odinger operators with complex-valued regular
potentials can be found in~\cite{AAD01, F11, LS09}.
They also appear in the abstract settings of so-called $p$-subordinate perturbations~\cite{W10}.
Finally, as the last topic within the abstract part of this paper,
we prove in Theorem~\ref{th:weyl_decay} that the Weyl function decays
along the negative real line or in suitable complex sectors with a certain
rate if the map $\Gamma_1\vert A_0-\mu\vert^{-\alpha}$ is bounded
for some $\alpha\in (0,\tfrac{1}{2}]$ and some $\mu\in\rho(A_0)$,
where the rate of the decay depends on $\alpha$. Example~\ref{exampleoptimal}
shows the sharpness of this result.

Our abstract results are applied in Section~\ref{sec:Robin} to elliptic
partial differential operators with (in general non-local) Robin boundary conditions
on domains with possibly non-compact boundaries; the class of admissible unbounded domains includes,
for instance, domains of waveguide-type as considered in~\cite{BK08, EK15}.
In Section~\ref{sec:delta} we apply our abstract results to Schr\"{o}dinger operators
in $\dR^n$ with $\delta$-potentials of complex strength supported on
(not necessarily bounded) hypersurfaces. We indicate also how our abstract methods
can be combined with very recent norm estimates from~\cite{GS15} in order to
obtain further spectral enclosures and to establish absence of non-real spectrum
for `weak' complex $\delta$-interactions in space dimensions $n\ge 3$ for compact hypersurfaces.
Finally, we apply our machinery to Schr\"{o}dinger operators on the real line
with non-Hermitian $\delta$-interactions supported on infinitely many points
in Section~\ref{sec:pointinteractions}, and to Laplacians on finite
(not necessarily compact) graphs with non-self-adjoint vertex couplings
in Section~\ref{sec:graphs}.
Each of these sections has the same structure: after the problem under
consideration is explained, first a quasi (or ordinary) boundary triple
and its Weyl function are provided; next a lemma on the decay of the Weyl
function is proved, and then a main result on spectral properties and enclosures
is formulated, which can be derived easily from that decay together with
the abstract results in the first part of this paper in each particular situation.
To illustrate the different types of boundary conditions and interactions,
more specialized cases and explicit examples are included in
Sections~\ref{sec:Robin}--\ref{sec:graphs}.

Finally, let us fix some notation.  By $\sqrt{\cdot}$ we denote the branch
of the complex square root such that $\imag \sqrt{\lambda} > 0$
for all $\lambda\in\dC\setminus[0,\infty)$.
Let us set $\dR_+\defeq[0,\infty)$ and
$\dC^\pm \defeq \{\lambda\in\dC: \pm \Im\lambda >0\}$.
Moreover, for any bounded, complex-valued function $\alpha$ we use
the abbreviation $\|\alpha\|_\infty \defeq \sup|\alpha|$.
The space of bounded, everywhere defined operators from a
Hilbert space $\cH_1$ to another Hilbert space $\cH_2$
is denoted by $\cB(\cH_1,\cH_2)$,
and we set $\cB(\cH_1)\defeq\cB(\cH_1,\cH_1)$.
The Schatten--von Neumann ideal that consists of all compact operators from $\cH_1$
to $\cH_2$ whose singular values are $p$-summable is denoted by $\sS_p(\cH_1,\cH_2)$,
and we set $\sS_p(\cH_1)\defeq\sS_p(\cH_1,\cH_1)$;
see, e.g.\ \cite{GK69} for a detailed study of the $\sS_p$-classes.
Furthermore, for each densely defined operator $A$ in a Hilbert space
we write $\Re A \defeq \frac{1}{2}(A+A^*)$ and $\Im A \defeq \frac{1}{2i}(A-A^*)$
for its real and imaginary part, respectively, and,
if $A$ is closed, we denote by $\rho(A)$ and $\sigma(A)$ its
resolvent set and spectrum, respectively.

% *******************************************************************
\section{Quasi boundary triples and their Weyl functions} \label{sec:QBT}
% *******************************************************************

\noindent
In this preparatory section we first recall the notion and some properties of
quasi boundary triples and their Weyl functions from \cite{BL07,BL12}.
Moreover, we discuss some elementary estimates and decay properties of the Weyl function.

In the following let $S$ be a densely defined, closed, symmetric operator
in a Hilbert space $\cH$.

\begin{definition}\label{qbtdef}
Let $T\subset S^*$ be a linear operator in $\cH$ such that $\overline T=S^*$.
A triple $\{\cG,\Gamma_0,\Gamma_1\}$ is called a \emph{quasi boundary triple}
for $T\subset S^*$ if $\cG$ is a Hilbert space and $\Gamma_0,\Gamma_1:\dom T\rightarrow\cG$
are linear mappings such that
\begin{myenum}
  \item the abstract Green identity
    \begin{equation}\label{green}
      (T f,g) - (f,Tg)
      = (\Gamma_1 f,\Gamma_0 g)
      - (\Gamma_0 f,\Gamma_1 g)
    \end{equation}
    holds for all $f,g\in\dom T$,
    where $(\,\cdot,\,\cdot)$ denotes the inner product both in $\cH$ and $\cG$;
  \item the map $\Gamma \defeq (\Gamma_0,\Gamma_1)^\top :
    \dom T \to \mathcal G \times \mathcal G$ has dense range;
  \item $A_0 \defeq T\upharpoonright \ker\Gamma_0$ is a self-adjoint operator in $\cH$.
\end{myenum}
If condition {\rm (ii)} is replaced by the condition
\begin{myenum}
 \item[{\rm(ii)'}] the map $\Gamma_0: \dom T \to \mathcal G$ is onto,
\end{myenum}
then $\{\cG,\Gamma_0,\Gamma_1\}$ is called a \emph{generalized boundary triple}
for $T\subset S^*$.
\end{definition}

The notion of quasi boundary triples was introduced in \cite[Definition~2.1]{BL07}.
The concept of generalized boundary triples appeared first
in \cite[Definition~6.1]{DM95}.  It follows from \cite[Lemma~6.1]{DM95}
that each generalized boundary triple is also a quasi boundary triple.
We remark that the converse is in general not true.
A quasi or generalized boundary triple reduces to an ordinary boundary triple
if the map $\Gamma$ in condition (ii) is onto (see \cite[Corollary~3.2]{BL07}).
In this case $T$ is closed and coincides with $S^*$,
and $A_0$ in condition (iii) is automatically self-adjoint. For the
convenience of the reader we recall the usual definition of ordinary boundary triples.

\begin{definition}\label{obtdef}
A triple $\{\cG,\Gamma_0,\Gamma_1\}$ is called an \emph{ordinary boundary triple}
for $S^*$ if $\cG$ is a Hilbert space and $\Gamma_0,\Gamma_1:\dom S^*\rightarrow\cG$
are linear mappings such that
\begin{myenum}
  \item the abstract Green identity
    \begin{equation}\label{green2}
      (S^* f,g) - (f,S^*g)
      = (\Gamma_1 f,\Gamma_0 g)
      - (\Gamma_0 f,\Gamma_1 g)
    \end{equation}
    holds for all $f,g\in\dom S^*$;
  \item the map $\Gamma \defeq (\Gamma_0,\Gamma_1)^\top : \dom S^* \rightarrow
    \mathcal G \times \mathcal G$ is onto.
\end{myenum}
\end{definition}

\medskip

We refer the reader to \cite{BL07,BL12} for a detailed study of quasi boundary triples,
to \cite{DHMS06,DM95} for generalized boundary triples
and to \cite{B76,BGP08,DM91,GG91,K75} for ordinary boundary triples.
For later purposes we recall the following result,
which is useful to determine the adjoint and a (quasi) boundary triple
for a given symmetric operator; see \cite[Theorem~2.3]{BL07}.

\begin{theorem}\label{thm:ratetheorem}
Let $\cH$ and $\cG$ be Hilbert spaces and let $T$ be a linear operator in $\cH$.
Assume that $\Gamma_0,\Gamma_1: \dom T\rightarrow\cG$ are linear mappings such
that the following conditions hold:
\begin{myenum}
\item
the abstract Green identity
\[
  (Tf,g)-(f,Tg)=(\Gamma_1 f,\Gamma_0 g)-(\Gamma_0 f,\Gamma_1 g)
\]
holds for all $f,g\in\dom T$;
\item
the map $(\Gamma_0,\Gamma_1)^\top: \dom T\rightarrow\cG\times\cG$
has dense range and $\ker\Gamma_0\cap\ker\Gamma_1$ is dense in $\cH$;
\item
$T\upharpoonright \ker\Gamma_0$ is an extension of a self-adjoint
operator $A_0$.
\end{myenum}
Then the restriction
\[
  S \defeq T\upharpoonright(\ker\Gamma_0\cap\ker\Gamma_1)
\]
is a densely defined closed symmetric operator in $\cH$,
$\overline T= S^*$, and $\{\cG,\Gamma_0,\Gamma_1\}$
is a quasi boundary triple for $T \subset S^*$ with $A_0=T\upharpoonright \ker\Gamma_0$.
If, in addition, the operator $T$ is closed or, equivalently, the map
$(\Gamma_0,\Gamma_1)^\top: \dom T\rightarrow\cG\times\cG$ is onto,
then $T=S^*$ and  $\{\cG,\Gamma_0,\Gamma_1\}$ is an ordinary boundary triple
for $S^*$ with $A_0=T\upharpoonright \ker\Gamma_0$.
\end{theorem}

In the following let $\{\cG,\Gamma_0,\Gamma_1\}$ be a quasi boundary triple
for $T \subset S^*$.  Since $A_0=T\upharpoonright\ker\Gamma_0$ is self-adjoint,
we have $\dC\setminus\dR\subset\rho(A_0)$, and
for each $\lambda\in\rho(A_0)$  the direct sum decomposition
\[
  \dom T = \dom A_0\,\dot +\,\ker(T-\lambda)
  = \ker\Gamma_0\,\dot +\,\ker(T-\lambda)
\]
holds. In particular, the restriction of the map $\Gamma_0$ to $\ker(T-\lambda)$
is injective.  This allows the following definition.

\begin{definition}\label{gmdef}
The \emph{$\gamma$-field} $\gamma$ and the \emph{Weyl function} $M$ corresponding to the
quasi boundary triple $\{\mathcal G,\Gamma_0,\Gamma_1\}$ are defined by
\[
  \lambda \mapsto \gamma(\lambda)\defeq\bigl(\Gamma_0\upharpoonright\ker(T-\lambda)\bigr)^{-1},
  \qquad \lambda \in \rho(A_{0}),
\]
and
\[
  \lambda \mapsto M(\lambda) \defeq \Gamma_1 \gamma(\lambda),
  \qquad \lambda \in \rho(A_{0}),
\]
respectively.
\end{definition}

The values $\gamma(\lambda)$ of the $\gamma$-field are operators defined
on the dense subspace $\ran\Gamma_0\subset\cG$ which map
onto $\ker(T-\lambda)\subset\cH$.  The values $M(\lambda)$ of the Weyl function
are densely defined operators in $\cG$ mapping $\ran\Gamma_0$ into $\ran\Gamma_1$.
In particular, if $\{\mathcal G,\Gamma_0,\Gamma_1\}$ is a generalized
or ordinary boundary triple, then $\gamma(\lambda)$ and $M(\lambda)$
are defined on $\cG=\ran\Gamma_0$, and it can be shown
that $\gamma(\lambda)\in\cB(\cG,\cH)$ and $M(\lambda)\in\cB(\cG)$
in this case.

Next we list some important properties of the $\gamma$-field and the
Weyl function corresponding to a quasi boundary triple $\{\mathcal G,\Gamma_0,\Gamma_1\}$,
which can be found in \cite[Proposition~2.6]{BL07} or \cite[Propositions~6.13 and 6.14]{BL12}.
These properties are well known for the $\gamma$-field and Weyl function
corresponding to a generalized or ordinary boundary triple.
Let $\lambda,\mu\in\rho(A_0)$.  Then the adjoint operator $\gamma(\lambda)^*$
is bounded and satisfies
\begin{equation}\label{gammastar}
  \gamma(\lambda)^*=\Gamma_1(A_0 - \overline \lambda)^{-1}\in\cB(\cH,\cG);
\end{equation}
hence also $\gamma(\lambda)$ is bounded and
$\overline{\gamma(\lambda)}=\gamma(\lambda)^{**}\in\cB(\cG,\cH)$.
One has the useful identity
\begin{equation}\label{gammalamu}
  \gamma(\lambda) = \bigl(I+(\lambda-\mu)(A_0-\lambda)^{-1}\bigr)\gamma(\mu)
  = (A_0-\mu)(A_0-\lambda)^{-1}\gamma(\mu)
\end{equation}
for $\lambda,\mu\in\rho(A_0)$, which implies
\begin{equation}\label{gammalm}
  \overline{\gamma(\lambda)}
  = \bigl(I+(\lambda-\mu)(A_0-\lambda)^{-1}\bigr)\overline{\gamma(\mu)}
  = (A_0-\mu)(A_0-\lambda)^{-1}\overline{\gamma(\mu)}.
\end{equation}
With the help of the functional calculus of the self-adjoint operator $A_0$
one can conclude from~\eqref{gammalm} that
\begin{equation}\label{eq:gammaWunder}
  \big\| \ov{\gamma(\ov\lambda)} \big\|
  = \big\| \ov{\gamma(\lambda)} \big\| \qquad \text{for all}\;\; \lambda \in \rho(A_0).
\end{equation}
The values $M(\lambda)$ of the Weyl function satisfy
$M(\lambda)\subset M(\overline\lambda)^*$ and,
in particular, the operators $M(\lambda)$ are closable.
In general, the operators $M(\lambda)$ and their closures $\overline{M(\lambda)}$
are not bounded.
However, if $M(\lambda_0)$ is bounded for some $\lambda_0\in\rho(A_0)$,
then $M(\lambda)$ is bounded for all $\lambda\in\rho(A_0)$; see Lemma~\ref{le:mbdd} below.
The function $\lambda\mapsto M(\lambda)$ is holomorphic in the sense that
for any fixed $\mu \in \rho(A_0)$ it can be written as the sum of the possibly
unbounded operator $\Re M(\mu)$ and a $\cB(\cG)$-valued holomorphic function,
\[
  M(\lambda) = \Re M(\mu)
  +\gamma(\mu)^*\bigl[(\lambda-\Re\mu)
  +(\lambda-\mu)(\lambda-\overline\mu)(A_0-\lambda)^{-1}
  \bigr] \overline{\gamma(\mu)}
\]
for all $\lambda \in\rho(A_0)$.  In particular, $\Im M(\lambda)$ is a bounded operator
for each $\lambda\in\rho(A_0)$.

Further, for every $x\in\ran\Gamma_0$ we have
\[
  \frac{\drm^n}{\drm\lambda^n}\bigl( M(\lambda)x\bigr)
  = \frac{\drm^n}{\drm\lambda^n}\bigl( \gamma(\mu)^*\bigl[(\lambda-\Re\mu)
  +(\lambda-\mu)(\lambda-\overline\mu)(A_0-\lambda)^{-1}\bigr]\ov{\gamma(\mu)}x\bigr)
\]
for all $\lambda\in\rho(A_0)$ and all $n\in\NN$, and hence the $n$th strong
derivative $M^{(n)}(\lambda)$ (viewed as an operator defined on $\ran \Gamma_0$)
admits a continuous extension $\ov{M^{(n)}(\lambda)} \in \cB(\cG)$.
It satisfies
\begin{equation}\label{der_M}
  \ov{M^{(n)}(\lambda)} = n! \, \gamma(\ov\lambda)^*(A_0-\lambda)^{-(n-1)}\ov{\gamma(\lambda)},
  \qquad \lambda\in\rho(A_0),\;n\in\NN;
\end{equation}
see~\cite[Lemma~2.4\,(iii)]{BLL13trace}.

The Weyl function also satisfies (see \cite[Proposition~2.6\,(v)]{BL07})
\begin{equation}\label{diffM}
  M(\lambda)-M(\mu) = (\lambda-\mu)\gamma(\ov\mu)^*\gamma(\lambda),
\end{equation}
and with $\mu=\ov\lambda$ and the
relation $M(\ov\lambda)\subset M(\lambda)^*$ it follows that
\begin{equation}\label{gammmm}
  \Im M(\lambda) = (\Im\lambda)\gamma(\lambda)^*\gamma(\lambda)
  \qquad\text{and}\qquad
  \overline{\Im M(\lambda)} = (\Im\lambda)\gamma(\lambda)^*\overline{\gamma(\lambda)}.
\end{equation}
In the case when the values of $M$ are bounded operators we provide
a simple bound for the norms $\|\overline{M(\lambda)}\|$
in the next lemma.

\begin{lemma}\label{le:mbdd}
Let $\{\cG,\Gamma_0,\Gamma_1\}$ be a quasi boundary triple for $T\subset S^*$
with corresponding Weyl function $M$.  Assume that $M(\lambda)$ is bounded
for one $\lambda\in\rho(A_0)$.
Then $M(\lambda)$ is bounded for all $\lambda\in\rho(A_0)$, and the estimate
\begin{equation}\label{estMlaMmu}
  \bigl\|\ov{M(\lambda)}\bigr\|
  \le \biggl(1+\frac{|\lambda-\mu|}{|\Im\mu|}
  +\frac{|\lambda-\mu| |\lambda - \ov \mu|}{|\Im\lambda|\cdot|\Im\mu|}\biggr)\bigl\|\ov{M(\mu)}\bigr\|
\end{equation}
holds for all $\lambda,\mu\in\CC\setminus\RR$.
\end{lemma}

\begin{proof}
It follows from \eqref{diffM}, the relation $\ov{\gamma(\lambda)}\in\cB(\cG,\cH)$
and \eqref{gammastar}
that $M(\lambda)$ is bounded for all $\lambda\in\rho(A_0)$ if it is bounded for one $\lambda\in\rho(A_0)$.
Moreover, from the second identity in \eqref{gammmm} we conclude that
\begin{equation}\label{normgamma}
  \begin{aligned}
    \bigl\|\ov{\gamma(\ov \mu)}\bigr\|
    &= \bigl\|\gamma(\ov \mu)^*\ov{\gamma(\ov \mu)}\bigr\|^{1/2}
      \\[1ex]
    &= \frac{\bigl\|\ov{\Im M(\ov \mu)}\bigr\|^{1/2}}{|\Im \ov \mu|^{1/2}}
    = \frac{\bigl\|\ov{\Im M(\mu)}\bigr\|^{1/2}}{|\Im \mu|^{1/2}}\,,
  \hspace*{10ex} \mu\in\CC\setminus\RR,
  \end{aligned}
\end{equation}
where we have used that $\overline{M(\ov\mu)} = M(\mu)^*$.
If we replace $\gamma(\lambda)$ on the right-hand side of \eqref{diffM}
with the right-hand side of \eqref{gammalamu}, we obtain the representation
\begin{equation}\label{MlaMmu}
  M(\lambda) = M(\mu) + (\lambda-\mu)\gamma(\ov\mu)^*\gamma(\ov\mu)
  + (\lambda-\mu)(\lambda-\ov\mu) \gamma(\ov\mu)^*(A_0-\lambda)^{-1}\gamma(\ov\mu).
\end{equation}
By combining~\eqref{normgamma} and~\eqref{MlaMmu}, for $\lambda,\mu\in\CC\setminus\RR$
we obtain the estimate
\begin{align*}
  \bigl\|\ov{M(\lambda)}\bigr\|
  &\le \bigl\|\ov{M(\mu)}\bigr\| + \Bigl( |\lambda-\mu|
  + |\lambda-\mu| |\lambda - \ov \mu| \bigl\|(A_0-\lambda)^{-1}\bigr\| \Bigr)
  \bigl\|\ov{\gamma(\ov \mu)}\bigr\|^2
    \\[1ex]
  &\le \bigl\|\ov{M(\mu)}\bigr\|
  + \biggl(|\lambda-\mu|+\frac{|\lambda-\mu| |\lambda - \ov \mu|}{|\Im\lambda|}\biggr)
  \frac{\bigl\|\ov{\Im M(\mu)}\bigr\|}{|\Im\mu|}
    \\[1ex]
  &\le \biggl(1+\frac{|\lambda-\mu|}{|\Im\mu|}
  +\frac{|\lambda-\mu| |\lambda - \ov \mu|}{|\Im\lambda|\cdot|\Im\mu|}\biggr)\bigl\|\ov{M(\mu)}\bigr\|.
  \qedhere
\end{align*}
\end{proof}

\medskip

Decay properties of the Weyl function play an important role in this paper.
The next lemma shows that a decay of the Weyl function along a non-real ray
implies a uniform decay in certain sectors.

\begin{lemma}\label{le:raysector}
Let $\{\cG,\Gamma_0,\Gamma_1\}$ be a quasi boundary triple for $T\subset S^*$
with corresponding Weyl function $M$.  Assume that $M(\lambda)$ is bounded
for one (and hence for all) $\lambda\in\rho(A_0)$
and fix $\varphi\in(-\pi,0)\cup(0,\pi)$.
Then for every interval $[\psi_1,\psi_2]\subset(-\pi,0)$
or $[\psi_1,\psi_2]\subset(0,\pi)$ one has
\begin{equation}\label{Mdecaysect}
  \bigl\|\ov{M(re^{i\psi})}\bigr\| = \rmO\Bigl(\bigl\|\ov{M(re^{i\varphi})}\bigr\|\Bigr)
  \qquad\text{as}\;\; r\to\infty \;\; \text{uniformly in}\;\; \psi\in[\psi_1,\psi_2].
\end{equation}
In particular, if $\|\ov{M(re^{i\varphi})}\| \to 0$ as $r\to\infty$,
then $\|\ov{M(re^{i\psi})}\| \to 0$ as $r\to\infty$ uniformly
in $\psi\in[\psi_1,\psi_2]$.
\end{lemma}

\begin{proof}
Let $\mu=re^{i\varphi}$ and $\lambda=re^{i\psi}$ with $\psi\in[\psi_1,\psi_2]$ and $r>0$.
Then
\[
  |\lambda-\mu| = r\Bigl|e^{i\frac{\psi+\varphi}{2}}
  \Bigl(e^{i\frac{\psi-\varphi}{2}}-e^{-i\frac{\psi-\varphi}{2}}\Bigr)\Bigr|
  = 2r\Bigl|\sin \Big( \frac{\psi-\varphi}{2} \Big) \Bigr|
\]
and
\[
  |\lambda-\ov\mu|
  = 2r\Bigl|\sin \Big( \frac{\psi+\varphi}{2} \Big) \Bigr|.
\]
Now \eqref{estMlaMmu} yields
\begin{equation}\label{boundMsector}
  \bigl\|\ov{M(re^{i\psi})}\bigr\|
  \le \biggl(1+\frac{2|\sin\frac{\psi-\varphi}{2}|}{|\sin\varphi|}
  +\frac{4|\sin\frac{\psi-\varphi}{2}| |\sin\frac{\psi+\varphi}{2}|}{|\sin\psi|\cdot|\sin\varphi|}\biggr)
  \bigl\|\ov{M(re^{i\varphi})}\bigr\|,
\end{equation}
which shows \eqref{Mdecaysect}
since the expression in the brackets on the right-hand side of \eqref{boundMsector}
is uniformly bounded in $\psi\in[\psi_1,\psi_2]$.
\end{proof}

In the context of the previous lemma we remark that
$\lambda \mapsto \|\overline{M(\lambda)}\|$ decays at
most as $|\lambda|^{-1}$
since $\lambda\mapsto-(M(\lambda)x,x)^{-1}$ grows at most linearly
as it is a Nevanlinna function for every $x\in\ran\Gamma_0$.
We also recall from \cite[Lemma~2.3]{BLLR17} that for $x\in\ran\Gamma_0\setminus\{0\}$
the function
\[
  \lambda \mapsto \bigl(M(\lambda)x,x\bigr)
\]
is strictly increasing on each interval in $\rho(A_0)\cap\RR$; moreover, if
$A_0$ is bounded from below and
\[
  \bigl(M(\lambda)x,x\bigr) \to 0 \qquad\text{as}\;\;\lambda\to-\infty
\]
for all $x\in\ran\Gamma_0$, then
\begin{equation}\label{mmm}
  \bigl(M(\lambda)x,x\bigr) > 0, \qquad x\in\ran\Gamma_0\setminus\{0\},\;
  \lambda<\min\sigma(A_0).
\end{equation}

In the next proposition the case when the self-adjoint
operator $A_0=T\upharpoonright\ker\Gamma_0$ is bounded from below
and $\|\ov{M(\lambda)}\|\to0$ as $\lambda\to-\infty$ is considered.
Here the extension
\begin{equation}\label{defA1}
  A_1 \defeq T\upharpoonright \ker\Gamma_1
\end{equation}
is investigated.  Observe that the abstract Green identity \eqref{green2} yields
that $A_1$ is symmetric in $\cH$, but in the setting of quasi boundary triples
or generalized boundary triples $A_1$ is not necessarily self-adjoint
(in contrast to the case of ordinary boundary triples).

\begin{proposition}\label{pr:A1bddbelow}
Let $\{\cG,\Gamma_0,\Gamma_1\}$ be a quasi boundary triple for $T\subset S^*$
with corresponding Weyl function $M$ and suppose that $A_1=T\upharpoonright \ker\Gamma_1$
is self-adjoint and that $A_0$ and $A_1$ are bounded from below.
Further, assume that $M(\lambda)$ is bounded for one
(and hence for all) $\lambda\in\rho(A_0)$ and
that $\|\ov{M(\lambda)}\|\to0$ as $\lambda\to-\infty$.
Then
\begin{equation}\label{min10}
  \min\sigma(A_0) \le \min\sigma(A_1).
\end{equation}
\end{proposition}

\begin{proof}
The assumption $\|\ov{M(\lambda)}\|\to0$ as $\lambda\to-\infty$ implies
that \eqref{mmm} holds for all $x\in\ran\Gamma_0\setminus\{0\}$.
Fix $\lambda\in\RR$ such that $\lambda<\min\sigma(A_0)$ and $\lambda<\min\sigma(A_1)$.
It follows from \cite[Theorem~3.8]{BLL13IEOT} and \eqref{mmm} that
\begin{align*}
  \bigl((A_1-\lambda)^{-1}f,f\bigr) &= \bigl((A_0-\lambda)^{-1}f,f\bigr)
  - \bigl(M(\lambda)^{-1}\gamma(\lambda)^*f,\gamma(\lambda)^*f\bigr)
  \\[0.5ex]
  &= \bigl((A_0-\lambda)^{-1}f,f\bigr)
  - \bigl(M(\lambda)M(\lambda)^{-1}\gamma(\lambda)^*f,M(\lambda)^{-1}\gamma(\lambda)^*f\bigr)
  \\[0.5ex]
  &\le \bigl((A_0-\lambda)^{-1}f,f\bigr)
\end{align*}
for $f\in\cH$.  Since $(A_1-\lambda)^{-1}$ and $(A_0-\lambda)^{-1}$ are
bounded non-negative operators, we conclude that
\[
  \max\sigma\bigl((A_1-\lambda)^{-1}\bigr)\le\max\sigma\bigl((A_0-\lambda)^{-1}\bigr)
\]
and hence
\begin{align*}
 \min\sigma(A_0-\lambda)\le\min\sigma(A_1-\lambda),
\end{align*}
which is equivalent to~\eqref{min10}.
\end{proof}

% *******************************************************************
\section{Sectorial extensions of symmetric operators}\label{sec:sectorial}
% *******************************************************************

\noindent
Let $S$ be a densely defined, closed, symmetric operator in a Hilbert space $\cH$
and let $\{\cG,\Gamma_0,\Gamma_1\}$ be a quasi boundary triple for $T\subset S^*$.
For a linear operator $B$ in $\cG$ we define the operator $\AB$ in $\cH$ by
\begin{equation}\label{AB}
  \AB f = Tf, \qquad
  \dom \AB = \bigl\{f\in\dom T:\Gamma_0 f= B\Gamma_1 f\bigr\},
\end{equation}
where the boundary condition $\Gamma_0 f=B\Gamma_1 f$ is understood
in the sense that $\Gamma_1f\in\dom B$ and $\Gamma_0f=B\Gamma_1f$ holds.
Clearly, $\AB$ is a restriction of $T$ and hence of $S^*$.
Moreover, $\AB$ is an extension of $S$
since $S=T\upharpoonright(\ker\Gamma_0\cap\ker\Gamma_1)$
by \cite[Proposition~2.2]{BL07}.
Recall that in the special case of an ordinary boundary triple there is a
one-to-one correspondence between closed linear relations $B$ in $\cG$
and closed extensions $\AB$ of $S$ that are restrictions of $S^*$ via \eqref{AB};
for proper relations $B$ the definition of $\AB$ has to be interpreted accordingly.
For generalized and quasi boundary triples one has to impose
additional assumptions on $B$ to guarantee that $\AB$ is closed.
In this and the following sections we study the operators $\AB$ thoroughly;
in particular, we are interested in their spectral properties.

In the next theorem it is shown that under additional assumptions on $B$
and the Weyl function $M$ that corresponds to $\{\cG,\Gamma_0,\Gamma_1\}$
the operator $\AB$ is sectorial.
Recall first that the \emph{numerical range}, $W(A)$, of a linear operator $A$
is defined as
\[
  W(A) \defeq \bigl\{(Af,f): f\in\dom A,\;\|f\|=1\bigr\},
\]
and that $A$ is called \emph{sectorial} if $W(A)$ is contained
in a sector of the form
\begin{equation}\label{sector}
  \big\{z \in \C : \Re z \ge \eta, |\Im z| \le \kappa(\Re z-\eta) \big\}
\end{equation}
for some $\eta \in \RR$ and $\kappa > 0$. An operator $A$ is called \emph{m-sectorial}
if $W(A)$ is contained in a sector~\eqref{sector} and the complement
of \eqref{sector} has a non-trivial intersection with~$\rho(A)$.
In this case the spectrum of $A$ is contained in the closure of $W(A)$;
see, e.g.\ \cite[Propositions~2.8 and 3.19]{S12}.
Note that if $A$ is m-sectorial, then $-A$ generates an analytic semigroup;
see, e.g.\ \cite[Theorem~IX.1.24]{K95}.

\begin{theorem}\label{thmsec}
Let $\{\cG,\Gamma_0,\Gamma_1\}$ be a quasi boundary triple for $T\subset S^*$
with corresponding Weyl function $M$ such
that $A_1=T\upharpoonright \ker\Gamma_1$ is self-adjoint and bounded
from below and  $\rho(A_0)\cap(-\infty,\min\sigma(A_1))\not=\emptyset$.
Moreover, suppose that $M(\lambda)$ is bounded for one
(and hence for all) $\lambda\in\rho(A_0)$ and that
\begin{equation}\label{Meta01}
  M(\eta)\geq 0\quad \text{for some}\,\,\eta<\min\sigma(A_1),\,\,\eta\in\rho(A_0).
\end{equation}
Let $B$ be a closable operator in $\cG$ and assume that there exists $b\in\dR$ such that
\begin{myenum}
\item $\Re(Bx,x) \leq b\|x\|^2 $ \, for all $x \in \dom B$;
\item $ b\bigl\|\ov{M(\eta)}\bigr\| < 1$;
\item
  $\ran\ov{M(\eta)}^{1/2}\subset\dom B$.
\end{myenum}
Then the operator $\AB$ is sectorial and the numerical range $W(\AB)$ is
contained in the sector
\begin{equation}
  \cS_{\eta}(B) \defeq\Bigl\{z\in\CC: \Re z \ge \eta,\,
  |\Im z| \le \kappa_B(\eta)\bigl(\Re z-\eta\bigr)\Bigr\},
\end{equation}
where
\begin{equation}\label{kappaeta}
  \kappa_B(\eta) \defeq
  \frac{\bigl\|\Im\bigl(\ov{M(\eta)}^{1/2}B\ov{M(\eta)}^{1/2}\bigr)\bigr\|}{1-b\bigl\|\ov{M(\eta)}\bigr\|}\,.
\end{equation}
In particular,  if $\rho(\AB)\cap(\dC\setminus \cS_{\eta}(B))\not=\emptyset$,
then the operator $\AB$ is m-sectorial and $\sigma(\AB)$ is contained
in the sector $\cS_{\eta}(B)$.
\end{theorem}

\begin{proof}
Let $\eta<\min\sigma(A_1)$ be such that $\eta\in\rho(A_0)$ and $M(\eta)\ge0$,
which exists by \eqref{Meta01}.
Moreover, let $f\in\dom \AB$ with $\|f\|=1$.
Based on the decomposition
\[
  \dom T = \dom A_1  \dotplus \ker(T - \eta)
  = \ker \Gamma_1 \dotplus \ker(T - \eta)
\]
we can write $f$ in the form $f = f_1 + f_\eta$ with $f_1\in\ker\Gamma_1 = \dom A_1$
and $f_\eta\in\ker(T-\eta)$.  This yields
\begin{equation}\label{formula1}
\begin{split}
  (\AB f,f) &= \bigl(T(f_1+f_\eta),f_1+f_\eta\bigr) \\[0.5ex]
  &= (A_1f_1,f_1)
  + (Tf_1,f_\eta) + (Tf_\eta,f_\eta) + (Tf_\eta,f_1) \\[0.5ex]
  &= (A_1f_1,f_1) + (T f_1,f_\eta) + \eta\big[\|f_\eta\|^2 + (f_\eta,f_1)\big].
\end{split}
\end{equation}
Making use of the abstract Green identity~\eqref{green} we obtain
\begin{equation}\label{formula2}
\begin{split}
  (Tf_1,f_\eta) &=  (f_1,Tf_\eta)
  + (\Gamma_1f_1,\Gamma_0f_\eta)-(\Gamma_0f_1,\Gamma_1f_\eta) \\[0.5ex]
  & = \eta(f_1,f_\eta) -  (\Gamma_0f_1,\Gamma_1f_\eta).
\end{split}
\end{equation}
Moreover, since $f\in\dom \AB$ and $f_1\in\ker\Gamma_1$,
we have $\Gamma_1f_\eta\in\dom B$ and
\begin{equation}\label{formula3}
  \Gamma_0f_1 = B \Gamma_1f
  - \Gamma_0f_\eta = B\Gamma_1f_\eta - \Gamma_0f_\eta.
\end{equation}
Combining \eqref{formula2} and \eqref{formula3} we can rewrite
the right-hand side of \eqref{formula1}
in the form
\begin{align*}
  (\AB f,f) &= (A_1f_1,f_1) + \eta(f_1,f_\eta)
  - \bigl(B\Gamma_1f_\eta-\Gamma_0f_\eta,\Gamma_1f_\eta\bigr)
  + \eta\bigl[\|f_\eta\|^2+(f_\eta,f_1)\bigr]
  \\[0.5ex]
  &= (A_1f_1,f_1)+\eta\bigl[\|f_\eta\|^2+2\Re(f_\eta,f_1)\bigr]
  - \bigl(B\Gamma_1f_\eta-\Gamma_0f_\eta,\Gamma_1f_\eta\bigr).
\end{align*}
Next we use
\[
  \|f_\eta\|^2 + 2\Re(f_\eta,f_1) = \|f\|^2 - \|f_1\|^2 = 1-\|f_1\|^2
\]
and the definition of $M(\eta)$ to obtain
\begin{equation}\label{scalar}
\begin{split}
  (\AB f,f) &= (A_1f_1,f_1) + \eta - \eta\|f_1\|^2
    \\[0.5ex]
  &\quad  - \bigl(B M(\eta)\Gamma_0f_\eta,M(\eta)\Gamma_0f_\eta\bigr) +
  \bigl(\Gamma_0f_\eta, M(\eta)\Gamma_0f_\eta\bigr)
    \\[0.5ex]
  &= \bigl((A_1-\eta)f_1,f_1\bigr) + \eta
    \\[0.5ex]
  &\quad  - \bigl(B\overline{M(\eta)}\Gamma_0f_\eta,\overline{M(\eta)}\Gamma_0f_\eta\bigr)
  + \bigl\|\overline{M(\eta)}^{1/2}\Gamma_0f_\eta\bigr\|^2;
\end{split}
\end{equation}
recall that $\ov{M(\eta)}$ is a bounded, self-adjoint, non-negative operator.
Using assumption (i) we obtain
\begin{equation}\label{eq:Realteil}
\begin{split}
  \Re\bigl(B\ov{M(\eta)}\Gamma_0f_\eta,\ov{M(\eta)}\Gamma_0f_\eta\bigr)
  &\le b\bigl\|\ov{M(\eta)}\Gamma_0f_\eta\bigr\|^2
    \\[0.5ex]
  &\le b\bigl\|\ov{M(\eta)}^{1/2}\bigr\|^2\bigl\|\ov{M(\eta)}^{1/2}\Gamma_0f_\eta\bigr\|^2
    \\[0.5ex]
  &= b\bigl\|\ov{M(\eta)}\bigr\|\,\bigl\|\ov{M(\eta)}^{1/2}\Gamma_0f_\eta\bigr\|^2.
\end{split}
\end{equation}
From this, \eqref{scalar} and the fact that $\eta<\min\sigma(A_1)$ we conclude that
\begin{equation}\label{okgut}
\begin{split}
  \Re(\AB f,f) &\ge \eta - \Re\bigl(B\ov{M(\eta)}\Gamma_0f_\eta,
  \ov{M(\eta)}\Gamma_0f_\eta\bigr) + \bigl\|\ov{M(\eta)}^{1/2}\Gamma_0f_\eta\bigr\|^2 \\[0.5ex]
  &\ge \eta + \bigl(1-b\bigl\|\ov{M(\eta)}\bigr\|\bigr)
  \bigl\|\ov{M(\eta)}^{1/2}\Gamma_0f_\eta\bigr\|^2.
\end{split}
\end{equation}
This, together with assumption (ii), implies that
\begin{equation}\label{Regeeta}
  \Re(\AB f,f)\geq \eta.
\end{equation}
Moreover, it follows with assumption (iii) that the operator $B\ov{M(\eta)}^{1/2}$
is everywhere defined and closable since $B$ is closable.  Hence
\begin{equation}\label{bbmclos}
  B\ov{M(\eta)}^{1/2}\in\cB(\cG) \qquad\text{and}\qquad
  \ov{M(\eta)}^{1/2} B \ov{M(\eta)}^{1/2}\in\cB(\cG).
\end{equation}
With \eqref{scalar} we obtain that
\begin{align*}
  \big|\Im(\AB f,f)\big|
  &= \big|\Im\bigl(B\ov{M(\eta)}\Gamma_0f_\eta,\ov{M(\eta)}\Gamma_0f_\eta\bigr)\big|
    \\[0.5ex]
  &= \Big|\Im\Bigl(\ov{M(\eta)}^{1/2}B\ov{M(\eta)}^{1/2}\ov{M(\eta)}^{1/2}\Gamma_0f_\eta,
  \ov{M(\eta)}^{1/2}\Gamma_0f_\eta\Bigr)\Big|
    \\[0.5ex]
  &= \Big|\Bigl(\Im\Bigl(\ov{M(\eta)}^{1/2} B \ov{M(\eta)}^{1/2}\Bigr)\ov{M(\eta)}^{1/2}\Gamma_0 f_\eta,
  \ov{M(\eta)}^{1/2}\Gamma_0 f_\eta\Bigr)\Big|
    \\[0.5ex]
  &\le \big\|\Im\bigl(\ov{M(\eta)}^{1/2}B\ov{M(\eta)}^{1/2}\bigr)\big\|\,
  \big\|\ov{M(\eta)}^{1/2}\Gamma_0f_\eta\big\|^2.
\end{align*}
This, together with \eqref{okgut}, implies that
\begin{equation}\label{okgut2}
  \big|\Im(\AB f,f)\big|
  \le \frac{\big\|\Im\bigl(\ov{M(\eta)}^{1/2}B\ov{M(\eta)}^{1/2}\bigr)\big\|}{1-b\big\|\ov{M(\eta)}\big\|}
  \Bigl(\Re(\AB f,f)-\eta\Bigr).
\end{equation}
The inequalities \eqref{Regeeta} and \eqref{okgut2} show that the
numerical range of $\AB$ is contained in the
sector $\cS_{\eta}(B)$, and hence the operator $\AB$ is sectorial.
The last statement of the theorem is well known;
see, e.g.\ \cite[Proposition~3.19]{S12}.
\end{proof}

\begin{remark}
In Theorem~\ref{thmsec} it is not assumed explicitly that the
self-adjoint extension $A_0=T\upharpoonright\ker\Gamma_0$ is bounded from below.
However, the operator $B=0$ satisfies assumptions (i)--(iii)
in Theorem~\ref{thmsec} with $b=0$, which yields $\kappa_B(\eta)=0$.
Thus the spectrum of the operator $A_0=A_{[0]}$ is contained in $[\eta,\infty)$
and therefore $A_0$ is bounded from below by $\eta$.
\end{remark}

Theorem~\ref{thmsec} provides explicit sufficient conditions for
the extension $\AB$ in \eqref{AB} to be sectorial.
However, in applications it is essential to ensure that $\AB$ is m-sectorial,
i.e.\ to guarantee that $\rho(\AB)\cap (\dC\setminus\cS_\eta(B)) \ne \emptyset$.
We consider one particular situation in the next proposition, but deal in
more detail with this question in the next section.

In the next proposition we specialize Theorem~\ref{thmsec} to the situation
of an ordinary boundary triple, where we can actually prove that the operator
$\AB$ is m-sectorial; to the best of our knowledge the assertion is new.
We remark that in the following proposition it is possible
to choose $b = \max\sigma(\Re B)$.

\begin{proposition}\label{corsec}
Let $\{\cG,\Gamma_0,\Gamma_1\}$ be an ordinary boundary triple for $S^*$
with corresponding Weyl function $M$ and assume that $A_1$ is bounded from below
and that $\rho(A_0)\cap(-\infty,\min\sigma(A_1)) \ne \emptyset$.
Moreover, assume that
\[
  M(\eta) \ge 0 \qquad\text{for some}\;\; \eta<\min\sigma(A_1),\; \eta\in\rho(A_0).
\]
Let $B \in \cB(\cG)$, let $b \in \R$ be such that $\Re(Bx,x) \le b\|x\|^2$
for all $x\in\cG$, and assume that $b\|M(\eta)\|<1$.
Then the operator $\AB$ is m-sectorial and we have
\begin{equation}\label{obtsect}
  \sigma(\AB) \subset \ov{W(\AB)} \subset \Bigl\{z\in\CC: \Re z \ge \eta,\,
  |\Im z| \le \kappa_B(\eta)\bigl(\Re z-\eta\bigr)\Bigr\},
\end{equation}
where
\[
  \kappa_B(\eta) \defeq
  \frac{\bigl\|\Im\bigl(M(\eta)^{1/2}B M(\eta)^{1/2}\bigr)\bigr\|}{1-b\bigl\|M(\eta)\bigr\|}\,.
\]
\end{proposition}

\begin{proof}
The fact that $\AB$ is sectorial and the second inclusion in \eqref{obtsect}
follow directly from Theorem~\ref{thmsec}.
To prove that $\AB$ is m-sectorial we show that $\eta\in\rho(\AB)$.
Without loss of generality we can assume that $b\ge0$.
Observe that $M(\eta)^{1/2}$ is well defined since $M(\eta)\ge0$ by assumption.
For $x \in \cG$ with $\|x\| = 1$ we have
\begin{align*}
  \Re \bigl(M(\eta)^{1/2}BM(\eta)^{1/2} x, x\bigr)
  &= \Re \bigl(BM(\eta)^{1/2}x,M(\eta)^{1/2}x\bigr)
  \\[0.5ex]
  & \le b \bigl\|M(\eta)^{1/2} x\bigr\|^2
  = b \bigl(M(\eta)x,x\bigr)
  \le b\bigl\|M(\eta)\bigr\|,
\end{align*}
which implies that
\begin{align*}
  \sigma\bigl(M(\eta)^{1/2}BM(\eta)^{1/2}\bigr)
  &\subset \ov{W\bigl(M(\eta)^{1/2}BM(\eta)^{1/2}\bigr)} \\[0.5ex]
  &\subset \bigl\{z\in\CC:\Re z\le b\|M(\eta)\|\bigr\}.
\end{align*}
Since $b\|M(\eta)\|<1$, this yields
\[
  1\in\rho\bigl(M(\eta)^{1/2}BM(\eta)^{1/2}\bigr)
\]
and hence $1\in\rho(BM(\eta))$.
Now \cite[Proposition~1.6]{DM95} implies that $\eta\in\rho(\AB)$,
and therefore $\AB$ is m-sectorial, which also proves the
first inclusion in \eqref{obtsect}.
\end{proof}

% *******************************************************************
\section{Sufficient conditions for closed extensions with non-empty resolvent set}
\label{sec:closedoperators}
% *******************************************************************

\noindent
Let $S$ be a densely defined, closed, symmetric operator in a Hilbert space $\cH$
and let $\{\cG,\Gamma_0,\Gamma_1\}$ be a quasi boundary triple for $T\subset S^*$.
In this section we provide some abstract sufficient conditions on
the (boundary) operator $B$ in $\cG$ such that the operator $\AB$
defined in \eqref{AB} is closed and has a non-empty resolvent set.

\begin{theorem}\label{mainthm}
Let $\{\cG,\Gamma_0,\Gamma_1\}$ be a quasi boundary triple for $T\subset S^*$
with corresponding $\gamma$-field $\gamma$ and Weyl function $M$.
Let $B$ be a closable operator in $\cG$ and assume that there
exists $\lambda_0\in\rho(A_0)$ such that the following conditions are satisfied:
\begin{myenum}
  \item $1\in\rho\bigl(B\ov{M(\lambda_0)}\bigr)$;
  \item $B\bigl(\ran\ov{M(\lambda_0)}\cap\dom B\bigr)\subset\ran\Gamma_0$;
  \item $\ran\Gamma_1\subset\dom B$;
  \item $B(\ran\Gamma_1)\subset\ran\Gamma_0$ \, or \,  $\lambda_0\in\rho(A_1)$.
\end{myenum}
Then the operator
\begin{equation}\label{ab}
  \AB f=Tf, \qquad
  \dom \AB=\bigl\{f\in\dom T:\Gamma_0 f= B\Gamma_1 f\bigr\},
\end{equation}
is a closed extension of $S$ in $\cH$ such that $\lambda_0\in\rho(\AB)$, and
\begin{equation}\label{resformula}
  (\AB-\lambda)^{-1}=(A_0-\lambda)^{-1}
  +\gamma(\lambda)\bigl(I-BM(\lambda)\bigr)^{-1}B\gamma(\overline\lambda)^*
\end{equation}
holds for all $\lambda\in\rho(\AB)\cap\rho(A_0)$.

Further, let $B'$ be a linear operator in $\cG$ that satisfies {\rm(i)--(iv)}
with $B$ replaced by $B'$ and $\lambda_0$ replaced by $\ov{\lambda_0}$,
and assume that
\begin{equation}\label{BBprime}
  (Bx,y) = (x,B'y) \qquad\text{for all}\;\;x\in\dom B,\;y\in\dom B'.
\end{equation}
Then $A_{[B']}$ is closed and
\begin{equation}\label{ABprimeABstar}
  A_{[B']} = \AB^*.
\end{equation}
In particular, $\ov{\lambda_0} \in \rho(A_{[B']})$.
\end{theorem}

\begin{remark}
In the special case when the operator $B$ in Theorem~\ref{mainthm} is symmetric
and the assumptions (i) and (ii) hold for some $\lambda_0\in\rho(A_0)
\cap\dR$ the result reduces to \cite[Theorem~2.6]{BLLR17},
where self-adjointness of $\AB$ was shown;
cf.\ also \cite[Theorem~2.4]{BLLR17}.
In this sense Theorem~\ref{mainthm} can be seen as a generalization
of the considerations in \cite[Section~2]{BLLR17}
to non-self-adjoint extensions.
\end{remark}

Before we prove Theorem~\ref{mainthm}, we formulate some corollaries.
If $\{\cG,\Gamma_0,\Gamma_1\}$ is a generalized boundary triple,
then $\ran\Gamma_0=\cG$ and $M(\lambda_0)\in\cB(\cG)$.
Hence in this case the above theorem reads as follows.

\begin{corollary}\label{gbtcor}
Let $\{\cG,\Gamma_0,\Gamma_1\}$ be a generalized boundary triple for $T\subset S^*$
with corresponding $\gamma$-field $\gamma$ and Weyl function $M$.
Let $B$ be a closable operator in $\cG$ and assume that there
exists $\lambda_0\in\rho(A_0)$ such that the following conditions are satisfied:
\begin{myenum}
 \item $1\in\rho(BM(\lambda_0))$;
 \item $\ran\Gamma_1\subset\dom B$.
\end{myenum}
Then the operator $\AB$ in \eqref{ab} is a closed extension of $S$
such that $\lambda_0\in\rho(\AB)$,
and the resolvent formula \eqref{resformula} holds for
all $\lambda\in\rho(\AB)\cap\rho(A_0)$.

Further, let $B'$ be a linear operator in $\cG$ that satisfies {\rm(i)} and {\rm(ii)}
with $B$ replaced by $B'$ and $\lambda_0$ replaced by $\ov{\lambda_0}$,
and assume that \eqref{BBprime} holds.
Then $A_{[B']}$ is closed and $A_{[B']} = \AB^*$.
In particular, $\ov{\lambda_0} \in \rho(A_{[B']})$.
\end{corollary}

In the special case when $\{\cG,\Gamma_0,\Gamma_1\}$ in Theorem~\ref{mainthm}
or Corollary~\ref{gbtcor} is an ordinary boundary triple
the condition  $\ran\Gamma_1\subset\dom B$ implies $\dom B=\cG$.
Since $B$ is assumed to be closable, it follows that $B$ is closed and
hence $B\in\cB(\cG)$.  In this case the statements in Theorem~\ref{mainthm}
and Corollary~\ref{gbtcor} are well known.

In the next corollary we return to the general situation of a
quasi boundary triple, but we assume that $B$ is bounded
and everywhere defined on $\cG$.

\begin{corollary}\label{maincor}
Let $\{\cG,\Gamma_0,\Gamma_1\}$ be a quasi boundary triple for $T\subset S^*$
with corresponding Weyl function $M$.
Let $B\in\cB(\cG)$ and assume that there
exists $\lambda_0\in\rho(A_0)$ such that the following conditions are satisfied:
\begin{myenum}
\item
$1\in\rho\bigl(B\overline{M(\lambda_0)}\bigr)$;
\item
$B\bigl(\ran\ov{M(\lambda_0)}\bigr)\subset\ran\Gamma_0$;
\item
$B(\ran\Gamma_1)\subset\ran\Gamma_0$ \, or \, $\lambda_0\in\rho(A_1)$.
\end{myenum}
Then the operator $\AB$ in \eqref{ab} is a closed extension of $S$
such that $\lambda_0\in\rho(\AB)$,
and the resolvent formula \eqref{resformula} holds for
all $\lambda\in\rho(\AB)\cap\rho(A_0)$.

Further, if conditions {\rm(i)--(iii)} are satisfied also for $B^*$
instead of $B$ and $\lambda_0$ replaced by $\ov{\lambda_0}$,
then $A_{[B^*]}=\AB^*$.  In particular, $\ov{\lambda_0} \in \rho(A_{[B^*]})$.
\end{corollary}

Note that if in Corollary~\ref{maincor} the triple $\{\cG,\Gamma_0,\Gamma_1\}$
is a generalized boundary triple, then assumptions (ii) and (iii) are
automatically satisfied.

In the next two corollaries a set of conditions is provided
which guarantee that condition~(i) in Theorem~\ref{mainthm} is satisfied;
here Corollary~\ref{corchenbbb} is a special case of Corollary~\ref{corchenaaa}
for bounded $B$.  In contrast to the previous results it is also assumed
that $M(\lambda)$ is bounded for one (and hence for all) $\lambda\in\rho(A_0)$
and that the set $\rho(A_0)\cap\dR$ is non-empty.

\begin{corollary}\label{corchenaaa}
Let $\{\cG,\Gamma_0,\Gamma_1\}$ be a quasi boundary triple for $T\subset S^*$
with corresponding Weyl function $M$, and assume that $M(\lambda)$ is
bounded for one (and hence for all) $\lambda\in\rho(A_0)$.
Let $B$ be a closable operator in $\cG$ and assume that there
exist $b\in\dR$ and $\lambda_0\in\rho(A_0)\cap\dR$
such that the following conditions are satisfied:
\begin{myenum}
  \item $\Re(Bx,x)\leq b\Vert x\Vert^2$ \, for all $x \in \dom B$;
  \item $M(\lambda_0)\geq 0$ and $b\bigl\|\ov{M(\lambda_0)}\bigr\|<1$;
  \item $\ran\ov{M(\lambda_0)}^{1/2}\subset\dom B$;
  \item $B\bigl(\ran\ov{M(\lambda_0)}\bigr)\subset\ran\Gamma_0$;
  \item $\ran\Gamma_1\subset\dom B$;
  \item $B(\ran\Gamma_1)\subset\ran\Gamma_0$ \, or \, $\lambda_0\in\rho(A_1)$.
\end{myenum}
Then the operator $\AB$ in \eqref{ab} is a closed extension of $S$
such that $\lambda_0\in\rho(\AB)$,
and the resolvent formula \eqref{resformula} holds for
all $\lambda\in\rho(\AB)\cap\rho(A_0)$.

Further,  let $B'$ be a linear operator in $\cG$ that satisfies {\rm(i)--(vi)}
with $B$ replaced by $B'$
and assume that \eqref{BBprime} holds.
Then $A_{[B']}$ is closed and $A_{[B']} = \AB^*$.
In particular, $\lambda_0 \in \rho(A_{[B']})$.
\end{corollary}

For $B\in\cB(\cG)$, Corollary~\ref{corchenaaa} reads as follows.

\begin{corollary}\label{corchenbbb}
Let $\{\cG,\Gamma_0,\Gamma_1\}$ be a quasi boundary triple for $T\subset S^*$
with corresponding Weyl function $M$, and assume that $M(\lambda)$ is
bounded for one (and hence for all) $\lambda\in\rho(A_0)$.
Let $B\in\cB(\cG)$ and $b\in\dR$ such that
\[
  \Re(Bx,x)\leq b\|x\|^2 \qquad\text{for all}\;\; x \in \cG
\]
and assume that for some $\lambda_0\in\rho(A_0)\cap\dR$ the
following conditions are satisfied:
\begin{myenum}
  \item $M(\lambda_0)\geq 0$ and $b\bigl\|\ov{M(\lambda_0)}\bigr\|<1$;
  \item $B\bigl(\ran\ov{M(\lambda_0)}\bigr)\subset\ran\Gamma_0$;
  \item $B(\ran\Gamma_1)\subset\ran\Gamma_0$ \, or \, $\lambda_0\in\rho(A_1)$.
\end{myenum}
Then the operator $\AB$ in \eqref{ab} is a closed extension of $S$
such that $\lambda_0\in\rho(\AB)$,
and the resolvent formula \eqref{resformula} holds for
all $\lambda\in\rho(\AB)\cap\rho(A_0)$.

Further, if conditions {\rm(i)--(iii)} are satisfied also for $B^*$
instead of $B$,
then $A_{[B^*]}=\AB^*$.  In particular, $\lambda_0 \in \rho(A_{[B^*]})$.
\end{corollary}

\begin{proof}[Proof of Corollary~\ref{corchenaaa}]
It suffices to show that assumptions (i)--(iii) in Corollary~\ref{corchenaaa}
imply assumption (i) in Theorem~\ref{mainthm}.
The assumption (ii) in Theorem~\ref{mainthm} is satisfied since the inclusion
\begin{equation*}
 \ran\ov{M(\lambda_0)}\subset\ran\ov{M(\lambda_0)}^{1/2}\subset\dom B
\end{equation*}
holds by (iii) in Corollary~\ref{corchenaaa}, and hence
(iv) in Corollary~\ref{corchenaaa} coincides with (ii) in Theorem~\ref{mainthm};
the assumptions (iii) and (iv) in Theorem~\ref{mainthm} coincide
with (v) and (vi) in Corollary~\ref{corchenaaa}.

In order to show (i) in Theorem~\ref{mainthm} we use a similar idea as in the
proof of Proposition~\ref{corsec}, but we have to be more careful with operator domains.
Note first that a negative $b$
in (i) and (ii) in Corollary~\ref{corchenaaa} can always be replaced by $0$;
hence without loss of generality we can assume that $b\ge0$.
For $\lambda_0 \in\rho(A_0)\cap\dR$ such that $M(\lambda_0)\ge 0$
we have $\ov{M(\lambda_0)}\geq 0$.
As in \eqref{bbmclos} in the proof of Theorem~\ref{thmsec} the operator
\begin{equation}\label{bmhalf}
  B \overline{M(\lambda_0)}^{1/2}
\end{equation}
is defined on all of $\cG$ by (iii) and is closable since $B$ is closable.  Hence
\begin{equation}\label{halfbmhalf}
  B \overline{M(\lambda_0)}^{1/2}\in\cB(\cG) \qquad\text{and}\qquad
  \overline{M(\lambda_0)}^{1/2}B\overline{M(\lambda_0)}^{1/2}\in\cB(\cG).
\end{equation}
Then for $x \in \cG$ with $\|x\| = 1$ we conclude from assumption (i) that
\begin{align*}
  \Re \bigl(\ov{M(\lambda_0)}^{1/2}B\ov{M(\lambda_0)}^{1/2} x, x\bigr)
  &= \Re \bigl(B\ov{M(\lambda_0)}^{1/2}x,\ov{M(\lambda_0)}^{1/2}x\bigr)
  \\[0.5ex]
  & \le b \bigl\|\ov{M(\lambda_0)}^{1/2} x\bigr\|^2
  = b \bigl(\ov{M(\lambda_0)}x,x\bigr)
  \le b\bigl\|\ov{M(\lambda_0)}\bigr\|.
\end{align*}
Thus
\begin{align*}
  \sigma\bigl(\ov{M(\lambda_0)}^{1/2}B\ov{M(\lambda_0)}^{1/2}\bigr)
  &\subset \ov{W\bigl(\ov{M(\lambda_0)}^{1/2}B\ov{M(\lambda_0)}^{1/2}\bigr)} \\[0.5ex]
  &\subset \bigl\{z\in\CC:\Re z\le b\bigl\|\ov{M(\lambda_0)}\bigr\|\bigr\},
\end{align*}
and hence assumption (ii) implies that
\[
  1\in\rho\bigl(\overline{M(\lambda_0)}^{1/2}B\overline{M(\lambda_0)}^{1/2}\bigr).
\]
This shows that also $1\in\rho(B\overline{M(\lambda_0)})$
and therefore (i) in Theorem~\ref{mainthm} holds.
\end{proof}

Now we finally turn to the proof of Theorem~\ref{mainthm}. We note that the arguments in
Steps~2, 4 and 5 are similar to those in the proof of~\cite[Theorem~2.4]{BLLR17},
where the case when $B$ is symmetric was treated.
For the convenience of the reader we provide a self-contained and complete proof.

\begin{proof}[Proof of Theorem~\ref{mainthm}]
The proof of Theorem~\ref{mainthm} consists of six separate steps.
During the first four steps of the proof we assume that the first condition
in (iv) is satisfied.  In Step~5 of the proof we show that the second condition
in (iv) and assumptions (ii) and (iii) imply the first condition in (iv).
Finally, in Step~6 we prove the statements about $A_{[B']}$.

\medskip

\textit{Step 1.}
We claim that $\ker(\AB-\lambda_0)=\{0\}$.
To this end, let $f\in\ker(\AB-\lambda_0)$.  Then $f$ satisfies the
equation $Tf = \lambda_0 f$ and the abstract boundary condition $\Gamma_0 f=B\Gamma_1 f$.
It follows that
\begin{equation*}
  \Gamma_0f = B\Gamma_1f = BM(\lambda_0)\Gamma_0f = B\ov{M(\lambda_0)}\Gamma_0f,
\end{equation*}
that is, $\Gamma_0f \in \ker(I-B\ov{M(\lambda_0)})$.  From this and assumption (i)
of the theorem it follows that $\Gamma_0f = 0$ and, thus, $f\in\ker(A_0-\lambda_0)$.
Since $\lambda_0\in\rho(A_0)$, we obtain that $f = 0$.
Therefore we have $\ker(\AB-\lambda_0)=\{0\}$.

\textit{Step 2.}
Next we show that
\begin{equation}\label{surj}
  \ran(\AB-\lambda_0)=\cH
\end{equation}
holds. In order to do so, we first verify the inclusion
\begin{equation}\label{woracek}
  \ran\bigl(B\gamma(\overline{\lambda_0})^*\bigr)\subset\ran\bigl(I-BM(\lambda_0)\bigr).
\end{equation}
Note that the product $B\gamma(\overline{\lambda_0})^*$ on
the left-hand side of \eqref{woracek}
is defined on all of $\cH$ since $\gamma(\overline{\lambda_0})^*=\Gamma_1(A_0-\lambda_0)^{-1}$
by \eqref{gammastar} and $\ran\Gamma_1\subset\dom B$ by condition (iii).
For the inclusion in \eqref{woracek} consider
$\psi=B\gamma(\overline{\lambda_0})^* f$ for some $f\in\cH$.
From \eqref{gammastar} and the first condition in (iv) we obtain
that $\psi\in\ran\Gamma_0$.  Making use of assumption (i) we see that
\begin{equation}\label{varphi}
  \varphi\defeq\bigl(I-B\overline{M(\lambda_0)}\bigr)^{-1}\psi\in\dom
  \bigl(B\overline{M(\lambda_0)}\bigr)
\end{equation}
is well defined.  Hence
\begin{equation*}
 \varphi=B\overline{M(\lambda_0)}\varphi+\psi,
\end{equation*}
and since $\overline{M(\lambda_0)}\varphi\in\ran\overline{M(\lambda_0)}\cap\dom B$,
it follows from (ii) and $\psi\in\ran\Gamma_0$ that
$\varphi\in\ran\Gamma_0=\dom M(\lambda_0)$.
Thus we conclude from \eqref{varphi} that
\begin{equation*}
 \bigl(I-BM(\lambda_0)\bigr)\varphi=\psi,
\end{equation*}
which shows the inclusion \eqref{woracek}.

To verify~\eqref{surj}, let $f\in\cH$ and consider
\begin{equation}\label{hhh}
  h \defeq (A_0-\lambda_0)^{-1}f+\gamma(\lambda_0)
  \bigl(I-BM(\lambda_0)\bigr)^{-1}B\gamma(\overline{\lambda_0})^*f.
\end{equation}
Observe that $h$ is well defined since
$\dom\gamma(\lambda_0)=\dom M(\lambda_0)\supset\ran(I-BM(\lambda_0))^{-1}$
and the product of $(I-BM(\lambda_0))^{-1}$
and $B\gamma(\overline{\lambda_0})^*$ makes sense by \eqref{woracek}.
It is clear that $h\in\dom T$.
Moreover, from $\dom A_0=\ker\Gamma_0$, the definitions of
the $\gamma$-field and Weyl function, and \eqref{gammastar} we conclude that
\begin{equation*}
  \Gamma_0h =\bigl(I-BM(\lambda_0)\bigr)^{-1}B\gamma(\overline{\lambda_0})^*f
\end{equation*}
and
\begin{equation*}
  \Gamma_1 h=\gamma(\overline{\lambda_0})^*f+M(\lambda_0)
  \bigl(I-BM(\lambda_0)\bigr)^{-1}B\gamma(\overline{\lambda_0})^*f.
\end{equation*}
Now it follows that
\begin{equation*}
  B\Gamma_1 h=\bigl(I-BM(\lambda_0)\bigr)^{-1}B\gamma(\overline{\lambda_0})^*f=\Gamma_0h,
\end{equation*}
and therefore $h\in\dom\AB$. From the definition of $h$ in \eqref{hhh}
and $\ran\gamma(\lambda_0)=\ker(T-\lambda_0)$ we obtain that
\begin{equation*}
  (\AB-\lambda_0)h=(T-\lambda_0)h=f.
\end{equation*}
Hence we have proved \eqref{surj}.
Moreover, since $h=(\AB-\lambda_0)^{-1}f$, we also conclude from \eqref{hhh} that
\begin{equation}\label{resformula0}
  (\AB-\lambda_0)^{-1}f=(A_0-\lambda_0)^{-1}f
  + \gamma(\lambda_0)\bigl(I-BM(\lambda_0)\bigr)^{-1}B\gamma(\overline{\lambda_0})^*f.
\end{equation}

\textit{Step 3.}
We verify that $\AB$ is closed and that $\lambda_0\in\rho(\AB)$.
Since $B$ is closable by assumption and $\gamma(\ov{\lambda_0})^* \in \cB(\cH,\cG)$,
it follows that $B\gamma(\ov{\lambda_0})^*$ is closable and hence closed, so that
\begin{equation}\label{bg}
  B\gamma(\overline{\lambda_0})^* \in \cB(\cH,\cG).
\end{equation}
The operators $\gamma(\lambda_0)$ and $(I-BM(\lambda_0))^{-1}$ in \eqref{resformula0} are
bounded by \eqref{gammastar} and assumption~(i), respectively.
Therefore \eqref{resformula0} shows that the operator $(\AB-\lambda_0)^{-1}$ is bounded.
Since $(\AB-\lambda_0)^{-1}$ is defined on $\cH$ by \eqref{surj}, it follows
that $\AB$ is closed and $\lambda_0\in\rho(\AB)$.

\textit{Step 4.}
Now we prove the resolvent formula \eqref{resformula} for
all $\lambda\in\rho(\AB)\cap\rho(A_0)$.
We first observe that $I-BM(\lambda)$ is injective for $\lambda\in\rho(\AB)\cap\rho(A_0)$.
In fact, let $\varphi\in\ker(I-BM(\lambda))$. Then $\varphi\in\dom M(\lambda)=\ran\Gamma_0$
and $f \defeq \gamma(\lambda)\varphi$ belongs to $\ker(T-\lambda)$.
Furthermore, $\Gamma_0 f=\varphi$, and from
\begin{equation*}
  B\Gamma_1 f = BM(\lambda)\Gamma_0 f
  = B M(\lambda)\varphi = \varphi = \Gamma_0 f
\end{equation*}
we conclude that $f\in\dom \AB$.
Since $f\in\ker(T-\lambda)$, this implies that $f\in\ker(\AB-\lambda)$, and hence $f=0$
as $\lambda\in\rho(\AB)$ by assumption.  It follows that $\varphi=\Gamma_0 f=0$,
and therefore $I-BM(\lambda)$ is injective.

Now let $f\in\cH$, $\lambda\in\rho(\AB)\cap\rho(A_0)$, and set
\begin{equation}\label{kkk}
  k \defeq (\AB-\lambda)^{-1}f-(A_0-\lambda)^{-1}f.
\end{equation}
With $g \defeq (\AB-\lambda)^{-1}f\in\dom\AB$ we have
$B\Gamma_1 g=\Gamma_0g=\Gamma_0k$.  Since $k\in\ker(T-\lambda)$, it is
also clear that $M(\lambda)\Gamma_0 k=\Gamma_1 k$.
Moreover, $\Gamma_1(g - k) = \gamma(\overline\lambda)^*f$ by \eqref{gammastar},
and therefore
\begin{equation*}
  \bigl(I-BM(\lambda)\bigr)\Gamma_0 k
  = \Gamma_0 g- BM(\lambda)\Gamma_0 k
  = B \Gamma_1 g - B \Gamma_1 k
  = B\gamma(\overline\lambda)^*f
\end{equation*}
yields $\Gamma_0 k=(I-BM(\lambda))^{-1}B\gamma(\overline\lambda)^*f$.
Since $k\in\ker(T-\lambda)$, we have
\[
  k = \gamma(\lambda)\Gamma_0k
  = \gamma(\lambda)\bigl(I-BM(\lambda)\bigr)^{-1}B\gamma(\overline\lambda)^*f,
\]
which, together with \eqref{kkk}, yields \eqref{resformula} for $\lambda\in\rho(\AB)\cap\rho(A_0)$.

\textit{Step 5.}
Now assume that $\lambda_0\in\rho(A_1)$, i.e.\ the second condition in (iv) holds.
We claim that in this situation $B(\ran\Gamma_1)\subset\ran\Gamma_0$ follows.
In fact, suppose that $g\in\ran\Gamma_1$.  Then $g\in\dom B$ by condition (iii).
Since $\ran\Gamma_1=\ran M(\lambda_0)\subset\ran(\ov{M(\lambda_0)})$
in the present situation by \cite[Proposition~2.6\,(iii)]{BL07},
we conclude from (ii) that $Bg\in\ran\Gamma_0$.

\textit{Step 6.}
Now let $B'$ be as in the last part of the statement of the theorem.
By assumption (iii) for $B$ and $B'$, both operators are densely defined.
Hence relation~\eqref{BBprime} implies that $B'$ is also closable.
It follows from Steps~1--5 that $A_{[B']}$ is closed and
that $\ov\lambda_0\in\rho(A_{[B']})$.
Let $f\in\dom \AB$ and $g\in\dom A_{[B']}$.
Then $\Gamma_1f\in\dom B$, $\Gamma_1g\in\dom B'$ and
\[
  \Gamma_0f = B\Gamma_1f \qquad\text{and}\qquad
  \Gamma_0g = B'\Gamma_1g.
\]
Hence Green's identity \eqref{green} and the relation \eqref{BBprime} yield
\begin{align*}
  (\AB f,g) - (f,A_{[B']}g)
  &= (Tf,g) - (f,Tg)
  = (\Gamma_1f,\Gamma_0g) - (\Gamma_0f,\Gamma_1g)
  \\[0.5ex]
  &= (\Gamma_1f,B'\Gamma_1g) - (B\Gamma_1f,\Gamma_1g) = 0,
\end{align*}
which implies that
\begin{equation}\label{ABprimesubABstar}
  A_{[B']} \subset \AB^*.
\end{equation}
Since $\lambda_0 \in \rho(A_{[B]})$, we have $\ov\lambda_0\in\rho(\AB^*)$.
This, together with $\ov\lambda_0\in\rho(A_{[B']})$ and \eqref{ABprimesubABstar},
proves the relation in \eqref{ABprimeABstar}.
\end{proof}

In the next proposition we consider Schatten--von Neumann properties of
certain resolvent differences (see the end of the introduction for the definition
of the classes $\sS_p$).
For the self-adjoint case parts of the results of the following proposition
can be found in~\cite[Theorem~3.17]{BLL13IEOT}.

\begin{proposition}\label{spprop}
Let $\{\cG,\Gamma_0,\Gamma_1\}$ be a quasi boundary triple for $T\subset S^*$
with corresponding $\gamma$-field $\gamma$ and Weyl function $M$.
Let $B$ be a closable operator in $\cG$ and assume that there
exists $\lambda_0\in\rho(A_0)$ such that conditions {\rm(i)--(iv)}
in Theorem~\ref{mainthm} are satisfied.
Moreover, assume that
\begin{equation}\label{sp1}
  \gamma(\lambda_1)^* \in \sS_p(\cH,\cG)
\end{equation}
for some $\lambda_1\in\rho(A_0)$ and some $p > 0$.  Then
\begin{equation}\label{sp2}
  (\AB-\lambda)^{-1}-(A_0-\lambda)^{-1} \in \sS_p(\cH)
\end{equation}
for all $\lambda\in\rho(\AB)\cap\rho(A_0)$.
If, in addition, $A_1$ is self-adjoint, then
\begin{equation}\label{sp2mitA1}
  (\AB - \lambda)^{-1}-(A_1 - \lambda)^{-1} \in \sS_p(\cH)
\end{equation}
for all $\lambda\in\rho(\AB)\cap\rho(A_1)$.
\end{proposition}

\begin{proof}
By Theorem~\ref{mainthm}, the resolvent formula \eqref{resformula}
holds for all $\lambda\in\rho(\AB)\cap\rho(A_0)$,
and it can also be written in the form
\begin{equation}\label{resformula2}
  (A_{[B]}-\lambda)^{-1}-(A_0-\lambda)^{-1}
  = \overline{\gamma(\lambda)}\bigl(I-B\overline{M(\lambda)}\bigr)^{-1}B		
  \gamma(\overline\lambda)^*.
\end{equation}
Moreover, it follows from~\eqref{sp1} and~\cite[Proposition~3.5\,(ii)]{BLL13IEOT}
that $\gamma(\lambda)^*\in\sS_p(\cH,\cG)$ for all $\lambda\in\rho(A_0)$ and, hence,
also $\ov{\gamma(\lambda)}=\gamma(\lambda)^{**}\in\sS_p(\cG,\cH)$
for all $\lambda \in \rho(A_0)$.

To prove \eqref{sp2}, let first $\lambda = \lambda_0$ be given as in the assumptions
of the proposition. Since $B\gamma(\overline\lambda)^*\in\cB(\cH,\cG)$
can be shown as in \eqref{bg} and $(I-B\ov{M(\lambda)})^{-1}\in\cB(\cG)$
holds by assumption (i) of Theorem~\ref{mainthm}, it is clear that the
right-hand side of \eqref{resformula2} belongs to the Schatten--von Neumann
ideal $\sS_p(\cH)$, which proves~\eqref{sp2} for $\lambda = \lambda_0$.
With the help of~\cite[Lemma 2.2]{BLL13IEOT} this property extends to
all~$\lambda\in\rho(\AB)\cap\rho(A_0)$.

Assume now, in addition, that $A_1$ is self-adjoint and fix some
$\lambda \in \rho(\AB)\cap\rho(A_0)\cap\rho(A_1)$.
Note that by~\cite[Theorem~3.8]{BLL13IEOT} the identity
\begin{equation}\label{resformulaA1}
  (A_1 - \lambda)^{-1} - (A_0 - \lambda)^{-1}
  = - \ov{\gamma(\lambda)}M(\lambda)^{-1}\gamma(\overline{\lambda})^*
\end{equation}
is true.
It follows from \cite[Proposition~6.14\,(iii)]{BL12} that the
operator $M(\lambda)^{-1}$ is closable,
and~\cite[Proposition~2.6\,(iii)]{BL07} implies that
\[
  \ran\bigl(\gamma(\overline{\lambda})^*\bigr)
  \subset \ran\Gamma_1 = \ran M(\lambda).
\]
Thus, the operator $M(\lambda)^{-1}\gamma(\overline{\lambda})^*$
is everywhere defined and closable and hence closed, so that
$M(\lambda)^{-1}\gamma(\overline{\lambda})^*\in \cB(\cH,\cG)$.
Since $\overline{\gamma(\lambda)} \in \sS_p(\cG,\cH)$ by the first part of the proof,
the identity~\eqref{resformulaA1} implies that
\begin{equation}\label{wichtig}
  (A_1 - \lambda)^{-1} - (A_0 - \lambda)^{-1} \in \sS_p(\cH).
\end{equation}
From~\eqref{sp2} and~\eqref{wichtig} we conclude that~\eqref{sp2mitA1} holds 	
for all $\lambda\in\rho(\AB)\cap\rho(A_0)\cap\rho(A_1)$,
and again with the help of \cite[Lemma 2.2]{BLL13IEOT}
this property extends to all $\lambda\in\rho(\AB)\cap\rho(A_1)$.
\end{proof}

In the case when $B$ is bounded and everywhere defined the assertion of the
previous proposition improves as follows.

\begin{proposition}\label{spprop2}
Let $\{\cG,\Gamma_0,\Gamma_1\}$ be a quasi boundary triple for $T\subset S^*$
with corresponding $\gamma$-field $\gamma$ and Weyl function $M$.
Let $B\in\cB(\cG)$ and assume that there exists $\lambda_0\in\rho(A_0)$
such that conditions {\rm(i)--(iii)} in Corollary~\ref{maincor} are satisfied.
Further, assume that
\begin{equation}\label{sp3}
  \gamma(\lambda_1)^* \in \sS_p(\cH,\cG)
\end{equation}
for some $\lambda_1\in\rho(A_0)$ and some $p > 0$. Then
\begin{equation}\label{sp4}
  (A_{[B]}-\lambda)^{-1}-(A_0-\lambda)^{-1} \in \sS_\frac{p}{2}(\cH)
\end{equation}
for all $\lambda\in\rho(A_{[B]})\cap\rho(A_0)$.
If, in addition, $A_1$ is self-adjoint and
\[
  M(\lambda_2)^{-1} \gamma(\overline{\lambda_2})^* \in \sS_q(\cH,\cG)
\]
for some $\lambda_2 \in \rho(A_0) \cap \rho(A_1)$
and some $q > 0$, then
\begin{equation}\label{sp2mitA1nochmal}
  (A_{[B]} - \lambda)^{-1}-(A_1 - \lambda)^{-1} \in \sS_r(\cH)
  \qquad \text{with} \quad
  r = \max\biggl\{\frac{p}{2}, \Bigl(\frac{1}{p} + \frac{1}{q}\Bigr)^{-1}\biggr\}
\end{equation}
for all $\lambda\in\rho(A_{[B]})\cap\rho(A_1)$.
\end{proposition}

\begin{proof}
By Corollary~\ref{maincor} the resolvent formula \eqref{resformula2}
holds for all $\lambda$ in the non-empty set $\rho(A_{[B]})\cap\rho(A_0)$.
As in the proof of Proposition~\ref{spprop}
we conclude that $\gamma(\lambda)^*\in\sS_p(\cH,\cG)$
and $\overline{\gamma(\lambda)}\in\sS_p(\cG,\cH)$ for all
$\lambda\in\rho(A_0)$.  Since $B\in\cB(\cG)$,
the operator $(I-B\overline{M(\lambda)})^{-1}B$ is also in $\cB(\cG)$,
and hence standard properties of Schatten--von Neumann ideals imply that
the right-hand side of \eqref{resformula2} belongs
to the Schatten--von Neumann ideal $\sS_\frac{p}{2}(\cH)$.

Assume now that $A_1$ is self-adjoint and that
$M(\lambda_2)^{-1} \gamma(\overline{\lambda_2})^* \in \sS_q(\cH,\cG)$
for some $\lambda_2\in\rho(A_0) \cap \rho(A_1)$.
From the first part of the proof we have that $\ov{\gamma(\lambda_2)}\in\sS_p(\cG,\cH)$.
Using the identity \eqref{resformulaA1}, standard properties
of Schatten--von Neumann classes and \cite[Lemma 2.2]{BLL13IEOT} we obtain that
\begin{equation}\label{wichtig2}
  (A_1 - \lambda)^{-1} - (A_0 - \lambda)^{-1} \in
  \sS_{(1/p + 1/q)^{-1}}(\cH)
\end{equation}
for all $\lambda\in\rho(A_0)\cap\rho(A_1)$.
From \eqref{sp4} and \eqref{wichtig2}
we conclude that \eqref{sp2mitA1nochmal} holds 	
for $\lambda \in \rho(A_{[B]})\cap\rho(A_0)\cap \rho(A_1)$,
and again \cite[Lemma~2.2]{BLL13IEOT}
shows that this property extends to all
$\lambda\in\rho(A_{[B]})\cap\rho(A_1)$.
\end{proof}

\begin{remark}\label{re:weakSvN}
Propositions~\ref{spprop} and \ref{spprop2} can also be formulated
for abstract operator ideals (see \cite{BLL13IEOT} and \cite{P87} for more details).
In particular, they remain true for the so-called weak Schatten--von Neumann
ideals $\sS_{p,\infty}$ and $\sS^{(0)}_{p,\infty}$ instead of $\sS_p$,
where the ideals $\sS_{p,\infty}$ and $\sS^{(0)}_{p,\infty}$
consist of those compact operators whose singular values $s_k$
satisfy $s_k=\rmO(k^{-1/p})$ and $s_k=\rmo(k^{-1/p})$, respectively, as $k\to\infty$;
cf.\ \cite{GK69}.
\end{remark}

% *******************************************************************
\section{Consequences of the decay of the Weyl function}
\label{sec:consequences}
% *******************************************************************

\noindent
In this section we continue the theme from Section~\ref{sec:closedoperators}.
In addition to the assumptions of the previous section
we now assume that the Weyl function $M$ decays
as $\dist(\lambda,\sigma(A_0)) \to \infty$.
In the first theorem we deal with a situation where $A_0$ is bounded from below.
Recall from~\eqref{mmm} that in this case a decay assumption of the
form $\|\overline{M(\lambda)}\| \to 0$ as $\lambda \to - \infty$
implies that $\overline{M(\lambda)}$ is a non-negative operator in $\cG$
for all $\lambda < \min \sigma(A_0)$.
The following theorem is now a consequence of Corollary~\ref{corchenaaa};
cf.\ \cite[Theorem~2.8]{BLLR17} for the special case when $B$ is symmetric.
Recall that a linear operator $A$ in a Hilbert space is
called \emph{dissipative} (resp., \emph{accumulative})
if $W(A) \subset\ov{\dC^+}$ (resp., $W(A)\subset\ov{\dC^-}$),
and \emph{maximal dissipative} (resp., \emph{maximal accumulative})
if $W(A) \subset\ov{\dC^+}$ and $\rho(A)\cap\dC^- \ne \emptyset$ (resp., $W(A) \subset\ov{\dC^-}$
and $\rho(A)\cap\dC^+ \ne \emptyset$).

\begin{theorem}\label{thm:operators1}
Let $\{\cG,\Gamma_0,\Gamma_1\}$ be a quasi boundary triple for $T\subset S^*$
with corresponding Weyl function $M$.
Assume that $A_0$ is bounded from below,
that $M(\lambda)$ is bounded for one (and hence for all) $\lambda\in\rho(A_0)$
and that
\begin{equation}\label{eq:Mconvergence}
  \bigl\|\overline{M(\lambda)}\bigr\| \to 0 \quad \text{as} \quad \lambda \to - \infty.
\end{equation}
Let $B$ be a closable operator in $\cG$ and assume that there
exists $b\in\dR$ such that
\begin{myenum}
\item % -----
$\Re(Bx,x) \leq b\|x\|^2 $ \, for all $x \in \dom B$;
\item % -----
$\ran\ov{M(\lambda)}^{1/2}\subset\dom B$ \, for all $\lambda < \min\sigma(A_0)$;
\item % -----
$B\bigl(\ran\ov{M(\lambda)}\bigr)\subset\ran\Gamma_0$ \,
for all $\lambda < \min\sigma(A_0)$;
\item % -----
$\ran\Gamma_1\subset\dom B$;
\item % -----
$B(\ran\Gamma_1)\subset\ran\Gamma_0$ \, or \, $\rho(A_1)\cap(-\infty,\min\sigma(A_0))\ne\emptyset$.
\end{myenum}
Then the operator
\begin{equation}\label{ab2}
  A_{[B]}f=Tf, \qquad
  \dom A_{[B]}=\bigl\{f\in\dom T:\Gamma_0 f= B\Gamma_1 f\bigr\},
\end{equation}
is a closed extension of $S$ in $\cH$ and
\begin{equation}\label{eq:resSet}
  \big\{\lambda<\min\sigma(A_0): b\|\overline{M(\lambda)}\| < 1 \big\}
  \subset \rho(A_{[B]}).
\end{equation}
In particular, there exists $\mu \leq \min \sigma(A_0)$ such
that $(- \infty, \mu) \subset \rho(A_{[B]})$.  Moreover, the resolvent formula
\begin{equation}\label{resformula4}
  (A_{[B]}-\lambda)^{-1}=(A_0-\lambda)^{-1}
  +\gamma(\lambda)\bigl(I-BM(\lambda)\bigr)^{-1}B\gamma(\overline\lambda)^*
\end{equation}
holds for all $\lambda \in \rho(A_{[B]}) \cap \rho(A_0)$.
If, in addition, $B$ is symmetric (dissipative, accumulative, respectively),
then $A_{[B]}$ is self-adjoint and bounded from below
(maximal accumulative, maximal dissipative, respectively).
	
Further, let $B'$ be a linear operator in $\cG$ that satisfies {\rm(i)--(v)}
with $B$ replaced by $B'$ and assume that
\begin{equation}\label{BBprime_sec5}
  (Bx,y) = (x,B'y) \qquad\text{for all}\;\;x\in\dom B,\;y\in\dom B'.
\end{equation}
Then $A_{[B']}=A_{[B]}^*$ and the left-hand side of \eqref{eq:resSet}
is contained in $\rho(A_{[B']})$.
\end{theorem}

\begin{proof}
First note that it can be shown in the same way as in Step~5 in the
proof of Theorem~\ref{mainthm} that the second condition in (v)
and (ii)--(iv) imply the first condition in (v).
Further, the assumption~\eqref{eq:Mconvergence} implies $M(\lambda) \ge 0$
for every $\lambda<\min\sigma(A_0)$; see~\eqref{mmm}.
It follows from Corollary~\ref{corchenaaa} that $\AB$ is a
closed extension of $S$ in $\cH$ and that every point $\lambda<\min\sigma(A_0)$
with the property $b\|\overline{M(\lambda)}\| < 1$ belongs to $\rho(\AB)$.
Note that such $\lambda$ exist due to the decay condition \eqref{eq:Mconvergence}.
Condition~\eqref{eq:Mconvergence} and relation \eqref{eq:resSet} also imply
that there exists $\mu \leq \min \sigma(A_0)$ with
\begin{equation}\label{eq:semibounded}
  (-\infty, \mu) \subset \rho(A_{[B]}).
\end{equation}
The resolvent formula \eqref{resformula4} and the assertions on  $A_{[B']}$
are immediate from Corollary~\ref{corchenaaa}.
	
It remains to show that $A_{[B]}$ is self-adjoint (maximal accumulative,
maximal dissipative, respectively) if $B$ is symmetric
(dissipative, accumulative, respectively).
For this let $f \in \dom A_{[B]}$ and observe that the abstract
Green identity \eqref{green} yields
\begin{equation}\label{eq:Im}
\begin{split}
  \Im(A_{[B]}f,f) &= \frac{1}{2i}\bigl((Tf,f)-(f,Tf)\bigr)
  = \frac{1}{2i}\bigl((\Gamma_1 f,\Gamma_0 f)-(\Gamma_0f,\Gamma_1 f)\bigr) \\
  &= \frac{1}{2i}\Bigl((\Gamma_1 f,B\Gamma_1 f)-(B\Gamma_1 f,\Gamma_1 f)\Bigr)
  = - \Im(B\Gamma_1 f,\Gamma_1 f).
\end{split}
\end{equation}
If $B$ is symmetric (dissipative, accumulative), then $\Im(Bx,x)$ is zero
(non-negative, non-positive, respectively) for all $x \in \dom B$,
and it follows from~\eqref{eq:Im} that $A_{[B]}$ is symmetric
(accumulative, dissipative, respectively).
Now~\eqref{eq:semibounded} implies that $A_{[B]}$ is self-adjoint and
bounded from below (maximal accumulative, maximal dissipative, respectively).
\end{proof}

In the case when $\{\cG, \Gamma_0, \Gamma_1\}$ is a generalized boundary triple,
Theorem~\ref{thm:operators1} simplifies in the following way.

\begin{corollary}\label{cor:operators1GenBT}
Let $\{\cG, \Gamma_0, \Gamma_1\}$ be a generalized boundary triple
for $T \subset S^*$ with corresponding Weyl function $M$.
Assume that $A_0$ is bounded from below
and that
\begin{align*}
  \|M(\lambda)\| \to 0 \quad \text{as} \quad \lambda \to - \infty.
\end{align*}
Let $B$ be a closable operator in $\cG$ and assume that there exists $b\in\dR$ such that
\begin{myenum}
\item % -----
$\Re(Bx,x) \leq b\|x\|^2 $ for all $x \in \dom B$;
\item % -----
$\ran M(\lambda)^{1/2}\subset\dom B$ for all $\lambda < \min\sigma(A_0)$;
\item % -----
$\ran\Gamma_1\subset\dom B$.
\end{myenum}
Then the operator $A_{[B]}$ in~\eqref{ab2} is a closed extension of $S$ in $\cH$ and
\begin{equation}\label{eq:resSet3}
  \big\{\lambda < \min\sigma(A_0): b \| M(\lambda) \| < 1\big\} \subset \rho(A_{[B]}).
\end{equation}
In particular, there exists $\mu \leq \min \sigma(A_0)$
such that $(-\infty,\mu) \subset \rho(A_{[B]})$.
Moreover, the resolvent formula~\eqref{resformula4} holds for all
$\lambda \in \rho(A_{[B]}) \cap \rho(A_0)$.
If, in addition, $B$ is symmetric (dissipative, accumulative, respectively),
then $A_{[B]}$ is self-adjoint and bounded from below
(maximal accumulative, maximal dissipative, respectively).
	
Further, let $B'$ be a linear operator in $\cG$ that satisfies {\rm(i)--(iii)}
with $B$ replaced by $B'$ and assume that \eqref{BBprime_sec5} holds.
Then $A_{[B']}=A_{[B]}^*$ and the left-hand side of \eqref{eq:resSet3}
is contained in $\rho(A_{[B']})$.
\end{corollary}

\begin{remark}\label{obtremaa}
Note that for an ordinary boundary triple $\{\cG, \Gamma_0, \Gamma_1\}$
condition (iv) in Theorem~\ref{thm:operators1}
(condition (iii) in Corollary~\ref{cor:operators1GenBT}) implies that $B\in\cB(\cG)$.
In this situation the conditions (ii), (iii), and the first condition in (v)
in Theorem~\ref{thm:operators1} (condition (ii) in Corollary~\ref{cor:operators1GenBT})
are automatically satisfied. We shall formulate a corollary on spectral enclosures
in the case of an ordinary boundary triple in Corollary~\ref{obtcoraa} below.
\end{remark}

Let us formulate another corollary of Theorem~\ref{thm:operators1}
(in particular, of the inclusion in \eqref{eq:resSet}).

\begin{corollary}\label{cor:bnegative}
Let all assumptions of Theorem~\ref{thm:operators1} be satisfied and
assume that $b \leq 0$ in {\rm(i)} of Theorem~\ref{thm:operators1}.
Then the closed operator $A_{[B]}$ in~\eqref{ab2} satisfies
\[
  \bigl(-\infty, \min\sigma(A_0)\bigr) \subset \rho(A_{[B]}).
\]
\end{corollary}

\medskip

We now turn to situations where the rate of decay of the Weyl function
for $\lambda \to -\infty$ is known in more detail.
In such cases we derive spectral estimates for the operator $A_{[B]}$,
which refine the inclusion~\eqref{eq:resSet} in Theorem~\ref{thm:operators1}.
The following proposition provides a first, easy step towards this.
Here we assume that $b$ in Theorem~\ref{thm:operators1}\,(i) is positive;
the case $b\le0$ is treated in Corollary~\ref{cor:bnegative} above.
The proposition is a generalization of \cite[Theorem~2.8\,(b)]{BLLR17}
to the non-self-adjoint setting.

\begin{proposition}\label{prop:specEstReal}
Let $\{\cG,\Gamma_0,\Gamma_1\}$ be a quasi boundary triple for $T\subset S^*$
with corresponding Weyl function $M$.
Assume that $A_0$ is bounded from below, that $M(\lambda)$ is bounded for one
(and hence for all) $\lambda\in\rho(A_0)$ and that there exist
$\beta\in(0,1]$, $C > 0$ and $\mu \leq \min \sigma(A_0)$ such that
\begin{equation}\label{powerdecay}
  \bigl\|\overline{M(\lambda)}\bigr\| \le \frac{C}{(\mu-\lambda)^\beta}
  \qquad \text{for all }\;\lambda < \mu.
\end{equation}
Moreover, let $B$ be a closable operator in $\cG$, let $b>0$, and assume that
conditions {\rm(i)--(v)} in Theorem~\ref{thm:operators1} are satisfied.
Then the operator $A_{[B]}$ in~\eqref{ab2} is closed and satisfies
\begin{equation}\label{eq:specInclReal}
  \bigl(-\infty,\mu-(Cb)^{1/\beta}\bigr) \subset \rho(A_{[B]}).
\end{equation}
\end{proposition}

\begin{proof}
That $\AB$ is closed follows from Theorem~\ref{thm:operators1}.
Consider $\lambda < \mu-(Cb)^{1/\beta}$.  Then $(\mu-\lambda)^\beta > C b$ and hence
\[
  b\bigl\|\overline{M(\lambda)}\bigr\| \leq b\frac{C}{(\mu-\lambda)^\beta} < 1.
\]
Now Theorem~\ref{thm:operators1} yields that $\lambda \in \rho(A_{[B]})$.
\end{proof}

In the next theorem we study the m-sectorial case discussed in Theorem~\ref{thmsec}
in more detail and obtain refined estimates for the numerical range of $\AB$.
Roughly speaking, if the Weyl function decays for $\lambda\to -\infty$,
then there exists an $\eta_*\in\RR$ such that the assumptions in
Theorem~\ref{thmsec} are satisfied for every $\eta<\eta_*$ and hence
\[
  \sigma(\AB) \subset \ov{W(\AB)} \subset
  \bigcap_{\eta \in (-\infty,\eta_*)} \cS_{\eta}(B).
\]
In the particular case when $\Im B$ is bounded and the Weyl function satisfies
a decay condition as in Proposition~\ref{prop:specEstReal}, we use this fact
to obtain an extension of Proposition~\ref{prop:specEstReal} including estimates
for the non-real spectrum.

\begin{theorem}\label{thm:parabola}
Let $\{\cG,\Gamma_0,\Gamma_1\}$ be a quasi boundary triple for $T\subset S^*$
with corresponding Weyl function $M$ and suppose that $A_1$ is self-adjoint
and that $A_0$ and $A_1$ are bounded from below.
Further, assume that $M(\lambda)$ is bounded for one
(and hence for all) $\lambda\in\rho(A_0)$
and that there exist $\beta\in(0,1]$, $C>0$ and $\mu\le\min\sigma(A_0)$ such that
\begin{equation}\label{eq:estMpower}
  \bigl\|\ov{M(\lambda)}\bigr\| \le \frac{C}{(\mu-\lambda)^\beta}
  \qquad \text{for every}\;\; \lambda<\mu.
\end{equation}
Moreover, let $B$ be a closable linear operator in $\cG$ and
let $b\in\RR$ such that conditions {\rm(i)--(iv)} in
Theorem~\ref{thm:operators1} are satisfied.
Then the operator $\AB$ in \eqref{ab2} is m-sectorial and,
in particular, the inclusion $\sigma(\AB) \subset \ov{W(\AB)}$ holds.
	
Assume, in addition, that $\dom B^*\supset\dom B$ and that $\Im B$ is bounded.
Then the following assertions are true.
\begin{myenuma}
\item % ----- (a)
If\, $b>0$, then for every $\xi < \mu-(Cb)^{1/\beta}$,
\begin{equation}\label{inclWABa}
  W(\AB) \subset \Bigl\{z \in \C : \Re z \ge \mu-(Cb)^{1/\beta},\;
  |\Im z| \le K_\xi(\Re z-\xi)^{1-\beta}\Bigr\}, \hspace*{-3ex}
\end{equation}
where
\[
  K_\xi = \frac{2C\bigl\|\ov{\Im B}\bigr\|}{1-\frac{Cb}{(\mu-\xi)^\beta}}\,.
\]
\item % ----- (b)
If\, $b=0$, then
\begin{equation}\label{inclWABb}
  W(\AB)
  \subset \Bigl\{z \in \C : \Re z \ge \mu,\;
  |\Im z| \le K'_\beta(\Re z-\mu)^{1-\beta}\Bigr\},
\end{equation}
where
\begin{equation}\label{defKprbeta}
  K'_\beta = \begin{cases}
    \dfrac{C\bigl\|\ov{\Im B}\bigr\|}{\beta^\beta(1-\beta)^{1-\beta}} & \text{if}\;\; 0<\beta<1,
    \\[3ex]
    C\bigl\|\ov{\Im B}\bigr\| & \text{if}\;\; \beta=1,
  \end{cases}
\end{equation}
and the convention $0^0=1$ is used in \eqref{inclWABb} when $\beta=1$ and $\Re z=\mu$.
Moreover, $K'_\beta$ satisfies
$C\bigl\|\ov{\Im B}\bigr\| \le K'_\beta \le 2C\bigl\|\ov{\Im B}\bigr\|$.
\item % ----- (c)
If\, $b<0$, then
\begin{equation}\label{inclWABc}
  W(\AB)
  \subset \biggl\{z \in \C : \Re z \ge \mu,\;
  |\Im z| \le \frac{2C\bigl\|\ov{\Im B}\bigr\|(\Re z-\mu)}{(\Re z-\mu)^\beta-Cb}\biggr\}.
\end{equation}
\end{myenuma}
\end{theorem}

\medskip

\noindent
See Figure~\ref{fig:parabolic_regions} for plots of the regions given
by the right-hand sides of \eqref{inclWABa}, \eqref{inclWABb}, \eqref{inclWABc}.
Notice that in Theorem~\ref{thm:parabola}\,(a)
we get, in fact, a family of enclosures in parabola-type regions that depend
on the choice of the parameter $\xi$. By intersecting all these regions
with respect to $\xi \in (-\infty,\mu - (Cb)^{1/\beta})$ one gets
a finer enclosure for the numerical range of $\AB$.

\begin{figure}[ht]
\begin{center}
\hspace*{-8ex}
\begin{tabular}{cc}
\begin{minipage}{7cm}
\begin{center}
\includegraphics[width=3cm]{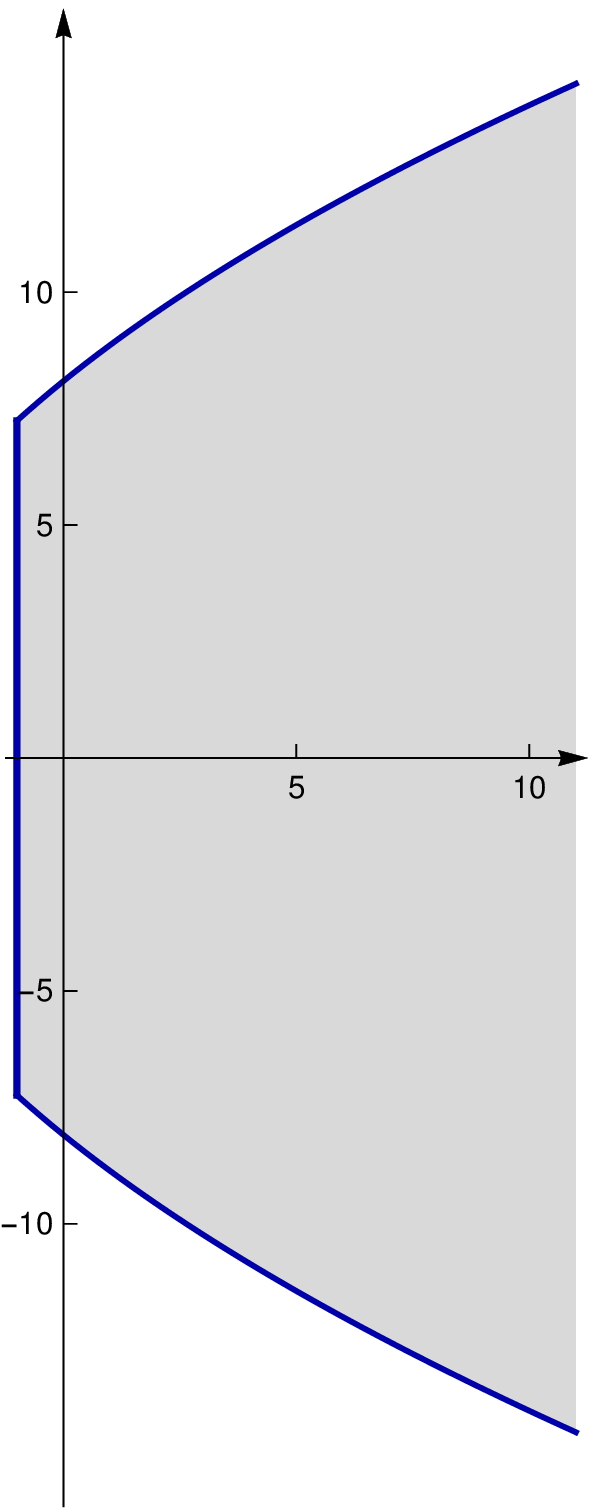} \\[1ex]
(a) $b=1$, $\beta=\frac{1}{2}$, $\xi=-5$
\\[2ex]
\includegraphics[width=4cm]{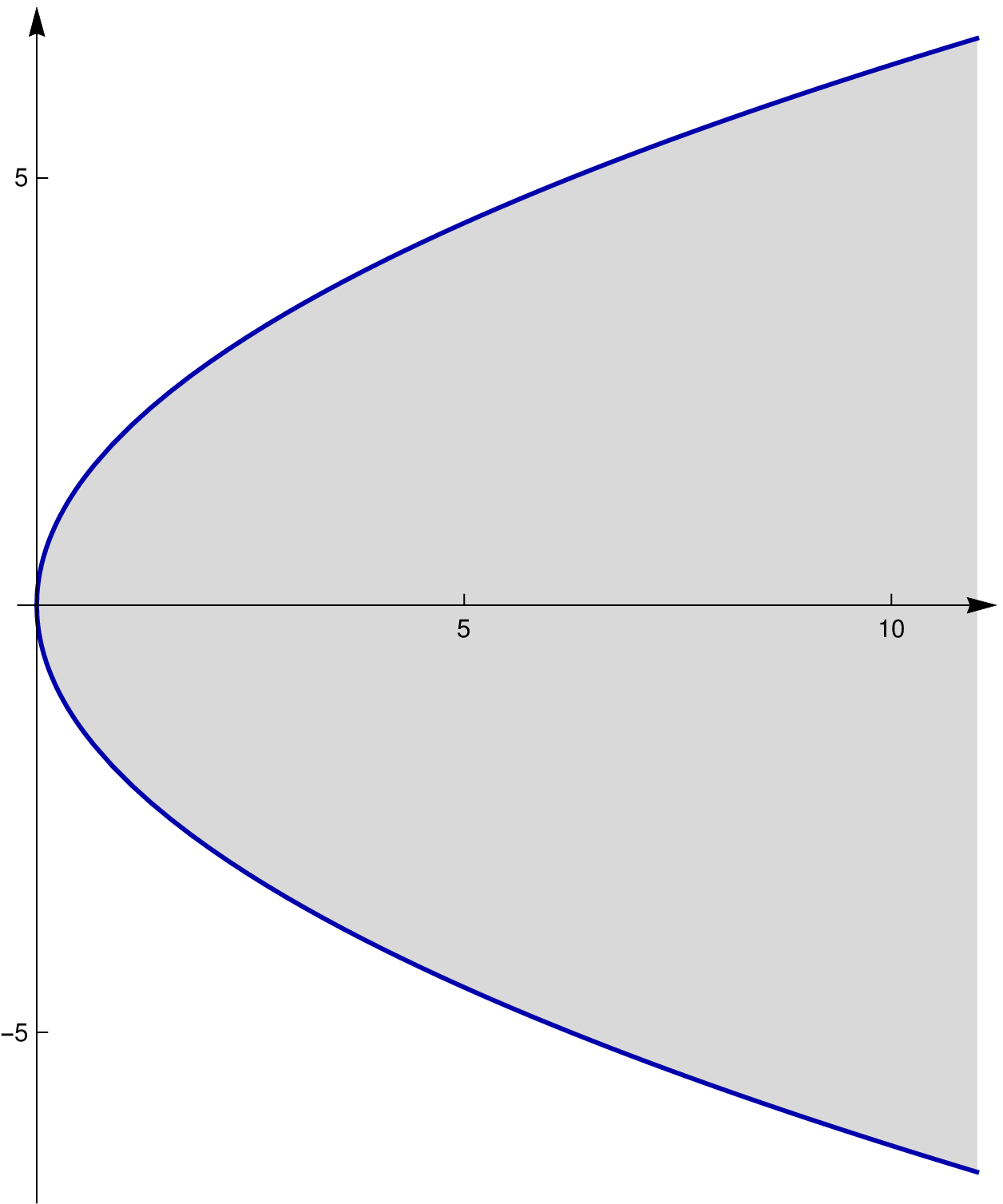} \\[1ex]
(b) $b=0$, $\beta=\frac{1}{2}$
\end{center}
\end{minipage}
\begin{minipage}{7cm}
\begin{center}
\includegraphics[width=5cm]{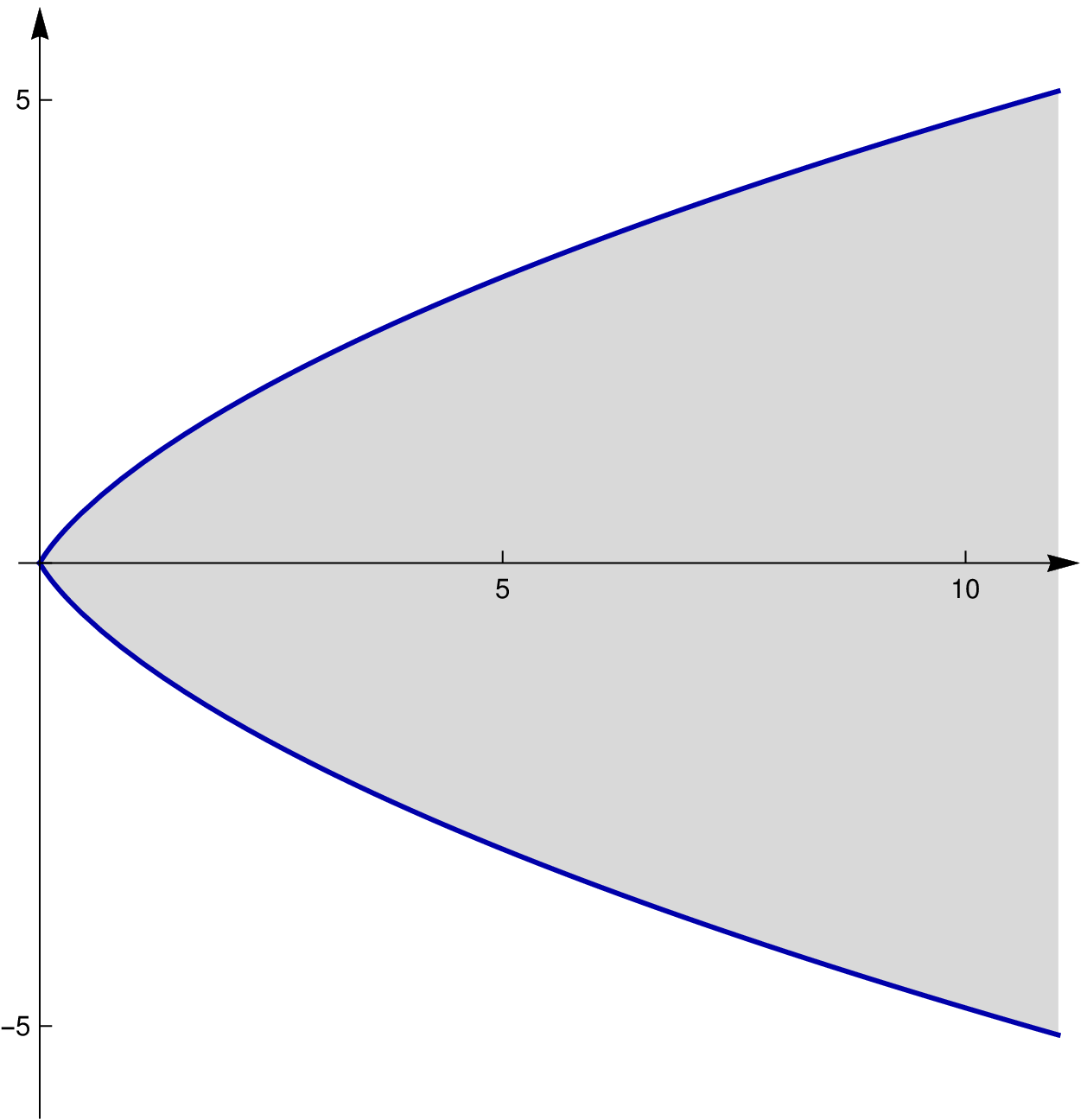} \\[1ex]
(c) $b=-1$, $\beta=\frac{1}{2}$
\\[4ex]
\includegraphics[width=5.5cm]{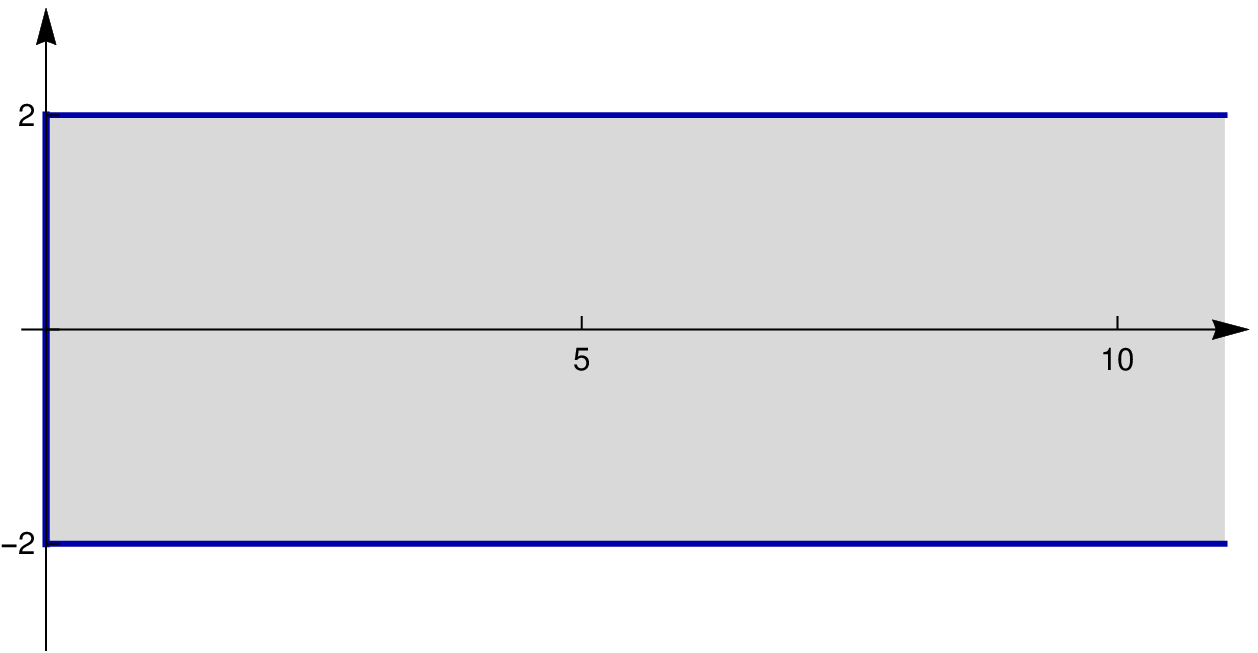} \\[1ex]
(d) $b=0$, $\beta=1$
\\[4ex]
\includegraphics[width=5.5cm]{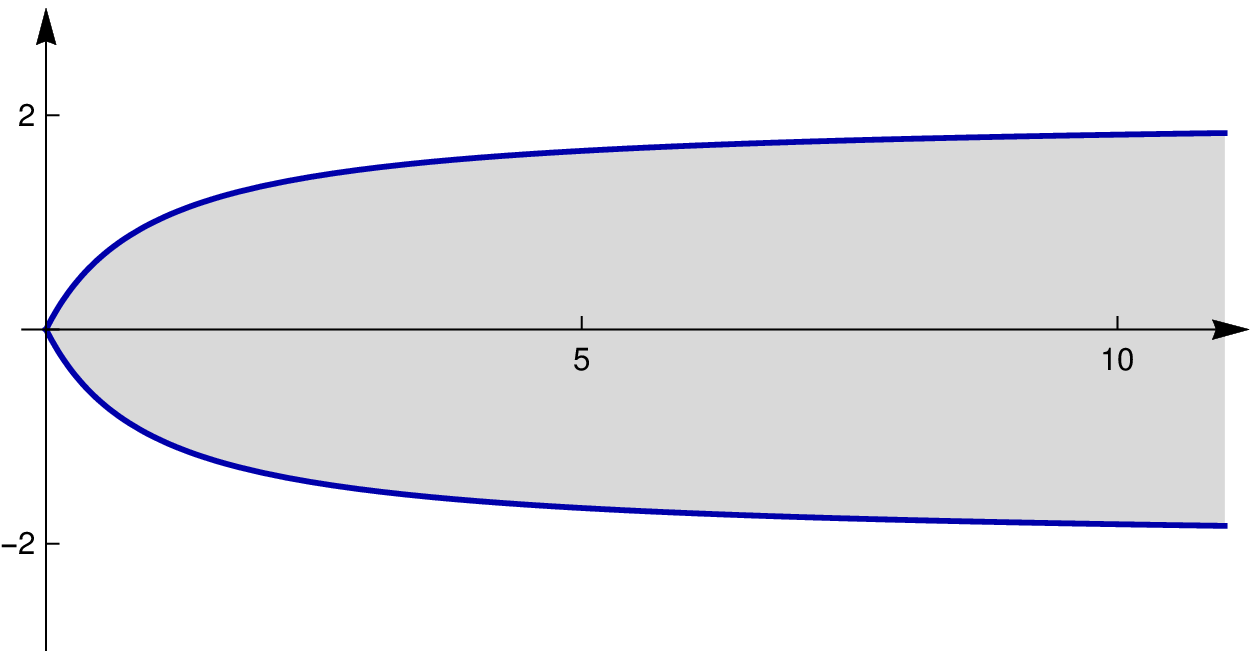} \\[1ex]
(e) $b=-1$, $\beta=1$
\end{center}
\end{minipage}
\end{tabular}
\caption{The plots show the regions given by the right-hand sides
of \eqref{inclWABa}, \eqref{inclWABb}, \eqref{inclWABc}
with $\mu=0$, $C=1$ and $\|\ov{\Im B}\|=1$ for the
following cases: $\beta=\frac{1}{2}$ in (a)--(c) ($b>0$, $b=0$, $b<0$, respectively)
and $\beta=1$ in (d), (e) ($b=0$, $b<0$, respectively).}
\label{fig:parabolic_regions}
\end{center}
\end{figure}

\begin{proof}
Note first, that the conditions of Theorem~\ref{thmsec} are satisfied;
we point out, particularly, that by~\eqref{eq:estMpower} and~\eqref{mmm}
we have $M(\lambda) \ge 0$ for each $\lambda < \min\sigma(A_0) \le \min\sigma(A_1)$
(see Proposition~\ref{pr:A1bddbelow}), and there exists $\eta < \mu$
such that $b\|\ov{M(\eta)}\| < 1$.  Hence, $\AB$ is sectorial.
Since $A_1$ is self-adjoint and bounded from below and the assumptions (i)--(iv)
in Theorem~\ref{thm:operators1} hold, the latter yields $\eta\in\rho(\AB)$.
Thus $\AB$ is m-sectorial and hence $\sigma(\AB)\subset\ov{W(\AB)}$.
	
For the rest of the proof assume that $\dom B^*\supset\dom B$ and
that $\Im B$ is bounded.  For every $\lambda<\min\sigma(A_0)$ we
have $\ran\ov{M(\lambda)}^{1/2}\subset\dom B\subset\dom B^*$ by condition~(ii)
of Theorem~\ref{thm:operators1}; in particular, $B\ov{M(\lambda)}^{1/2} \in \cB(\cG)$
and $B^* \ov{M(\lambda)}^{1/2} \in \cB(\cG)$.
Hence
\begin{align*}
  \bigl(\ov{M(\lambda)}^{1/2}B\ov{M(\lambda)}^{1/2}\bigr)^*
  &= \bigl(B\ov{M(\lambda)}^{1/2}\bigr)^*\ov{M(\lambda)}^{1/2}
  = \Big( \big(\ov{M(\lambda)}^{1/2}B^* \big)^* \Big)^* \ov{M(\lambda)}^{1/2}
  \\[0.5ex]
  &= \ov{\ov{M(\lambda)}^{1/2}B^*}\ov{M(\lambda)}^{1/2} = \ov{M(\lambda)}^{1/2}B^*\ov{M(\lambda)}^{1/2}.
\end{align*}
This implies that
\begin{equation}\label{estImMBM}
\begin{split}
  \bigl\|\Im\bigl(\ov{M(\lambda)}^{1/2}B\ov{M(\lambda)}^{1/2}\bigr)\bigr\|
  &= \frac{1}{2}\Bigl\|\ov{M(\lambda)}^{1/2}B\ov{M(\lambda)}^{1/2}
  -\bigl(\ov{M(\lambda)}^{1/2}B\ov{M(\lambda)}^{1/2}\bigr)^*\Bigr\|
  \\[0.5ex]
  &= \frac{1}{2}\Bigl\|\ov{M(\lambda)}^{1/2}(B-B^*)\ov{M(\lambda)}^{1/2}\Bigr\|
  \\[0.5ex]
  &\le \bigl\|\ov{\Im B}\bigr\|\,\bigl\|\ov{M(\lambda)}\bigr\|,
\end{split}
\end{equation}
where we have used that $\Im B$ is a bounded operator defined on the dense
subspace $\dom B$ of $\cG$.  Let $z\in W(\AB)$.
It follows from Theorem~\ref{thmsec} and \eqref{estImMBM} that, for every
$\eta<\min\sigma(A_0)$ for which $b\|\ov{M(\eta)}\|<1$, the inequalities
\begin{equation}\label{estRezImz}
  \Re z \ge \eta, \qquad
  |\Im z| \le \frac{\bigl\|\ov{\Im B}\bigr\|\,\bigl\|\ov{M(\eta)}\bigr\|}{1-b\bigl\|\ov{M(\eta)}\bigr\|}
  (\Re z-\eta)
\end{equation}
hold.
	
(a) Assume that $b > 0$.
For every $\eta < \mu - (Cb)^{1/\beta}$ we have $\eta<\min\sigma(A_0)$
and, by \eqref{eq:estMpower},
\begin{equation}\label{eq:estbM}
  b\bigl\|\ov{M(\eta)}\bigr\| \le \frac{Cb}{(\mu-\eta)^\beta} < 1.
\end{equation}
Hence \eqref{estRezImz} is true for each such $\eta$.
For the real part of $z$ this yields
\begin{equation}\label{eq:estx15}
  \Re z \ge \mu-(Cb)^{1/\beta}.
\end{equation}
To estimate $|\Im z|$ further, note that the function
\[
  \Bigl(-\infty, \frac{1}{b}\Bigr) \ni t
  \mapsto \frac{\bigl\|\ov{\Im B}\bigr\|t}{1 - bt}
\]
is strictly increasing and that
$\|\ov{M(\eta)}\| \le \frac{C}{(\mu-\eta)^\beta} < \frac{1}{b}$
for all $\eta < \mu - (C b)^{1/\beta}$ by~\eqref{eq:estbM}.
Hence \eqref{estRezImz} yields
\begin{equation}\label{eq:estIm}
  |\Im z| \le \frac{\bigl\|\ov{\Im B}\bigr\|
  \frac{C}{(\mu-\eta)^\beta}}{1-\frac{Cb}{(\mu-\eta)^\beta}}(\Re z-\eta)
  \qquad\text{for all}\;\; \eta < \mu - (C b)^{1/\beta}.
\end{equation}
Now let $\xi < \mu - (C b)^{1/\beta}$ be arbitrary.
Then \eqref{eq:estx15} implies that $\Re z > \xi$.
Choose $\eta \defeq 2\xi - \Re z$, which satisfies $\eta < 2\xi-\xi = \xi$.
From \eqref{eq:estIm} and $\xi < \mu$ we obtain the inequality
\[
  |\Im z| \le \frac{C\bigl\|\ov{\Im B}\bigr\|}{1-\frac{Cb}{(\mu-\xi)^\beta}}
  \cdot\frac{\Re z-\eta}{(\xi-\eta)^\beta}
  = \frac{C\bigl\|\ov{\Im B}\bigr\|}{1-\frac{Cb}{(\mu-\xi)^\beta}}
  \cdot\frac{2(\Re z-\xi)}{(\Re z -\xi)^\beta}\,,
\]
which, together with \eqref{eq:estx15}, shows \eqref{inclWABa}.
	
(b), (c) Assume now that $b \le 0$.
For every $\eta < \mu$ we have $\eta < \min \sigma(A_0)$ and
$b\|\ov{M(\eta)}\|\le0$.  Hence \eqref{estRezImz} is true for $\eta < \mu$,
which, in particular, shows that
\begin{equation}\label{eq:ersteHaelfte}
  \Re z \geq \mu.
\end{equation}
Note that $t \mapsto \frac{\|\ov{\Im B}\|t}{1 - b t}$ is strictly increasing
on $(0,\infty)$. Hence \eqref{estRezImz} and \eqref{eq:estMpower} imply that
\begin{equation}\label{eq:estkappaNochmal}
  |\Im z| \le \frac{\bigl\|\ov{\Im B}\bigr\|\,\bigl\|\ov{M(\eta)}\bigr\|}{1-b\bigl\|\ov{M(\eta)}\bigr\|}
  (\Re z-\eta)
  \le \frac{\bigl\|\ov{\Im B}\bigr\|\,\frac{C}{(\mu-\eta)^\beta}}{1-\frac{Cb}{(\mu-\eta)^\beta}}
  (\Re z-\eta).
\end{equation}
Assume first that $\Re z>\mu$.
Now we distinguish the two cases $b=0$ and $b<0$.
First let $b=0$ and $\beta\in(0,1)$.  We choose
\[
  \eta \defeq \frac{1}{1-\beta}(\mu-\beta\Re z),
\]
which yields
\[
  \Re z-\eta = \frac{1}{1-\beta}(\Re z-\mu) \qquad\text{and}\qquad
  \mu-\eta = \frac{\beta}{1-\beta}(\Re z-\mu);
\]
in particular, we have $\eta<\mu$.
Hence \eqref{eq:estkappaNochmal} implies that
\begin{align*}
  |\Im z| &\le C\bigl\|\ov{\Im B}\bigr\|\frac{\Re z-\eta}{(\mu-\eta)^\beta}
  = C\bigl\|\ov{\Im B}\bigr\|
  \frac{\frac{1}{1-\beta}(\Re z-\mu)}{\bigl[\frac{\beta}{1-\beta}(\Re z-\mu)\bigr]^\beta}
  \\[1ex]
  &= \frac{C\bigl\|\ov{\Im B}\bigr\|}{\beta^\beta(1-\beta)^{1-\beta}}(\Re z-\mu)^{1-\beta},
\end{align*}
which shows that $z$ is contained in the right-hand side of \eqref{inclWABb}.
Taking the limit $\beta\nearrow1$ we obtain this inclusion also for the case when $\beta=1$.
The estimates for $K'_\beta$ follow from the fact that the
function $f(\beta)=\beta^\beta(1-\beta)^{1-\beta}$, $\beta\in(0,1)$ has a
unique minimum at $\beta=\frac{1}{2}$ and that $f(\beta)\to1$
as $\beta\searrow0$ or $\beta\nearrow1$.

Now let $b<0$ (and still $\Re z>\mu$).
We choose $\eta \defeq 2\mu - \Re z$, which yields
\[
  \Re z - \eta = 2(\Re z - \mu) \qquad\text{and}\qquad \mu-\eta = \Re z - \mu.
\]
Therefore \eqref{eq:estkappaNochmal} implies that
\[
  |\Im z| \le \frac{2C\bigl\|\ov{\Im B}\bigr\|}{(\Re z -\mu)^\beta-Cb}(\Re z - \mu),
\]
and hence $z$ is contained in the right-hand side of \eqref{inclWABc}.
Since the numerical range $W(\AB)$ is a convex set, the inclusions
\eqref{inclWABb} and \eqref{inclWABc} hold also for $z$ with $\Re z=\mu$.
\end{proof}

\medskip

Next we formulate a variant of Theorem~\ref{thm:parabola} for the special case
when $\{\cG,\Gamma_0,\Gamma_1\}$ is an ordinary boundary triple.
In this case the assumptions in Theorem~\ref{thm:operators1} imply that $B$
is a bounded operator in $\cG$; cf.\ Remark~\ref{obtremaa}.

\begin{corollary}\label{obtcoraa}
Let $\{\cG,\Gamma_0,\Gamma_1\}$ be an ordinary boundary triple for $S^*$
with corresponding Weyl function $M$ and suppose that the self-adjoint
operators $A_0$ and $A_1$ are bounded from below.
Further, assume that there exist $\beta\in(0,1]$, $C>0$ and
$\mu\le\min\sigma(A_0)$ such that
\[
  \bigl\|M(\lambda)\bigr\| \le \frac{C}{(\mu-\lambda)^\beta}
  \qquad \text{for every}\;\; \lambda<\mu.
\]
Let $B\in\cB(\cG)$ be a bounded, everywhere defined operator in $\cG$ and
let $b\in\RR$ be such that $\Re(Bx,x) \le b\|x\|^2$  for all $x \in \cG$.
Then the operator $\AB$ in \eqref{ab2} is m-sectorial and,
in particular, the inclusion $\sigma(\AB) \subset \ov{W(\AB)}$ holds.
Moreover, the assertions in Theorem~\ref{thm:parabola}\,(a), (b) and (c) are true.
\end{corollary}

In the following theorem we drop the assumption that $A_0$ is bounded from below,
but we assume that $B\in\cB(\cG)$.
We remark that the condition~\eqref{eq:Mconvergence} does no longer make sense
if $A_0$ is not bounded from below.  Therefore we replace it by the
more appropriate condition~\eqref{eq:Mconvergence'} below.

\begin{theorem}\label{thm:operators2}
Let $\{\cG,\Gamma_0,\Gamma_1\}$ be a quasi boundary triple for $T\subset S^*$
with corresponding Weyl function $M$.  Assume that $M(\lambda)$ is bounded
for one (and hence for all) $\lambda\in\rho(A_0)$ and that
\begin{equation}\label{eq:Mconvergence'}
  \bigl\|\overline{M(re^{i\phi})}\bigr\| \to 0 \qquad \text{as} \;\; r \to  \infty
\end{equation}
for some fixed $\phi \in (-\pi,0)\cup(0,\pi)$. Let $B \in \cB(\cG)$ be such that
\begin{myenum}
\item
$B\bigl(\ran\ov{M(\lambda)}\bigr)\subset\ran\Gamma_0$ \, for all $\lambda\in\rho(A_0)$;
\item
$B(\ran\Gamma_1)\subset\ran\Gamma_0$ \, or \, $A_1$ is self-adjoint.
\end{myenum}
Then the operator $\AB$ in \eqref{ab2} is closed,
the resolvent formula \eqref{resformula4} holds for
all $\lambda\in \rho(\AB) \cap \rho(A_0)$, and
\begin{equation}\label{eq:resSet2}
  \big\{\lambda\in\rho(A_0): \bigl\|B\ov{M(\lambda)}\bigr\|<1\big\} \subset \rho(A_{[B]}).
\end{equation}
In particular, for every interval $[\psi_1,\psi_2]\subset(-\pi,0)$
or $[\psi_1,\psi_2]\subset(0,\pi)$ there exists $R_{[\psi_1,\psi_2]}>0$ such that
\begin{equation}\label{partsectrho}
  \bigl\{re^{i\psi}: r\ge R_{[\psi_1,\psi_2]},\,\psi\in[\psi_1,\psi_2]\bigr\}
  \subset \rho(\AB).
\end{equation}
Moreover, if $B$ is self-adjoint (accumulative, dissipative, respectively),
then $\AB$ is self-adjoint (maximal dissipative, maximal accumulative, respectively).

Further, if conditions {\rm(i)} and {\rm(ii)} are satisfied also for the adjoint
operator $B^*$ instead of $B$, then $A_{[B^*]}=A_{[B]}^*$.
\end{theorem}

\begin{proof}
Let $\lambda \in \rho(A_0)$ with $\|B \overline{M(\lambda)}\| < 1$;
such $\lambda$ exist by \eqref{eq:Mconvergence'}.  Then
\[
  1\in\rho\bigl(B\overline{M(\lambda)}\bigr).
\]
It follows from this and the assumptions of the current theorem that
Corollary~\ref{maincor} can be applied. Thus $A_{[B]}$ is closed,
the resolvent formula \eqref{resformula4} holds for
all $\lambda \in \rho(A_{[B]}) \cap \rho(A_0)$, \eqref{eq:resSet2} is valid,
and the statement on $A_{[B^*]}$ follows.  The relation~\eqref{partsectrho}
follows from \eqref{eq:Mconvergence'}, Lemma~\ref{le:raysector} and \eqref{eq:resSet2}.
If $B$ is symmetric (accumulative, dissipative, respectively), then it follows
as in the proof of Theorem~\ref{thm:operators1} that $A_{[B]}$ is symmetric
(dissipative, accumulative, respectively).
This, together with \eqref{partsectrho}, implies the remaining assertions.
\end{proof}

The next proposition complements Proposition~\ref{prop:specEstReal}.
Here we require a decay condition on the Weyl function on a set $G \subset \rho(A_0)$
that is sufficiently large.  In later sections this is applied to,
e.g.\ all of $\rho(A_0)$ or to certain sectors in the complex plane.

\begin{proposition}\label{prop:specEst}
Let $\{\cG,\Gamma_0,\Gamma_1\}$ be a quasi boundary triple for $T\subset S^*$
with corresponding Weyl function $M$ and assume that $M(\lambda)$ is bounded
for one (and hence for all) $\lambda\in\rho(A_0)$.
Further, let $B\in\cB(\cG)$ such that
\begin{myenum}
\item
$B\bigl(\ran\ov{M(\lambda)}\bigr)\subset\ran\Gamma_0$ for all $\lambda\in\rho(A_0)$;
\item
$B(\ran\Gamma_1)\subset\ran\Gamma_0$ \, or \, $A_1$ is self-adjoint.
\end{myenum}
Let $G \subset \rho(A_0)$ be a set such that there exist $\lambda_n \in G$, $n\in\NN$,
with
\[
  \dist(\lambda_n, \sigma(A_0)) \to \infty \quad \text{as}\;\; n \to \infty.
\]
Then the following assertions hold.
\begin{myenuma}
\item % -----
If there exist $\beta\in(0,1]$ and $C>0$ such that
\[
  \bigl\|\overline{M(\lambda)}\bigr\|
  \le \frac{C}{\bigl(\dist(\lambda,\sigma(A_0))\bigr)^\beta}
  \qquad \text{for all }\;\lambda\in G,
\]
then $\AB$ is a closed extension of $S$ and
\begin{equation}\label{spec_incl}
  \sigma(A_{[B]})\cap G \subset
  \Bigl\{z \in G: \dist(z,\sigma(A_0))
  \le \bigl(C\|B\|\bigr)^{1/\beta}\Bigr\}.
\end{equation}
\item % -----
If there exist $\beta\in(0,1]$, $C>0$ and $\mu \leq \min \sigma(A_0)$ such that
\begin{equation}\label{nuistgut}
  \bigl\|\overline{M(\lambda)}\bigr\| \le \frac{C}{|\lambda - \mu|^\beta}
  \qquad \text{for all }\;\lambda\in G,
\end{equation}
then $\AB$ is a closed extension of $S$ and
\begin{equation}\label{spec_inclnochmal}
  \sigma(\AB)\cap G \subset
  \Bigl\{z \in G: |z-\mu| \le \bigl(C\|B\|\bigr)^{1/\beta}\Bigr\}.
\end{equation}
\end{myenuma}
\end{proposition}

\begin{proof}
We prove only assertion (a); the proof of the second assertion is analogous.
Assume first that condition~(i) and the first condition in~(ii) are satisfied.
By the assumption on $G$, there exists $\lambda\in G$ such that
$\dist(\lambda,\sigma(A_0)) > (C\|B\|)^{1/\beta}$.  Then
\[
  \bigl\|B\overline{M(\lambda)}\bigr\| \leq \|B\| \bigl\|\overline{M(\lambda)}\bigr\|
  < \frac{\bigl(\dist(\lambda,\sigma(A_0))\bigr)^\beta}{C}
  \cdot\frac{C}{\bigl(\dist(\lambda,\sigma(A_0))\bigr)^\beta} = 1
\]
implies that $1\in\rho( B \overline{M(\lambda)})$.  It follows from Theorem~\ref{mainthm}
that $A_{[B]}$ is closed with $\lambda\in\rho(A_{[B]})$.
If the condition~(i) together with the second condition in~(ii) is
satisfied then $\rho(A_0) \cap \rho(A_1) \neq \emptyset$ and for each
$\lambda \in \rho(A_0) \cap \rho(A_1)$ we have
$\ran \Gamma_1 = \ran M(\lambda) \subset \ran \overline{M(\lambda)}$;
see~\cite[Proposition~2.6~(iii)]{BL07}.
Hence, for each such $\lambda$ we have $B(\ran \Gamma_1) \subset \ran \Gamma_0$ by~(i),
that is, the first condition of~(ii) is satisfied as well.
\end{proof}

In the special case $G=\rho(A_0)$ and $A_0\geq 0$ with $\mu=0$ in \eqref{nuistgut},
Proposition~\ref{prop:specEst}\,(b) reads as follows.

\begin{corollary}\label{cor:ganzTolleWeylfunktion}
Let the assumptions be as in Proposition~\ref{prop:specEst} and assume,
in addition, that $A_0$ is non-negative and that there exist $\beta\in(0,1]$
and $C>0$ such that
\[
  \bigl\|\overline{M(\lambda)}\bigr\| \le \frac{C}{|\lambda|^\beta}
  \qquad \text{for all}\;\;\lambda\in \rho(A_0).
\]
Then
\[
  \sigma(A_{[B]})\cap \rho(A_0)
  \subset \Bigl\{z \in \rho(A_0): |z| \le \bigl(C\|B\|\bigr)^{1/\beta}\Bigr\}.
\]
\end{corollary}

% *******************************************************************
\section{Sufficient conditions for decay of the Weyl function}
\label{sec:suffconddecay}
% *******************************************************************

\noindent
In this section we consider conditions on the quasi boundary triple
that ensure an asymptotic behaviour of the Weyl function $M$
as required in the results of the previous section. We emphasize
that these results are also new in the settings of ordinary and generalized
boundary triples.
For the next theorem some notation for sectors in the complex plane is needed.
For $z_0 \in \overline{\C^+}$
and $\theta\in\bigl(0,\frac{\pi}{2}\bigr)$ we define the
closed sector $\bbS_{z_0,\theta}$ in $\dC^+$ by
\begin{equation}\label{defSz0th}
  \bbS_{z_0,\theta} \defeq \Bigl\{z\in\CC: z\ne z_0,\,
  \arg(z-z_0)\in\Bigl[\frac{\pi}{2}-\theta,\frac{\pi}{2}+\theta\Bigr]\Bigr\}
  \cup \{z_0\}
\end{equation}
and we denote the corresponding complex conjugate sector in $\dC^-$
by $\bbS_{z_0,\theta}^*$, that is,
\begin{equation*}
  \bbS_{z_0,\theta}^* \defeq \bigl\{z\in\dC:\ov{z}\in\bbS_{z_0,\theta}\bigr\}.
\end{equation*}
Furthermore, for $w_0\in\dR$ and $\nu\in (0,\pi)$ we set
\begin{equation}\label{defUw0nu}
  \mathbb U_{w_0,\nu} \defeq \bigl\{z\in\CC: z\ne w_0,\,
  \arg(z-w_0)\in [\nu,2\pi-\nu]\bigr\}\cup\{w_0\};
\end{equation}
see Figure~\ref{fig:sectors}.

\begin{figure}[ht]
\begin{center}
\begin{tikzpicture}[scale=1.0]
  \draw[fill=lightgray] (-3.5,3.5) -- (-1,1) node[below left] {$z_0$}
  -- (1.5,3.5);
  \draw (-1,1) -- (-1,3.8);
  \draw (0,2) arc (45:135:1.41);
  \draw [->] (-4,0) -- (2.5,0);
  \draw [->] (0,-1.5) -- (0,4);
  \node at (-0.7,1.8) {$\theta$};
  \node at (-1.3,1.8) {$\theta$};
  \node at (-1.7,3) {$\mathbb{S}_{z_0,\theta}$};
\end{tikzpicture}
\hspace*{2ex}
\begin{tikzpicture}[scale=0.8]
  \fill[fill=lightgray] (-4,0) -- (-2,0) -- (1,3) -- (-4,3);
  \fill[fill=lightgray] (-4,0) -- (-2,0) -- (1,-3) -- (-4,-3);
  \draw (1,-3) -- (-2,0)  node[above left] {$w_0$}
  -- (1,3);
  \draw (-0.5,0) arc (0:45:1.5);
  \draw (-0.5,0) arc (360:315:1.5);
  \draw [->] (-4,0) -- (3,0);
  \draw [->] (0,-3.5) -- (0,3.5);
  \node at (-1.0,0.5) {$\nu$};
  \node at (-1.0,-0.5) {$\nu$};
  \node at (-2.5,2) {$\mathbb{U}_{w_0,\nu}$};
\end{tikzpicture}
\end{center}
\caption{The sectors $\mathbb{S}_{z_0,\theta}$ and $\mathbb{U}_{w_0,\nu}$, defined
in \eqref{defSz0th} and \eqref{defUw0nu}, respectively.}
\label{fig:sectors}
\end{figure}
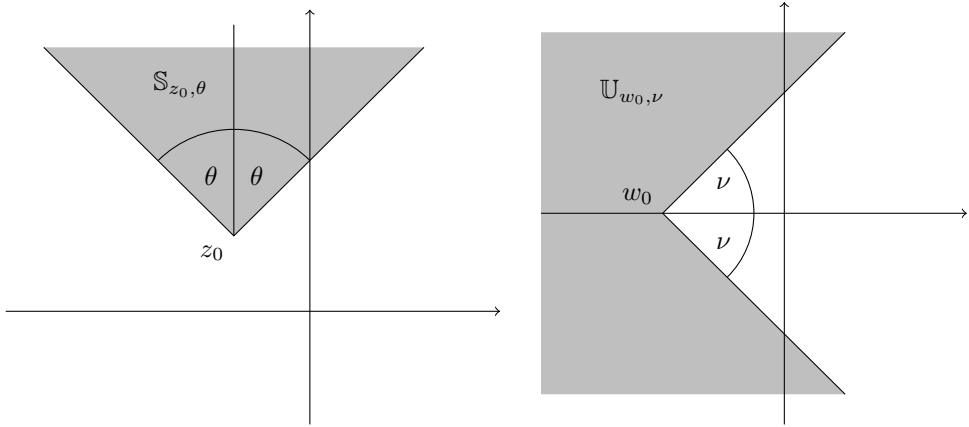

In the proof of the next theorem we need the following fact
from the functional calculus for self-adjoint operators, which
is found, e.g.\ in \cite[Theorem~5.9]{S12}:
for a self-adjoint operator $A$ and measurable functions $\Phi,\Psi:\sigma(A)\to\C$
one has
\begin{equation}\label{funct_calc}
  \ov{\Phi(A)\Psi(A)} = (\Phi\Psi)(A).
\end{equation}
If $\Psi$ is bounded on $\sigma(A)$, then the closure on the left-hand side
is not needed.

\begin{theorem}\label{th:weyl_decay}
Let $S$ be a densely defined, closed, symmetric operator in a Hilbert space $\cH$
and let $\Pi=\{\cG,\Gamma_0,\Gamma_1\}$ be a quasi boundary triple for $T\subset S^*$
with corresponding Weyl function $M$.  Moreover, assume that
\begin{equation}
\label{bnd}
  \Gamma_1|A_0 -\mu|^{-\alpha} : \cH \supset \dom(\Gamma_1|A_0
  -\mu|^{-\alpha}) \to \cG
\end{equation}
is bounded for some $\mu\in\rho(A_0)$ and some $\alpha\in\bigl(0,\frac12\bigr]$.
Then the following assertions hold.
\begin{myenuma}
\item % -----
$M(\lambda)$ is bounded for all $\lambda\in\rho(A_0)$.
\item % -----
For all $z_0 \in \overline{\CC^+} \cap \rho(A_0)$ and
all $\theta\in (0,\frac{\pi}{2})$
there exists $C = C(\Pi, \alpha, \mu, z_0, \theta) > 0$ such that
\begin{equation}\label{weyl_decay}
   \big\|\ov{M(\lambda)}\big\|
   \le \frac{C}{\bigl(\dist(\lambda,\sigma(A_0))\bigr)^{1-2\alpha}}
   \end{equation}
   for all $\lambda\in\bbS_{z_0,\theta}\cup\bbS_{z_0,\theta}^*$.
\item % -----
If $A_0$ is bounded from below, then for all $w_0<\min\sigma(A_0)$
and all $\nu\in (0,\pi)$
there exists $D = D(\Pi, \alpha, \mu, w_0, \nu) > 0$ such that
\begin{equation}\label{weyl_decay2}
   \big\|\ov{M(\lambda)}\big\|
   \le \frac{D}{\bigl(\dist(\lambda,\sigma(A_0))\bigr)^{1-2\alpha}}
\end{equation}
for all $\lambda\in\mathbb U_{w_0,\nu}$.
\end{myenuma}
\end{theorem}

\begin{proof}
Let us first observe that $\Gamma_1|A_0-\mu|^{-\alpha}$ is densely
defined.  Indeed, with the functions $\Phi(t) \defeq (t-\mu)^{-1}$ and
$\Psi(t) \defeq (t-\mu)|t-\mu|^{-\alpha}$ we can use \eqref{funct_calc}
and \eqref{gammastar} to write
\[
  \Gamma_1|A_0-\mu|^{-\alpha}
  = \Gamma_1(\Phi\Psi)(A_0) \supset \Gamma_1\Phi(A_0)\Psi(A_0)
  = \gamma(\ov{\mu})^*\Psi(A_0).
\]
Since $\gamma(\ov{\mu})^* \in \cB(\cH,\cG)$ and $\dom\Psi(A_0)
= \dom|A_0-\mu|^{1-\alpha}$ is dense in $\cH$, it follows
that $\Gamma_1|A_0-\mu|^{-\alpha}$ is densely defined.
By assumption~\eqref{bnd} we therefore have
\begin{equation}\label{eq:bounded}
  \overline{\Gamma_1|A_0-\mu|^{-\alpha}} \in \cB(\cH,\cG).
\end{equation}
Note that also $\gamma(\mu)^* |A_0 - \mu|^{1 - \alpha}$ is
densely defined since $\gamma(\mu)^* \in \cB(\cH, \cG)$
and $|A_0 - \mu|^{1 - \alpha}$ is self-adjoint.
Moreover, set
\[
  \begin{alignedat}{2}
    \Phi_1(t) &\defeq (t-\ov\mu)^{-1}, \qquad &
    \Psi_1(t) &\defeq |t-\mu|^{1-\alpha}, \\
    \Phi_2(t) &\defeq |t-\mu|^{-\alpha}, \qquad &
    \Psi_2(t) &\defeq |t-\mu|(t-\ov\mu)^{-1},
  \end{alignedat}
  \qquad t\in\sigma(A_0),
\]
and note that $\Phi_1\Psi_1=\Phi_2\Psi_2$ and that $\Psi_2$ is bounded.
We obtain from \eqref{gammastar},
\eqref{funct_calc} and \eqref{eq:bounded} that
\begin{align*}
  \gamma(\mu)^*|A_0-\mu|^{1-\alpha}
  &= \Gamma_1(A_0-\ov\mu)^{-1}|A_0-\mu|^{1-\alpha}
  \\[0.5ex]
  &= \Gamma_1\Phi_1(A_0)\Psi_1(A_0)
  \subset \Gamma_1(\Phi_1\Psi_1)(A_0)=\Gamma_1(\Phi_2\Psi_2)(A_0)
  \\[0.5ex]
  &= \Gamma_1\Phi_2(A_0)\Psi_2(A_0)
  \subset \ov{\Gamma_1|A_0-\mu|^{-\alpha}}\Psi_2(A_0)
  \in \cB(\cH, \cG).
\end{align*}
Thus $\gamma(\mu)^*|A_0-\mu|^{1-\alpha}$ is
bounded and densely defined. In particular,
\begin{equation}\label{bounded2}
  |A_0-\mu|^{1-\alpha} \ov{\gamma(\mu)}
  = \big[\gamma(\mu)^*|A_0-\mu|^{1-\alpha}\big]^* \in \cB(\cG,\cH),
\end{equation}
where we have used again that $\gamma(\mu)^* \in \cB(\cH,\cG)$.
Let $\lambda\in\rho(A_0)$ and define the functions
\[
  \begin{alignedat}{2}
    \Phi_3(t) &\defeq \frac{t-\mu}{t-\lambda}|t-\mu|^{\alpha-1}, \qquad &
    \Psi_3(t) &\defeq |t-\mu|^{1-\alpha}, \\
    \Phi_4(t) &\defeq |t-\mu|^{-\alpha}, \qquad &
    \Psi_4(t) &\defeq \frac{t-\mu}{t-\lambda}|t-\mu|^{2\alpha-1},
  \end{alignedat}
  \qquad t\in\sigma(A_0),
\]
which satisfy $\Phi_3=\Phi_4\Psi_4$. The functions $\Phi_3$, $\Phi_4$ and $\Psi_4$
are bounded on $\sigma(A_0)$ and
$\ran\ov{\gamma(\mu)}\subset\dom\Psi_3(A_0)$ by \eqref{bounded2}.
Hence for each $g \in \dom M(\lambda) = \ran\Gamma_0$ we have
(where we use \eqref{gammalm} in the second equality)
\begin{equation}\label{factorM}
\begin{split}
  M(\lambda)g &= \Gamma_1\ov{\gamma(\lambda)}g
  = \Gamma_1(A_0-\mu)(A_0-\lambda)^{-1}\ov{\gamma(\mu)}g
    \\[0.5ex]
  &= \Gamma_1(\Phi_3\Psi_3)(A_0)\ov{\gamma(\mu)}g
  = \Gamma_1\ov{\Phi_3(A_0)\Psi_3(A_0)}\;\ov{\gamma(\mu)}g
    \\[0.5ex]
  &= \Gamma_1\Phi_3(A_0)\Psi_3(A_0)\ov{\gamma(\mu)}g
    \\[0.5ex]
  &= \Gamma_1\Phi_4(A_0)\Psi_4(A_0)\Psi_3(A_0)\ov{\gamma(\mu)}g
    \\[0.5ex]
  &= \bigl[\,\ov{\Gamma_1|A_0-\mu|^{-\alpha}}\,\bigr]\Psi_4(A_0)
  \bigl[|A_0-\mu|^{1-\alpha}\ov{\gamma(\mu)}\,\bigr]g.
\end{split}
\end{equation}
According to \eqref{eq:bounded} and \eqref{bounded2} the terms in the square brackets
are bounded and everywhere defined operators, which are independent
of $\lambda$.  Since $\Psi_4$ is bounded on $\sigma(A_0)$, it follows
that $M(\lambda)$ is a bounded, densely defined operator,
and assertion (a) is proved.

Relations \eqref{factorM} and \eqref{bounded2} imply that
\[
  \big\|\ov{M(\lambda)}\big\|
  \le \big\|\ov{\Gamma_1|A_0-\mu|^{-\alpha}}\big\|^2\|\Psi_4(A_0)\|.
\]
Assertions (b) and (c) follow from suitable estimates of $\|\Psi_4(A_0)\|$.
Let $E$ be the spectral measure for the operator $A_0$.
For all $\lambda\in\rho(A_0)$ and all $f\in\cH$ we have
\begin{equation}\label{integral101}
\begin{split}
  \|\Psi_4(A_0)f\|^2
  &= \int_{\sigma(A_0)} \frac{|t-\mu|^{4\alpha}}{|t-\lambda|^2}\,\drm\bigl(E(t)f,f\bigr)
    \\[0.5ex]
  &= \int_{\sigma(A_0)} \frac{|t-\mu|^{4\alpha}}{|t-\lambda|^{4\alpha}}
  \cdot\frac{1}{|t-\lambda|^{2-4\alpha}}\,\drm\bigl(E(t)f,f\bigr)
    \\[0.5ex]
  &\le \frac{1}{\bigl(\dist(\lambda,\sigma(A_0))\bigr)^{2-4\alpha}}
  \int_{\sigma(A_0)} \frac{|t-\mu|^{4\alpha}}{|t-\lambda|^{4\alpha}}
  \,\drm\bigl(E(t)f,f\bigr).
\end{split}
\end{equation}
In order to prove (b), fix $z_0\in\ov{\C^+}\cap\rho(A_0)$ and $\theta\in(0,\pi/2)$.
It remains to estimate the integrand of the last integral in \eqref{integral101}
uniformly in $\lambda\in\bbS_{z_0,\theta}$ and $t\in\sigma(A_0)$.
To this end set $d_{z_0,\theta} \defeq \dist(\bbS_{z_0,\theta},\sigma(A_0))>0$.
Let $\lambda\in\bbS_{z_0,\theta}$, i.e.\
\[
  \Im\lambda \ge \Im z_0 \qquad\text{and}\qquad
  |\Re(\lambda-z_0)| \le \tan\theta \cdot \Im(\lambda-z_0).
\]
If $\lambda\ne z_0$, then
\begin{align*}
  \frac{|t-\mu|^2}{|t-\lambda|^2}
  &= \frac{(t-\Re\mu)^2+(\Im\mu)^2}{|t-\lambda|^2}
    \\
  &\le \frac{3\bigl[(t-\Re\lambda)^2+(\Re\lambda-\Re z_0)^2+(\Re z_0-\Re\mu)^2\bigr]
  +(\Im\mu)^2}{|t-\lambda|^2}
    \\
  &\le 3 + 3\frac{\bigl(\Re(\lambda-z_0)\bigr)^2}{\bigl(\Im(\lambda-z_0)\bigr)^2}
  + \frac{3\bigl(\Re(z_0-\mu)\bigr)^2+(\Im\mu)^2}{d_{z_0,\theta}^2}
    \\
  &\le 3 + 3\tan^2\theta
  + \frac{3\bigl(\Re(z_0-\mu)\bigr)^2+(\Im\mu)^2}{d_{z_0,\theta}^2}\,,
\end{align*}
where the right-hand side is independent of $\lambda$ and $t$; by continuity this
estimate extends to $\lambda=z_0$.
The case $\lambda\in\bbS_{z_0,\theta}^*$ can be treated analogously.
From this, together with~\eqref{factorM} and~\eqref{integral101}, the claim of (b) follows.

To prove (c), let $w_0 < \min\sigma(A_0)$ and $\nu\in(0,\pi)$;
note that $\dist(\bbU_{w_0,\nu},\sigma(A_0)) > 0$.
Let first $\lambda \in \C$ with $\real \lambda < w_0$.
Then with $m \defeq \min\sigma(A_0)$ the integrand of the last integral
in \eqref{integral101} can be estimated using
\begin{align*}
  \frac{|t-\mu|^2}{|t-\lambda|^2}
  &\le \frac{3\big[(t-\Re\lambda)^2+(\Re\lambda-m)^2+(m-\Re\mu)^2\big]
  +(\Im\mu)^2}{(t-\Re\lambda)^2+(\Im\lambda)^2}
    \\
  &\le 3 + 3 + \frac{3(m-\Re\mu)^2+(\Im\mu)^2}{(m-w_0)^2}\,,
\end{align*}
where we have used $t-\Re\lambda \ge m-\Re\lambda \ge m-w_0 > 0$.
If $\nu \ge \pi/2$, this and \eqref{integral101} lead to a uniform estimate
of $\Psi_4(A_0)$ in $\dU_{w_0, \nu}$.  If $\nu \in (0, \pi/2)$, then
\begin{align*}
  \bbU_{w_0,\nu} = \{z\in\C : \Re z < w_0\} \cup \bbS_{w_0,\theta,}
  \cup \bbS_{w_0,\theta,}^*
\end{align*}
with $\theta = \pi/2 - \nu$, and a uniform estimate of the last integral
in \eqref{integral101} for $\lambda \in \bbU_{w_0,\nu}$ follows from the
previous consideration and item (b).  The proof is complete.
\end{proof}

\begin{remark}
Suppose that the assumptions of Theorem~\ref{th:weyl_decay} are satisfied
for $\alpha=\frac12$.  It follows from Theorem~\ref{th:weyl_decay} that $M(\lambda)$
is bounded for every $\lambda\in\rho(A_0)$ and that $\|\ov{M(\lambda)}\|$ is uniformly
bounded on each sector $\bbS_{z_0,\theta}$ as in the theorem.
In addition, we can show (see below) that for each $\bbS_{z_0,\theta}$ as in the theorem,
\begin{equation}\label{Mto0strongly}
  \ov{M(\lambda)}g \to 0 \qquad
  \text{as} \;\; \lambda\to\infty \;\; \text{in} \;\; \bbS_{z_0,\theta}, \;\;
  g\in\cG.
\end{equation}
Similarly, if $A_0$ is bounded from below, then $M(\lambda)$
is bounded for every $\lambda\in\rho(A_0)$ and $\|\ov{M(\lambda)}\|$ is uniformly
bounded on each sector $\bbU_{w_0,\nu}$ as in the theorem, and for
each such $\bbU_{w_0,\nu}$,
\begin{equation}\label{kempinski}
  \ov{M(\lambda)}g \to 0 \qquad
  \text{as} \;\; \lambda\to\infty \;\; \text{in} \;\; \bbU_{w_0,\nu}, \;\;
  g\in\cG.
\end{equation}
To prove \eqref{Mto0strongly} set
\[
  f \defeq |A_0-\mu|^{1/2}\ov{\gamma(\mu)}g
\]
and observe that by \eqref{factorM} it is sufficient to show that
\[
  \|\Psi_4(A_0)f\|^2 = \int_{\sigma(A_0)} \frac{|t-\mu|^2}{|t-\lambda|^2}\,
  \drm\bigl(E(t)f,f\bigr)\to 0 \qquad
  \text{as} \;\; \lambda\to\infty \;\; \text{in} \;\; \bbS_{z_0,\theta}.
\]
It was shown in the proof of Theorem~\ref{th:weyl_decay} that the integrand
is uniformly bounded for $\lambda\in\bbS_{z_0,\theta}$ and $t\in\sigma(A_0)$.
Moreover, the measure $(E(\cdot)f,f)$ is finite and the integrand converges
to $0$ as $\lambda\to\infty$ for each fixed $t\in\sigma(A_0)$.
Hence the dominated convergence theorem implies that $\|\Psi_4(A_0)f\|\to0$
as $\lambda\to\infty$ in $\bbS_{z_0,\theta}$, which proves \eqref{Mto0strongly}.
The same argument also shows \eqref{kempinski}.
\end{remark}

\begin{corollary}
Let $\Pi=\{\cG,\Gamma_0,\Gamma_1\}$ be a quasi boundary triple for $T\subset S^*$
with corresponding $\gamma$-field $\gamma$ and Weyl function $M$ and
assume that the operator in \eqref{bnd}
is bounded for some $\mu\in\rho(A_0)$ and some $\alpha\in\bigl(0,\frac12\bigr]$.
Then the following assertions hold.
\begin{myenuma}
\item % -----
For all $z_0 \in \overline{\CC^+} \cap \rho(A_0)$ and
all $\theta\in (0,\frac{\pi}{2})$
there exist $C_1=C_1(\Pi,\alpha,\mu,z_0,\theta)$ and $C_2=C_2(\Pi,\alpha,\mu,z_0,\theta)$
such that
\begin{align}
   \big\|\ov{\gamma(\lambda)}\big\|
   &\le \frac{C_1}{\bigl(\dist(\lambda,\sigma(A_0))\bigr)^{1-\alpha}}\,,
     \label{gamma_decay}\\[1ex]
   \bigl\|\ov{M^{(n)}}(\lambda)\bigr\|
   &\le \frac{C_2\,n!}{\bigl(\dist(\lambda,\sigma(A_0))\bigr)^{n+1-2\alpha}}
     \label{weyl_der_decay}
\end{align}
for all $\lambda\in\bbS_{z_0,\theta}\cup\bbS_{z_0,\theta}^*$.
\item % -----
If $A_0$ is bounded from below, then for all $w_0<\min\sigma(A_0)$
and all $\nu\in (0,\pi)$
there exist $D_1=D_1(\Pi,\alpha,\mu,w_0,\nu)$ and $D_2=D_2(\Pi,\alpha,\mu,w_0,\nu)$ such that
\begin{align}
   \big\|\ov{\gamma(\lambda)}\big\|
   &\le \frac{D_1}{\bigl(\dist(\lambda,\sigma(A_0))\bigr)^{1-\alpha}}\,,
     \label{gamma_decay2}\\[1ex]
   \bigl\|\ov{M^{(n)}}(\lambda)\bigr\|
   &\le \frac{D_2\,n!}{\bigl(\dist(\lambda,\sigma(A_0))\bigr)^{n+1-2\alpha}}
     \label{weyl_der_decay2}
\end{align}
for all $\lambda\in\mathbb U_{w_0,\nu}$.
\end{myenuma}
\end{corollary}

\begin{proof}
(a)
First we prove \eqref{gamma_decay}.
Let $z_0\in\ov{\CC^+}\cap\rho(A_0)$ and $\theta\in\bigl(0,\frac{\pi}{2}\bigr)$.
For $\lambda\in\bbS_{z_0,\theta}$ with $\Im\lambda\ge1$ we have
\begin{align*}
  \dist(\lambda,\sigma(A_0))
  &\le |\lambda-z_0|+\dist\bigl(z_0,\sigma(A_0)\bigr) \\[0.5ex]
  &\le \frac{\Im\lambda}{\cos\theta}+\dist(z_0,\sigma(A_0)) \\[0.5ex]
  &\le \biggl(\frac{1}{\cos\theta}+\dist(z_0,\sigma(A_0))\biggr)\Im\lambda.
\end{align*}
This, \eqref{eq:gammaWunder}, \eqref{normgamma} and \eqref{weyl_decay} imply that
\begin{align*}
  \bigl\|\ov{\gamma(\lambda)}\bigr\|
  &= \frac{\bigl\|\ov{\Im M(\lambda)}\bigr\|^{1/2}}{(\Im\lambda)^{1/2}}
  \le \frac{\bigl\|\ov{M(\lambda)}\bigr\|^{1/2}}{(\Im\lambda)^{1/2}}
    \\[0.5ex]
  &\le \frac{C^{1/2}\bigl(\dist(\lambda,\sigma(A_0))\bigr)^{1/2}}{(\Im\lambda)^{1/2}\bigl(\dist(\lambda,\sigma(A_0))\bigr)^{1-\alpha}}
  \le \frac{C^{1/2}\bigl[\frac{1}{\cos\theta}+\dist(z_0,\sigma(A_0))\bigr]^{1/2}}{\bigl(\dist(\lambda,\sigma(A_0))\bigr)^{1-\alpha}}
\end{align*}
for $\lambda\in\bbS_{z_0,\theta}$ with $\Im\lambda\ge1$.
Since $\|\ov{\gamma(\ov\lambda)}\|=\|\ov{\gamma(\lambda)}\|$, see~\eqref{eq:gammaWunder},
and $\gamma$ is bounded on the
set $\{z\in\bbS_{z_0,\theta}\cup\bbS_{z_0,\theta}^*:|\Im z|\le1\}$,
the inequality~\eqref{gamma_decay} is proved.

The inequality in \eqref{weyl_der_decay} is obtained from \eqref{gamma_decay}
and \eqref{der_M} as follows:
\begin{align*}
  \bigl\|\ov{M^{(n)}}(\lambda)\bigr\|
  &\le n!\bigl\|\gamma(\ov\lambda)^*\bigr\|\,\bigl\|(A_0-\lambda)^{-(n-1)}\bigr\|\,
  \bigl\|\ov{\gamma(\lambda)}\bigr\| \\[1ex]
  &\le \frac{n!\,C_1^2}{\bigl(\dist(\lambda,\sigma(A_0))\bigr)^{1-\alpha+n-1+1-\alpha}}\,.
\end{align*}

(b)
Now assume that $A_0$ is bounded from below and set $m\defeq\min\sigma(A_0)$.
Let $w_0<m$ and, without loss of generality, $\nu\in\bigl(0,\frac{\pi}{2}\bigr)$.
Let $x\in\cG$ and $u\in\cH$ and define the function
\[
  f(z) \defeq (m-z)^{1-\alpha}\bigl(\ov{\gamma(z)}x,u\bigr),
  \qquad z\in\CC \;\; \text{with} \;\; \Re z\le w_0,
\]
where the function $\zeta\mapsto\zeta^{1-\alpha}$ is defined with a cut on
the negative half-line.
The already proved item (a) implies that \eqref{gamma_decay} is valid
for $z\in\bbS_{w_0,\theta}$ with $\theta\defeq \frac{\pi}{2}-\nu$ and
some $D_1>0$.  In particular, it is true for $z\in\CC$ with $\Re z=w_0$,
which yields that
\[
  |f(z)| \le |m-z|^{1-\alpha}\bigl\|\ov{\gamma(z)}\bigr\|\,\|x\|\,\|u\|
  \le |m-z|^{1-\alpha}\frac{D_1\|x\|\,\|u\|}{\bigl(\dist(z,\sigma(A_0))\bigr)^{1-\alpha}}
  = D_1\|x\|\,\|u\|
\]
for all $z\in\CC$ with $\Re z=w_0$.
Since by \eqref{gammalamu} the function $f$ grows at most like a power of $z$ on the
half-plane $\{z\in\CC:\Re z\le w_0\}$, the Phragm\'en--Lindel\"of principle
(see, e.g.\ \cite[Corollary VI.4.2]{ConwayBuch})
implies that
\[
  |f(z)| \le D_1\|x\|\,\|u\| \qquad
  \text{for all} \;\; z\in\CC \;\; \text{with} \;\; \Re z\le w_0.
\]
It follows from this that
\[
  \bigl\|\ov{\gamma(z)}\bigr\| \le \frac{D_1}{|m-z|^{1-\alpha}}
  \qquad \text{for all} \;\; z\in\CC \;\; \text{with} \;\; \Re z\le w_0.
\]
If we combine this with \eqref{gamma_decay} with $z_0=w_0$ and $\theta=\frac{\pi}{2}-\nu$,
we obtain \eqref{gamma_decay2}.
The estimate \eqref{weyl_der_decay2} follows from \eqref{gamma_decay2} in the
same way as in (a).
\end{proof}

The following example shows that Theorem~\ref{th:weyl_decay} is sharp in a certain sense.

\begin{example}\label{exampleoptimal}
Let $\alpha\in\bigl(0,\frac12\bigr]$ and
let $\mu$ be the Borel measure on $\RR$ that has support $[e,\infty)$,
is absolutely continuous and has density
\[
  \frac{\drm\mu(t)}{\drm t} = \frac{1}{t^{1-2\alpha}(\ln t)^2}\,, \qquad t\in[e,\infty).
\]
Moreover, define
\[
  M(\lambda) \defeq \int_e^\infty \frac{1}{t-\lambda}\,\drm\mu(t),
  \qquad \lambda\in\CC\setminus[e,\infty).
\]
This function is the Weyl function of the following ordinary boundary triple
\begin{align*}
  & \cH = L^2(\mu), \quad \cG = \CC, \\[1ex]
  & \dom T = \bigl\{f\in\cH: \exists\,c_f\in\CC
  \;\;\text{such that}\;\; tf(t)-c_f\in\cH\bigr\}, \\[1ex]
  & (Tf)(t) = tf(t)-c_f, \\[1ex]
  & \Gamma_0f = c_f, \qquad \Gamma_1f = \int_e^\infty f(t)\,\drm\mu(t);
\end{align*}
note that $c_f$ is uniquely determined by $f$ since the measure $\mu$ is infinite.
The operator $A_0$ is the multiplication operator by the independent variable.
The mapping in \eqref{bnd} with $\mu=0$ is bounded since for $f\in\cH$ with
compact support we have
\[
  \Gamma_1 A_0^{-\alpha}f = \int_e^\infty f(t)\,t^{-\alpha}\drm\mu(t)
  \le \|f\|_{\cH}\biggl[\int_e^\infty \frac{1}{t^{2\alpha}t^{1-2\alpha}(\ln t)^2}\,\drm t\biggr]^{1/2}
\]
and the last integral converges.
Hence Theorem~\ref{th:weyl_decay} yields that
\[
  M(\lambda) = \rmO\biggl(\frac{1}{|\lambda|^{1-2\alpha}}\biggr), \qquad \lambda\to-\infty.
\]
One can show that the actual asymptotic behaviour of $M$ is
\[
  M(\lambda) \sim \frac{C}{|\lambda|^{1-2\alpha}(\ln|\lambda|)^2}, \qquad \lambda\to-\infty,
\]
with a positive constant $C$.

Hence, apart from the logarithmic factor, Theorem~\ref{th:weyl_decay}
yields the correct asymptotic behaviour.
Using Krein's inverse spectral theorem (see, e.g.\ \cite{KK68})
one can rewrite this example as a Krein--Feller operator: $-D_m D_x$
with some mass distribution $m$ so that the measure $\mu$ becomes the
principal spectral measure of the string.
\end{example}

The next corollary is an immediate consequence of Theorem~\ref{th:weyl_decay}.

% -------------------------------------------------------------------
\begin{corollary}
Let $\{\cG,\Gamma_0,\Gamma_1\}$ be a quasi boundary triple for $T\subset S^*$
with corresponding Weyl function $M$ and assume that the operator in \eqref{bnd}
is bounded for some $\alpha\in\bigl(0,\frac12\bigr]$ and some $\mu\in\rho(A_0)$.
Then $M$ satisfies
\begin{equation}\label{kac}
  \int_1^\infty \frac{\bigl\|\Im\ov{M(iy)}\bigr\|}{y^\gamma} \drm y < \infty
\end{equation}
for every $\gamma>2\alpha$.
\end{corollary}

Condition \eqref{kac} says that the function $M$ belongs to
the \emph{Kac class} $\mathbf{N}_\gamma$ (see, e.g.\ \cite{KK74} for the scalar case).
Assume that $M$ satisfies \eqref{kac} for some $\gamma\in(0,2)$
and consider the integral representation
\[
  M(\lambda) = A+\lambda B+\int_{\RR}\biggl(\frac{1}{t-\lambda}-\frac{t}{1+t^2}\biggr)
  \drm\Sigma(t),
\]
where $A$ and $B\ge0$ are bounded symmetric operators and $\Sigma$ is
an operator-valued measure (see, e.g.\ \cite{MM03} or \cite[\S 3.4]{BL10}).
Often the measure $\Sigma$ plays the role of a spectral measure.
For each $\varphi\in\ran\Gamma_0$ we have
\[
  (M(\lambda)\varphi,\varphi) = (A\varphi,\varphi) + \lambda(B\varphi,\varphi)
  + \int_{\RR} \biggl(\frac{1}{t-\lambda}-\frac{t}{1+t^2}\biggr)
  \drm\bigl(\Sigma(t)\varphi,\varphi\bigr).
\]
It follows from \cite[Lemma~3.1]{W00} and its proof that $(B\varphi,\varphi)=0$
and that
\[
  \int_{\RR} \frac{1}{1+|t|^\gamma}\,\drm\bigl(\Sigma(t)\varphi,\varphi\bigr)
  \le C\|\varphi\|^2,
\]
with some $C>0$, which does not depend on $\varphi$.
Hence $B=0$ and
\[
  \int_{\RR} \frac{1}{1+|t|^\gamma}\,\drm\Sigma(t)
\]
is a bounded operator.

% *******************************************************************
\section{Elliptic operators with non-local Robin boundary conditions}
\label{sec:Robin}
% *******************************************************************

\noindent
In this section we apply the results of the previous sections to
elliptic differential operators on domains whose boundaries
are not necessarily compact.
Our main focus is on operators subject to non-self-adjoint boundary conditions.
For some recent investigations of non-self-adjoint elliptic operators
we refer the reader to \cite{BGW09,BMNW08,GLMZ05,G08,M10}.

Let $\Omega\subset\RR^n$ be a domain that is \emph{uniformly regular}\footnote{This means
that $\partial\Omega$ is $C^\infty$-smooth and that
there exists a covering of $\partial\Omega$ by open sets $\Omega_j$, $j\in\dN$, and $n_0\in\dN$
such that at most $n_0$ of the $\Omega_j$ have a non-empty intersection,
and a family of $C^\infty$-homeomorphisms
\[
  \varphi_j:\Omega_j\cap\Omega\rightarrow \cB_1\cap\{x_n>0\}, \qquad
  \text{where}\quad \cB_r=\{x\in\RR^n:\Vert x\Vert < r\},
\]
such that $\varphi_j:\Omega_j\cap\partial\Omega\rightarrow \cB_1\cap\{x_n=0\}$,
the derivatives of $\varphi_j$, $j\in\dN$, and their inverses are uniformly bounded, and
$\bigcup_j\varphi^{-1}_j(\cB_{1/2})$ covers a uniform neighbourhood of $\partial\Omega$.}
in the sense of \cite[p.~366]{B59} and \cite[page~72]{F67}; see also~\cite{B65,B61}.
This includes, e.g.\ domains with compact $C^\infty$-smooth boundaries or compact,
smooth perturbations of half-spaces. Moreover, the class of uniformly regular
unbounded domains includes certain quasi-conical
and quasi-cylindrical domains in the sense
of \cite[Definition~X.6.1]{EE87}.
Non-self-adjoint elliptic operators with Robin boundary conditions
on such domains have been investigated recently in connection with
non-Hermitian quantum waveguides and layers;
see, e.g.\ \cite{BK08, BK12, BZ17, LS17}.
Further, let
\begin{equation}\label{eq:diffexpr}
  \cL = - \sum_{j,k=1}^n \frac{\partial}{\partial x_j}a_{jk}\frac{\partial}{\partial x_k} + a
\end{equation}
be a differential expression on $\Omega$, where
we assume that $a_{j k} \in C^\infty (\overline \Omega)$ are bounded,
have bounded, uniformly continuous derivatives on $\overline{\Omega}$
and satisfy $a_{jk}(x)=\ov{a_{kj}(x)}$ for all $x\in\ov{\Omega}$, $1\le j,k\le n$,
and that $a \in L^\infty (\Omega)$ is real-valued; cf.\ \cite[(S1)--(S5) in Chapter~4]{B65}.
Moreover, we assume that $\cL$ is uniformly elliptic, i.e.\ there exists $E > 0$ such that
\[
  \sum_{j,k=1}^n a_{jk}(x) \xi_j\xi_k\geq E \sum_{k=1}^n\xi_k^2,
  \qquad \xi=(\xi_1,\dots\xi_n)^\top \in \dR^n, x \in \overline \Omega.
\]

In the following we denote by $H^s(\Omega)$ and $H^s(\partial\Omega)$
the Sobolev spaces of order $s \geq 0$ on $\Omega$ and $\partial\Omega$, respectively.
For $f \in C_0^\infty(\overline\Omega)$, where $C_0^\infty(\ov\Omega)$ denotes
the set of $C^\infty(\ov\Omega)$-functions with compact support, let
\[
  \frac{\partial f}{\partial \nu_\cL} \Big|_{\partial\Omega}
  \defeq \sum_{j, k = 1}^n a_{jk} \nu_j \frac{\partial f}{\partial x_k}\Big|_{\partial\Omega}
\]
denote the conormal derivative of $f$ at $\partial\Omega$ with respect to~$\cL$,
where $\nu = (\nu_1, \dots, \nu_n)^\top$ is the unit normal vector field
at $\partial\Omega$ pointing outwards.  Then Green's identity
\begin{equation}\label{eq:Green}
  (\cL f, g) - (f, \cL g)
  = \biggl(f|_{\partial\Omega}, \frac{\partial g}{\partial \nu_\cL}\Big|_{\partial\Omega}\biggr)
  - \biggl(\frac{\partial f}{\partial \nu_\cL}\Big|_{\partial\Omega}, g|_{\partial\Omega}\biggr)
\end{equation}
holds for all $f,g \in C_0^\infty(\ov\Omega)$, where the inner products are
in $L^2(\Omega)$ and $L^2(\partial\Omega)$, respectively.
Recall that the pair of mappings
\[
  C_0^\infty(\overline{\Omega}) \ni f \mapsto
  \biggl\{f|_{\partial\Omega};\; \frac{\partial f}{\partial \nu_\cL}\Big|_{\partial\Omega}\biggr\}
  \in H^{3/2}(\partial\Omega) \times H^{1/2}(\partial\Omega)
\]
extends by continuity to a bounded map from $H^2(\Omega)$
onto $H^{3/2}(\partial\Omega) \times H^{1/2}(\partial\Omega)$;
see, e.g.\ \cite[Theorem~3.9]{F67}.  The extended trace and conormal derivative
are again denoted by $f|_{\partial\Omega}$ and
$\frac{\partial f}{\partial \nu_\cL}\big|_{\partial\Omega}$, respectively.
Moreover, Green's identity~\eqref{eq:Green} extends to all $f,g \in H^2(\Omega)$;
see \cite[Theorem~4.4]{F67}.

In order to construct a quasi boundary triple, let us define the
operators $S$ and $T$ in $L^2(\Omega)$ via
\begin{equation}\label{eq:Selliptic}
  S f = \cL f, \qquad \dom S = \biggl\{ f \in H^2(\Omega) :
  f |_{\partial\Omega} = \frac{\partial f}{\partial \nu_\cL}\Big|_{\partial\Omega} = 0 \biggr\},
\end{equation}
and
\begin{equation}\label{eq:Telliptic}
  T f = \cL f, \qquad \dom T = H^2(\Omega).
\end{equation}
Moreover, we define boundary mappings $\Gamma_0,\Gamma_1: \dom T \to L^2(\partial\Omega)$ by
\[
  \Gamma_0 f = \frac{\partial f}{\partial\nu_\cL}\Big|_{\partial\Omega}, \qquad
  \Gamma_1 f = f|_{\partial\Omega} \qquad \text{for}\;f \in \dom T.
\]

The assertions of the following proposition can be found in
\cite[Propositions~3.1 and~3.2]{BLLR17}.

\begin{proposition}\label{prop:QBTelliptic}
The operator $S$ in~\eqref{eq:Selliptic} is closed, symmetric and densely defined
with $\overline{T} = S^*$ for $T$ in \eqref{eq:Telliptic},
and the triple $\{L^2(\partial\Omega), \Gamma_0, \Gamma_1\}$
is a quasi boundary triple for $T\subset S^*$ with the following properties.
\begin{myenum}
\item % ----------
$\ran(\Gamma_0, \Gamma_1)^{\top} = H^{1/2}(\partial\Omega) \times H^{3/2}(\partial\Omega)$.
\item % ----------
$A_0$ is the \emph{Neumann operator}
\begin{align*}
  A_{\rm N} f = \cL f, \qquad \dom A_{\rm N} = \left\{ f \in H^2(\Omega) :
  \frac{\partial f}{\partial \nu_\cL}\Big|_{\partial\Omega} = 0 \right\},
\end{align*}
and $A_1$ is the \emph{Dirichlet operator}
\begin{align*}
  A_{\rm D} f = \cL f, \qquad \dom A_{\rm D} = \left\{ f \in H^2(\Omega) :
  f |_{\partial\Omega} = 0 \right\}.
\end{align*}
Both operators, $\AN$ and $\AD$, are self-adjoint and bounded from below.
\item % ----------
For $\lambda \in \rho(A_{\rm N})$, the associated $\gamma$-field satisfies
\begin{equation}\label{eq:NDmap}
  \gamma(\lambda) \frac{\partial f}{\partial \nu_\cL} \Big|_{\partial\Omega}
  = f \qquad \text{for all}\;\; f \in \ker(T - \lambda),
\end{equation}
and the associated Weyl function is given by the \emph{Neumann-to-Dirichlet map},
\begin{equation}\label{eq:NDmap2}
  M(\lambda)\frac{\partial f}{\partial \nu_\cL} \Big|_{\partial\Omega}
  = f|_{\partial\Omega} \qquad \text{for all}\;\; f \in \ker(T - \lambda).
\end{equation}
Moreover, $M(\lambda)$ is a bounded, non-closed operator in $L^2(\partial\Omega)$
with domain $H^{1/2}(\partial\Omega)$ such that
$\ran\ov{M(\lambda)} \subset H^1(\partial\Omega)$.
\end{myenum}
\end{proposition}

\medskip

\noindent
In order to apply the results of Section~\ref{sec:consequences} to the quasi
boundary triple in Proposition~\ref{prop:QBTelliptic} we prove estimates
for the Weyl function in certain sectors using Theorem~\ref{th:weyl_decay}.

\begin{lemma}\label{lem:NDdecay}
Let $\bbU_{w_0,\nu}$ be defined as in \eqref{defUw0nu}. Then for
each $w_0 < \min\sigma(A_{\rm N})$, $\nu \in (0,\pi)$ and $\beta \in (0,\frac{1}{2})$
there exists $C = C(\cL, \Omega, w_0, \nu, \beta) > 0$ such that
\begin{equation}\label{powerdecayell}
  \big\|\ov{M(\lambda)}\big\| \le \frac{C}{\bigl(\dist(\lambda,\sigma(A_{\rm N}))\bigr)^{\beta}}
  \qquad\text{for all}\;\;\lambda \in \bbU_{w_0,\nu}.
\end{equation}
\end{lemma}

\begin{proof}
Let $\mu= \min \sigma(A_{\rm N})-1$.  Then $A_{\rm N}-\mu$ is a positive,
self-adjoint operator in $L^2(\Omega)$ and $\Lambda \defeq (A_{\rm N}-\mu)^{1/2}$
in $L^2(\Omega)$ is well defined, self-adjoint and positive.
It can be seen with the help of the quadratic form associated
with $A_{\rm N}$ that $\dom \Lambda = H^1(\Omega)$ and that
the $H^1(\Omega)$-norm is equivalent to the graph norm $\|\Lambda \cdot\|_{L^2(\Omega)}$.
Thus the identity operator provides an isomorphism between $H^1(\Omega)$
and $(\dom \Lambda, \|\Lambda\cdot\|_{L^2(\Omega)})$ as well as, trivially,
between $L^2(\Omega)$ and $(\dom \Lambda^0, \| \Lambda^0 \cdot \|_{L^2(\Omega)})$.
By interpolation (see, e.g.\ \cite[Theorems~5.1 and 7.7]{LM72}),
the identity operator is also an isomorphism
between $H^s(\Omega)$ and $(\dom \Lambda^s, \|\Lambda^s\cdot\|_{L^2(\Omega)})$
for each $s \in (0, 1)$.
In particular, $\dom(A_{\rm N}-\mu)^{s/2} = \dom \Lambda^s = H^s(\Omega)$
for each $s \in (0, 1)$. It follows from the closed graph theorem that
$(A_{\rm N} - \mu)^{-s/2}$ is bounded as an operator from $L^2(\Omega)$ to $H^s(\Omega)$
for each such $s$.  Since the trace map is bounded from $H^s(\Omega)$
to $L^2(\partial\Omega)$ for each $s \in (\frac{1}{2}, 1)$
by \cite[Theorem~3.7]{F67},
it follows that $f \mapsto ((A_{\rm N} - \mu)^{-s/2} f) |_{\partial\Omega}$
is bounded from $L^2(\Omega)$ to $L^2(\partial\Omega)$ for each $s \in (\frac{1}{2}, 1)$.
In particular, the operator
\begin{equation}\label{eq:prodOp}
  \Gamma_1(A_{\rm N}-\mu)^{-\alpha}: L^2(\Omega)
  \supset \dom\bigl(\Gamma_1(A_{\rm N}-\mu)^{-\alpha}\bigr) \to L^2(\partial\Omega)
\end{equation}
is bounded for each $\alpha \in (\frac{1}{4}, \frac{1}{2})$. By Theorem~\ref{th:weyl_decay}
for each $w_0 < \min \sigma(A_{\rm N})$, each $\nu \in (0, \pi)$
and each $\alpha \in (\frac{1}{4}, \frac{1}{2})$ there
exists $C = C(\cL, \Omega,w_0,\nu,\alpha) > 0$ such that
\[
  \big\|\overline{M(\lambda)}\big\|
  \leq \frac{C}{\bigl(\dist(\lambda,\sigma(A_{\rm N}))\bigr)^{1-2\alpha}}
\]
holds for all $\lambda \in \bbU_{w_0, \nu}$.
From this the claim of the lemma follows.
\end{proof}

\begin{remark}\label{re:NDdecay}
Along the negative real axis the result of Lemma~\ref{lem:NDdecay}
can be slightly improved.  It was proved in \cite[Proposition~3.2\,(iv)]{BLLR17}
(using techniques from \cite{A62}) that for each $\mu < \min\sigma(\AN)$
there exists $C = C(\cL, \Omega, \mu)$ such that
\begin{equation}\label{eq:smallAgmon}
  \big\|\overline{M(\lambda)}\big\| \le \frac{C}{(\mu-\lambda)^{1/2}} \qquad
  \text{for all}\;\; \lambda < \mu.
\end{equation}
\end{remark}

In the next theorem we apply Lemma~\ref{lem:NDdecay}, Remark~\ref{re:NDdecay}
and the results from Section~\ref{sec:consequences} to obtain m-sectorial
(self-adjoint, maximal dissipative, maximal accumulative) realizations of $\cL$
subject to generalized Robin boundary conditions and also spectral enclosures
for these realizations.

\begin{theorem}\label{thm:ABelliptic}
Let $B$ be a closable operator in $L^2(\partial\Omega)$ such that
\begin{equation}\label{assumpBelliptic}
  H^{1/2}(\partial\Omega) \subset \dom B \qquad\text{and}\qquad
  B\bigl(H^1(\partial\Omega)\bigr) \subset H^{1/2}(\partial\Omega).
\end{equation}
Assume further that there exists $b \in \R$ such that
\begin{equation}\label{eq:semibddElliptic}
  \Re(B\varphi,\varphi)_{L^2(\partial\Omega)} \le b\|\varphi\|_{L^2(\partial\Omega)}^2
  \qquad \text{for all}\;\; \varphi \in \dom B.
\end{equation}
Then the operator
\begin{equation}\label{eq:ABelliptic}
  \AB f = \cL f, \qquad
  \dom \AB = \biggl\{f\in H^2(\Omega):
  \frac{\partial f}{\partial \nu_\cL}\Big|_{\partial\Omega} = Bf|_{\partial\Omega}\biggr\},
\end{equation}
in $L^2(\Omega)$ is m-sectorial, one has $\sigma(\AB)\subset\ov{W(\AB)}$,
the resolvent formula
\begin{equation}\label{eq:KreinElliptic}
  (\AB-\lambda)^{-1} = (\AN-\lambda)^{-1}
  + \gamma(\lambda)\big(I-BM(\lambda)\big)^{-1}B\gamma(\overline{\lambda})^*
\end{equation}
holds for all $\lambda \in \rho(\AB) \cap \rho(\AN)$, and the following assertions are true.
\begin{myenum}
\item % ----------
If $B$ is symmetric, then $\AB$ is self-adjoint and bounded from below.
If $B$ is dissipative (accumulative, respectively),
then $\AB$ is maximal accumulative (maximal dissipative, respectively).
\item % ----------
If\, $B'$ is a closable operator in $L^2(\partial\Omega)$ that
satisfies \eqref{assumpBelliptic} and \eqref{eq:semibddElliptic} with $B$
replaced by $B'$ and
\begin{equation}\label{BBprell}
  (B\varphi,\psi) = (\varphi,B'\psi)
  \qquad\text{for all}\;\; \varphi\in\dom B,\; \psi\in\dom B'
\end{equation}
holds, then $A_{[B']}=A_{[B]}^*$.
\end{myenum}
Moreover, the following spectral enclosures hold.
\begin{myenum}
\setcounter{counter_a}{2}
\item % ----------
If\, $b \le 0$, then $(-\infty,\min\sigma(\AN)) \subset \rho(\AB)$.
\item % ----------
If\, $\dom B^*\supset\dom B$, $\Im B$ is bounded and $b > 0$,
then for each $\mu < \min \sigma(\AN)$
there exists $C > 0$ such that for each $\xi < \mu-(Cb)^2$ one has
(see Fig.~\ref{fig:elliptic}\,(a))
\begin{equation*}
  W(\AB)
  \subset \Biggl\{z\in\CC: \Re z \ge \mu-(Cb)^2,\;
  |\Im z| \le \frac{2C\bigl\|\ov{\Im B}\bigr\|}{1-\frac{Cb}{(\mu-\xi)^{1/2}}}(\Re z-\xi)^{1/2}\Biggr\}.
\end{equation*}
\item % ----------
If\, $\dom B^*\supset\dom B$, $\Im B$ is bounded and $b \le 0$,
then for each $\mu < \min \sigma(A_{\rm N})$ there exists $C > 0$ such that
(see Fig.~\ref{fig:elliptic}\,(b))
\begin{equation*}
  W(\AB) \subset \biggl\{z\in\CC: \Re z \ge \min\sigma(\AN),\;
  |\Im z| \le \frac{2C\bigl\|\ov{\Im B}\bigr\|(\Re z-\mu)}{(\Re z-\mu)^{1/2}-Cb}\biggr\}.
%  \cup\{\mu\}.
 \end{equation*}
\item % ----------
If $B$ is bounded, then for each $w_0 < \min \sigma(A_{\rm N})$, $\nu \in (0, \pi)$
and $\beta \in \bigl(0, \frac{1}{2}\bigr)$ there exists $C > 0$ such that
\begin{align*}
  \sigma(\AB) \cap \bbU_{w_0, \nu} \subset
  \Bigl\{z\in\bbU_{w_0,\nu}: \dist(z,\sigma(\AN)) \le (C\|B\|)^{1/\beta}\Bigr\},
\end{align*}
where $\bbU_{w_0,\nu}$ is defined in \eqref{defUw0nu}.
\end{myenum}
\end{theorem}

\begin{figure}[ht]
\begin{center}
\begin{tabular}{cc}
\includegraphics[width=3cm]{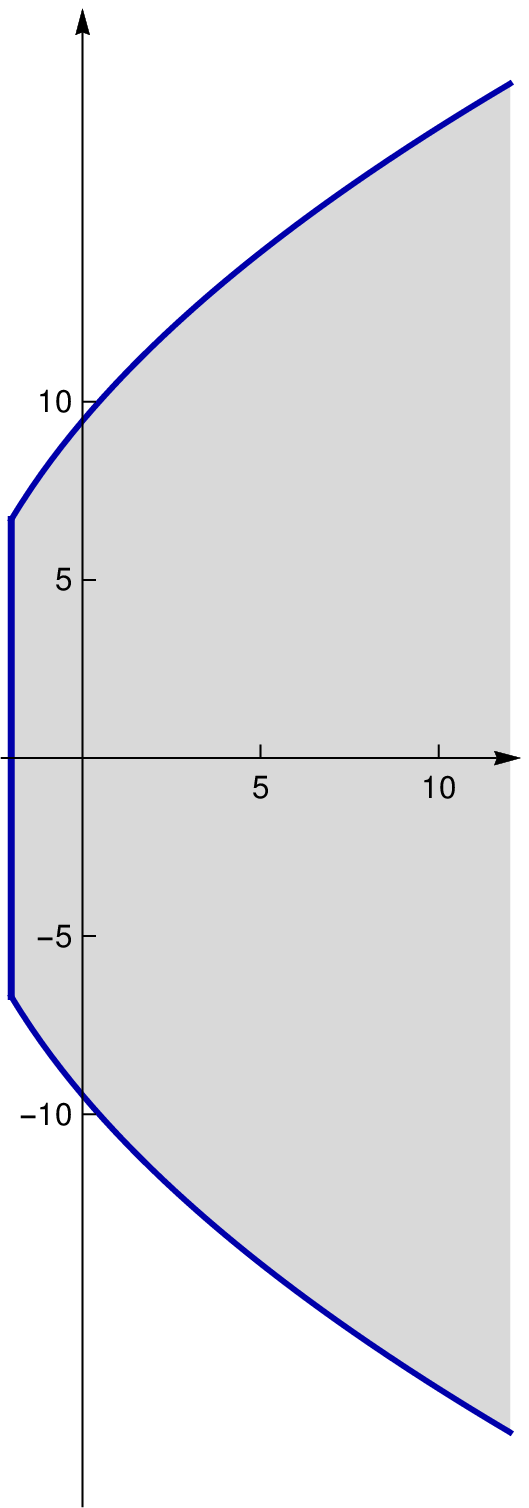} \hspace*{2ex}
&
\includegraphics[width=6cm]{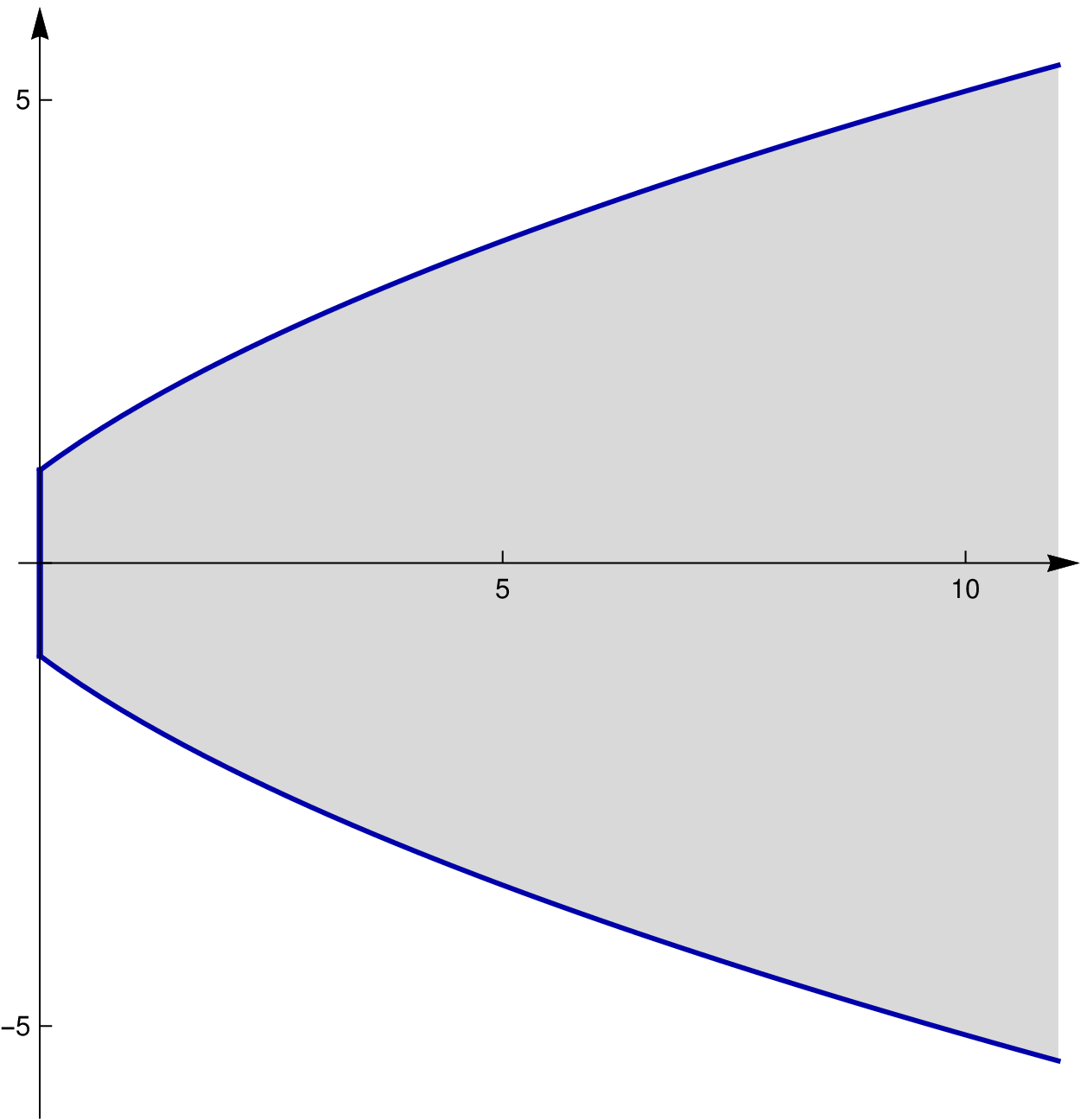}
\\[1ex]
(a) $b=1$ & (b) $b=-1$
\end{tabular}
\end{center}
\caption{The plots show the regions given in Theorem~\ref{thm:ABelliptic}\,(iv), (v),
respectively, that contain $W(\AB)$ for (a) $b>0$ and (b) $b<0$;
it is assumed that $\min\sigma(\AN)=0$, $C\|\ov{\Im B}\|=1$, $\mu=-1$
for both} cases and $\xi=-4$ in (a).
\label{fig:elliptic}
\end{figure}

\begin{proof}
Let $B$ be a closable operator in $L^2(\partial\Omega)$ that
satisfies \eqref{assumpBelliptic} and \eqref{eq:semibddElliptic} for some $b \in \R$.
Let $\{L^2(\partial\Omega), \Gamma_0, \Gamma_1\}$ be the quasi boundary triple
in Proposition~\ref{prop:QBTelliptic}.  It follows from
Lemma~\ref{lem:NDdecay} that~\eqref{eq:Mconvergence} is valid
for the corresponding Weyl function.  The assumptions (i) and (iv) and
the second assumption in (v) of Theorem~\ref{thm:operators1} are satisfied
due to the assumptions of the present theorem and the fact that $\AD=A_1$
is self-adjoint and bounded from below by Proposition~\ref{prop:QBTelliptic}.
Assumption (iii) of Theorem~\ref{thm:operators1} follows from the
last assertion of Proposition~\ref{prop:QBTelliptic}\,(iii) and \eqref{assumpBelliptic}.
For assumption (ii) of Theorem~\ref{thm:operators1} note that
\begin{align*}
  \ran \overline{M(\lambda)}^{1/2} = H^{1/2}(\partial\Omega),
  \qquad \lambda < \min \sigma(\AN),
\end{align*}
which can be verified as in the proof of \cite[Proposition~3.2\,(iii)]{BLLR17},
and use~\eqref{assumpBelliptic}.
It follows from Proposition~\ref{prop:QBTelliptic} that $A_0$ and $A_1$ are
bounded from below.
Thus Theorem~\ref{thm:operators1} and Corollary~\ref{cor:bnegative}
imply assertions (i)--(iii).
Moreover, Theorem~\ref{thm:parabola} and \eqref{eq:smallAgmon}
yield that $\AB$ is m-sectorial and the assertions in items (iv) and (v);
note that the estimate for $\Re z$ in (v) follows from taking the
estimates $\Re z>\mu$ in Theorem~\ref{thm:parabola}\,(b), (c) for all $\mu<\min\sigma(\AN)$.
Finally, to prove item (vi) one combines Lemma~\ref{lem:NDdecay} and
Proposition~\ref{prop:specEst}\,(a) with $G=\bbU_{w_0,\nu}$.
\end{proof}

\begin{remark}\rule{0ex}{1ex}
\begin{myenum}
\item
The constants $C$ in items (iv)--(vi) of the above theorem depend only
on the differential expression $\cL$ and the domain $\Omega$
and on $\mu$ in (iv), (v) and on $w_0, \nu, \beta$ in (vi);
the constants are independent of the operator $B$.
\item
In many cases (e.g.\ when $\Omega$ is bounded), one can define $T$
in \eqref{eq:Telliptic} on the larger domain
\[
  H^{3/2}_{\cL}(\Omega) \defeq \bigl\{f\in H^{3/2}(\Omega):
  \cL f\in L^2(\Omega)\bigr\};
\]
see \cite[\S4.2]{BL07}.  In this case the extensions of the
boundary mappings $\Gamma_0$ and $\Gamma_1$ to $H^{3/2}_{\cL}(\Omega)$ give rise
to a generalized boundary triple, and the second condition in \eqref{assumpBelliptic}
on $B$ is not needed to guarantee that the assertions of Theorem~\ref{thm:ABelliptic}
are true for the operator
\[
  \AB f = \cL f, \qquad
  \dom \AB = \biggl\{f\in H^{3/2}_{\cL}(\Omega):
  \frac{\partial f}{\partial \nu_\cL}\Big|_{\partial\Omega} = Bf|_{\partial\Omega}\biggr\},
\]
instead of~\eqref{eq:ABelliptic}.  In particular, for every bounded operator $B$
the statements (i)--(vi) in Theorem~\ref{thm:ABelliptic} are true.
The second condition in~\eqref{assumpBelliptic} is needed to obtain
the extra regularity $\dom A_{[B]} \subset H^2(\Omega)$;
see also \cite[Theorem~7.2]{AGW14} for a related result.

\item
The assertions in (iv) and (v) of Theorem~\ref{thm:ABelliptic} imply that the spectrum
of $\AB$ is contained in a parabola if $\dom B^*\supset\dom B$ and $\Im B$ is bounded.
This is in accordance with \cite[Theorem~5.14]{BF62},
where the Laplacian on a bounded domain with bounded $B$ was studied.
In that paper a setting with $H^{3/2}_{\cL}(\Omega)$ as mentioned in the previous item
of this remark was used.

\item
Under the basic assumptions of Theorem~\ref{thm:ABelliptic} the operator $A_{[B]}$
is m-sectorial and hence $-A_{[B]}$ generates an analytic semigroup.
For the Laplacian on a bounded domain $\Omega$ this was proved in \cite{ABN17}
in the $H^{3/2}_{\cL}(\Omega)$ setting as in (ii).
\end{myenum}
\end{remark}

\noindent
The next remark shows that the condition~\eqref{assumpBelliptic}
can be relaxed when an adjoint pair of boundary operators that
map $H^1(\partial\Omega)$ into $H^{1/2}(\partial\Omega)$ is given.
In this case the assumption $H^{1/2}(\partial \Omega) \subset \dom B$
is not needed.

\begin{remark}\label{re:dualpair}
Assume that $B_0$ and $B_0'$ are linear operators in $L^2(\partial\Omega)$ which satisfy
\begin{equation}\label{B0Bpr0ell}
  (B_0\varphi,\psi) = (\varphi,B_0'\psi)
  \qquad\text{for all}\;\; \varphi\in\dom B_0,\; \psi\in\dom B_0',
\end{equation}
and
\begin{alignat}{2}
  H^1(\partial\Omega) &\subset \dom B_0, \qquad &
  B_0\bigl(H^1(\partial\Omega)\bigr) &\subset H^{1/2}(\partial\Omega),
  \label{BBprell1}\\[0.5ex]
  H^1(\partial\Omega) &\subset \dom B_0', \qquad &
  B_0'\bigl(H^1(\partial\Omega)\bigr) &\subset H^{1/2}(\partial\Omega).
  \label{BBprell2}
\end{alignat}
Then $B_0$ and $B_0'$ have closable extensions $B$ and $B'$, respectively,
that satisfy \eqref{assumpBelliptic} and \eqref{BBprell}.
Indeed, it follows from \eqref{BBprell1} and \eqref{BBprell2} that $B_0$ and $B_0'$
are densely defined.  Hence \eqref{B0Bpr0ell} shows that $B_0$ and $B_0'$ are closable.
This and the second condition in \eqref{BBprell2}
imply that $B_0'\upharpoonright H^1(\partial\Omega)$ is bounded
from $H^1(\partial\Omega)$ to $H^{1/2}(\partial\Omega)$.
A duality argument as, e.g.\ in \cite[Lemma~4.4]{BLL13IEOT}
shows that the Banach space adjoint of $B_0'\upharpoonright H^1(\partial\Omega)$,
which we denote by $\wt B$, is an extension of $B_0$ and a bounded mapping
from $H^{-1/2}(\partial\Omega)$ to $H^{-1}(\partial\Omega)$.
Interpolation (see, e.g.\ \cite[Theorems~5.1 and 7.7]{LM72}) implies
that $B\defeq\wt B\upharpoonright H^{1/2}(\partial\Omega)$ is bounded from
$H^{1/2}(\partial\Omega)$ to $L^2(\partial\Omega)$.
Hence $H^{1/2}(\partial\Omega)\subset\dom B$
and \eqref{assumpBelliptic} is satisfied.
In a similar way one constructs an extension $B'$ of $B_0'$ that
satisfies $H^{1/2}(\partial\Omega)\subset\dom B'$.
The relation \eqref{BBprell} is obtained by continuity.
We emphasize that in this situation replacing $B$ by $B_0$ in the
definition of $A_{[B]}$ does not change the domain of the operator.
\end{remark}

If, for $B$, we choose a multiplication operator by some function $\alpha$,
we obtain classical Robin boundary conditions.  We formulate this situation
in the following corollary, which follows from Theorem~\ref{thm:ABelliptic}
and Remark~\ref{re:dualpair} with $B_0'$ being the multiplication operator
by $\ov{\alpha}$.

\begin{corollary}\label{cor:robin}
Let $\alpha$ be a measurable complex-valued function on $\partial\Omega$ such that
\begin{equation}\label{alphamultipl}
  \alpha\varphi \in H^{1/2}(\partial\Omega) \qquad\text{for all}\;\;
  \varphi \in H^1(\partial\Omega)
\end{equation}
and that
\begin{equation}\label{bsupReal}
  b \defeq \sup(\Re\alpha) < \infty.
\end{equation}
Then the operator
\[
  A_{[\alpha]} f = \cL f, \qquad
  \dom A_{[\alpha]} = \biggl\{f\in H^2(\Omega):
  \frac{\partial f}{\partial \nu_\cL}\Big|_{\partial\Omega}
  = \alpha f|_{\partial\Omega}\biggr\},
\]
in $L^2(\Omega)$ is m-sectorial, one has $\sigma(A_{[\alpha]})\subset\ov{W(A_{[\alpha]})}$,
and the resolvent formula
\[
  (A_{[\alpha]}-\lambda)^{-1} = (\AN-\lambda)^{-1}
  + \gamma(\lambda)\big(I-\alpha M(\lambda)\big)^{-1}\alpha\gamma(\overline{\lambda})^*
\]
holds for all $\lambda\in\rho(A_{[\alpha]})\cap\rho(\AN)$.
Moreover, the following assertions are true.
\begin{myenum}
\item
$A_{[\ov\alpha]}=A_{[\alpha]}^*$.
\item
If\, $\alpha$ is real-valued, then $A_{[\alpha]}$ is self-adjoint and bounded from below.
If\, $\Im(\alpha(x))\ge0$ ($\le 0$, respectively) for almost all $x\in\partial\Omega$,
then $A_{[\alpha]}$ is maximal accumulative (maximal dissipative, respectively).
\item
If\, $b\le0$ in \eqref{bsupReal}, then $(-\infty,\min\sigma(\AN))\subset\rho(A_{[\alpha]})$.
\end{myenum}
Further, if\, $\Im\alpha$ is bounded, then the enclosures for $W(A_{[\alpha]})$
in Theorem~\ref{thm:ABelliptic}\,{\rm(iv)} and {\rm(v)} hold with $\|\ov{\Im B}\|$
replaced by $\sup|\Im\alpha|$.
If $\alpha$ is bounded, then also the enclosure in Theorem~\ref{thm:ABelliptic}\,{\rm(vi)}
holds with $\|B\|$ replaced by $\sup|\alpha|$.
\end{corollary}

\begin{remark}\label{re:multiplier}
Condition \eqref{alphamultipl} says that $\alpha$ is a multiplier from
$H^1(\partial\Omega)$ to $H^{1/2}(\partial\Omega)$, in the notation of
\cite{MS09} written as
\[
  \alpha \in M\bigl(H^1(\partial\Omega)\to H^{1/2}(\partial\Omega)\bigr).
\]
In certain situations there exist characterizations or sufficient conditions
for this property.
For example let
\[
  \Omega = \RR_+^n = \bigl\{x=(x',x_n)^\top: x'\in\RR^{n-1},\, x_n>0\bigr\}.
\]
Then $\partial\Omega=\RR^{n-1}$.
The set of multipliers can be characterized using capacities;
see \cite[Theorem~3.2.2]{MS09}.
For the case $n=2$ there is a simpler characterization and for $n>2$ there
are simpler sufficient conditions.  To this end, let us recall some notation.
Let $H^{s,p}(\RR^{n-1})$ denote the (fractional) Sobolev space (or Bessel potential space)
defined as
\[
  H^{s,p}(\RR^{n-1}) = \bigl\{u\in\cS'(\RR^{n-1}): \cF M^s\cF^{-1}u \in L^p(\RR^{n-1})\bigr\}
\]
where $\cS'(\RR^{n-1})$ is the space of tempered distributions,
$\cF$ is the $(n-1)$-dimensional Fourier transform,
and $M$ is the operator of multiplication by $\sqrt{1+|\xi|^2}$;
see, e.g.\ \cite[\S 2.2.2\,(iii)]{ET96} or \cite[\S 3.1.1]{MS09}.
Further, let $\eta\in C_0^\infty(\RR^{n-1})$ be such that $\eta(x)=1$ on the
unit ball, and set $\eta_z(x)\defeq\eta(x-z)$ for $z\in\RR^{n-1}$.
Let
\[
  \Hlocunif^{s,p}(\RR^{n-1}) = \Bigl\{u\in\cS'(\RR^{n-1}):
  \sup_{z\in\RR^{n-1}}\|\eta_z u\|_{H^{s,p}(\RR^{n-1})}<\infty\Bigr\},
\]
a space of functions being in $H^{s,p}$ only locally but in a uniform way;
see \cite[p.~34]{MS09}.
We also set $\Hlocunif^s(\RR^{n-1}) \defeq \Hlocunif^{s,2}(\RR^{n-1})$.
When $n=2$, one obtains from \cite[Theorem~3.2.5]{MS09} that $\alpha$
satisfies \eqref{alphamultipl} if and only if
\begin{equation}\label{alpha_mult_n2}
  \alpha \in \Hlocunif^{\frac12}(\RR).
\end{equation}
In the case $n>2$ we can use \cite[Theorem~3.3.1\,(ii)]{MS09} to provide sufficient conditions:
$\alpha$ satisfies \eqref{alphamultipl} if
\begin{equation}\label{suffcondmult}
\begin{alignedat}{2}
  & \alpha \in \Hlocunif^{\frac12,p}(\RR^{n-1}) \;\;\text{for some $p\in(2,4)$} \qquad
  && \text{when}\;\; n=3, \\[1ex]
  & \alpha \in \Hlocunif^{\frac12,n-1}(\RR^{n-1}) \qquad
  && \text{when}\;\; n>3.
\end{alignedat}
\end{equation}
The implication in the case $n=3$ can be shown as follows:
if $\alpha\in \Hlocunif^{\frac12,p}(\RR^{n-1})$ and $p\in(2,4)$, then
$\alpha\in M\bigl(H^{\frac{2}{p}}(\RR^2) \to H^{\frac12}(\RR^2)\bigr)$
by \cite[Theorem~3.3.1\,(ii)]{MS09},
and since $H^1(\RR^2)$ is continuously embedded in $H^{\frac{2}{p}}(\RR^2)$,
we therefore have $\alpha\in M\bigl(H^1(\RR^2) \to H^{\frac12}(\RR^2)\bigr)$.

If $\Omega$ is a domain with smooth compact boundary, then one can
characterize multipliers using charts to reduce the situation to the half-space case,
i.e.\ $\alpha$ satisfies \eqref{alphamultipl} if and only
if $\alpha\in H^{\frac12}(\partial\Omega)$ when $n=2$;
when $n>2$, $\alpha$ satisfies \eqref{alphamultipl} if \eqref{suffcondmult} holds
with $\Hlocunif^{\frac12,p}(\RR^{n-1})$ replaced by $H^{\frac12,p}(\partial\Omega)$.
\end{remark}

\begin{example}\label{ex:unboundedalpha}
An example of an unbounded function $\alpha$ that satisfies \eqref{alpha_mult_n2} is
\[
  \alpha(x_1) = -\log\Bigl(\log\Bigl(1+\frac{1}{|x_1|}\Bigr)\Bigr),
  \qquad x_1\in(-1,1),
\]
smoothly connected, e.g.\ to the zero function outside $\RR\setminus(-2,2)$ or
to periodically shifted copies of this function.
That $\alpha$ belongs to $\Hlocunif^{\frac12}(\RR)$ can be seen from the fact
that it is the trace of a function $f\in H^1(\RR\times(0,\infty))$ that satisfies
\[
  f(x_1,x_2) = -\log\biggl(\log\biggl(1+\frac{1}{\sqrt{x_1^2+x_2^2}\,}\biggr)\biggr),
  \qquad x_1\in(-1,1),\,x_2\in(0,1).
\]
Note that such a function $\alpha$ also satisfies \eqref{bsupReal}
and hence Corollary~\ref{cor:robin} can be applied.
\end{example}

Let us consider an example in which the spectral estimates of the
previous theorem can be made more explicit.

\begin{example}\label{ex:halfspace}
Let $\Omega = \R^n_+ = \{ (x', x_n)^\top: x' \in \R^{n - 1}, x_n > 0 \}$,
so that $\partial\Omega = \R^{n - 1}$,
and consider the negative Laplacian $\cL = - \Delta$.
Then $\sigma(A_{\rm N}) = [0, \infty)$ and
the Weyl function of the quasi boundary triple in Proposition~\ref{prop:QBTelliptic}
can be calculated explicitly,
\begin{equation}\label{MLaplacian}
  \overline{M(\lambda)}
  = (-\Delta_{\RR^{n-1}}-\lambda)^{-1/2}, \qquad
  \lambda \in \CC \setminus [0, \infty);
\end{equation}
see, e.g.\ \cite[(9.65)]{G09}. Here $- \Delta_{\R^{n - 1}}$
denotes the self-adjoint Laplacian in $L^2(\R^{n - 1})$.
From \eqref{MLaplacian} we obtain
\begin{equation}\label{eq:WeylDecayExpPDE}
  \big\| \overline{M(\lambda)} \big\|
  = \frac{1}{\sqrt{\dist(\lambda,\RR_+)}\,}\,,
  \qquad \lambda \in \C \setminus [0, \infty).
\end{equation}
In particular, the estimate \eqref{eq:smallAgmon} is satisfied with $\mu=0$
and $C=1$.  Hence we can use Theorem~\ref{thm:parabola} to obtain
a better inclusion for the numerical range.
Let $B$ be a closable operator that satisfies \eqref{assumpBelliptic}
and \eqref{eq:semibddElliptic} such that $\dom B^*\supset\dom B$
and $\Im B$ is bounded.
If $b>0$, then for every $\xi<-b^2$ one has
\[
  W(\AB) \subset \Biggl\{z\in\CC: \Re z\ge -b^2,\;
  |\Im z| \le \frac{2\bigl\|\ov{\Im B}\bigr\|}{1-\frac{b}{\sqrt{|\xi|}}}
  (\Re z-\xi)^{1/2}\Biggr\}.
\]
If $b\le0$, then
\begin{equation}\label{Wenclhalfspace}
  W(\AB) \subset \biggl\{z\in\CC: \Re z>0,\;
  |\Im z| \le \frac{2\bigl\|\ov{\Im B}\bigr\|\Re z}{(\Re z)^{1/2}-b}\biggr\}\cup\{0\}.
\end{equation}
Note that $\sigma(\AB)\subset\ov{W(\AB)}$.
If $B$ is bounded, then we can use Proposition~\ref{prop:specEst}\,(a)
with $G=\CC\setminus[0,\infty)$ to obtain the spectral enclosure
\begin{equation}\label{diskecnlosure}
  \sigma(\AB)
  \subset \bigl\{z \in \CC: \dist(z,\RR_+) \le \|B\|^2 \bigr\}.
\end{equation}
In the case of the Robin boundary condition,
i.e.\ when $B$ is a multiplication operator with a complex-valued
function $\alpha$, an enclosure alternative to~\eqref{diskecnlosure}
can be found in~\cite[Theorem~2]{Fr15},
where the operator norm is replaced by an $L^p$-norm of $\alpha$
with a suitably chosen $p > 0$.
Finally, we remark that for $b \le 0$ and $z$ close to the origin,
the enclosure~\eqref{Wenclhalfspace} is sharper than~\eqref{diskecnlosure}.
\end{example}

If the boundary $\partial\Omega$ of $\Omega$ is compact, then the differences
of the resolvents of $\AB$ and $\AN$ or $\AD$, respectively,
belong to certain Schatten--von Neumann ideals as the following theorem shows.
For the case of a bounded self-adjoint operator $B$ in $L^2(\partial\Omega)$
the inclusions in \eqref{eq:three} and \eqref{eq:four} were proved
in \cite[Theorem~4.10 and Corollary~4.14]{BLL13IEOT}; cf.\ also~\cite{BLLLP10,G11}.

\begin{theorem}\label{thm:spelliptic}
Let $\partial\Omega$ be compact and let all assumptions of
Theorem~\ref{thm:ABelliptic} be satisfied.  Then
\begin{equation}\label{eq:one}
  (A_{[B]}-\lambda)^{-1} - (A_{\rm N}-\lambda)^{-1} \in \sS_p\bigl(L^2(\Omega)\bigr)
  \qquad \text{for all} \;\; p > \frac{2(n-1)}{3}
\end{equation}
and $\lambda \in \rho(A_{[B]}) \cap \rho(A_{\rm N})$, and
\begin{equation}\label{eq:two}
  (A_{[B]}-\lambda)^{-1} - (A_{\rm D}-\lambda)^{-1} \in \sS_p\bigl(L^2(\Omega)\bigr)
  \qquad \text{for all} \;\; p > \frac{2(n-1)}{3}
\end{equation}
and $\lambda \in \rho(A_{[B]}) \cap \rho(A_{\rm D})$.
If, in addition, $B \in \cB(L^2(\partial\Omega))$ then
\begin{equation}\label{eq:three}
  (A_{[B]}-\lambda)^{-1} - (A_{\rm N}-\lambda)^{-1} \in \sS_p\bigl(L^2(\Omega)\bigr)
  \qquad \text{for all} \;\; p > \frac{n-1}{3}
\end{equation}
and $\lambda \in \rho(A_{[B]}) \cap \rho(A_{\rm N})$, and
\begin{equation}\label{eq:four}
  (A_{[B]}-\lambda)^{-1} - (A_{\rm D}-\lambda)^{-1} \in \sS_p\bigl(L^2(\Omega)\bigr)
  \qquad \text{for all} \;\; p > \frac{n-1}{2}
\end{equation}
and $\lambda \in \rho(A_{[B]}) \cap \rho(A_{\rm D})$.
\end{theorem}

\begin{proof}
Let $\{L^2(\partial\Omega), \Gamma_0, \Gamma_1\}$ be the quasi boundary triple
in Proposition~\ref{prop:QBTelliptic} and let $\gamma$ be the corresponding $\gamma$-field.
Clearly, $\gamma(\lambda)^* \in \cB(L^2(\Omega), L^2(\partial\Omega))$,
and it follows from~\eqref{gammastar} that
$\ran\gamma(\lambda)^* = \ran(\Gamma_1\upharpoonright\dom A_{\rm N})
= H^{3/2}(\partial\Omega)$ for all $\lambda \in \rho(A_{\rm N})$.
Therefore we can conclude as in \cite[Lemma~3.4]{BLLLP10} that
\begin{equation}\label{eq:spdieerste}
  \gamma(\lambda)^* \in \sS_p\bigl(L^2(\Omega),L^2(\partial\Omega)\bigr)
  \quad \text{for~all} \quad p > \frac{2(n - 1)}{3}
\end{equation}
and for each $\lambda \in \rho(A_{\rm N})$.
Moreover, for $\lambda \in \rho(A_{\rm N}) \cap \rho(A_{\rm D})$
we have the relations $M(\lambda)^{-1} \gamma(\ov{\lambda})^* \in \cB(L^2(\Omega),L^2(\partial\Omega))$
and $\ran M(\lambda)^{-1} \gamma(\overline{\lambda})^* = H^{1/2}(\partial\Omega)$
since $M(\lambda)^{-1}$ maps $H^{3/2}(\partial\Omega)$ onto $H^{1/2}(\partial\Omega)$.
It follows again as in \cite[Lemma~3.4]{BLLLP10} that
\begin{equation}\label{eq:spdiezwote}
  M(\lambda)^{-1} \gamma(\overline{\lambda})^* \in \sS_q\bigl(L^2(\Omega),L^2(\partial\Omega)\bigr)
  \qquad \text{for all} \quad q > 2(n-1)
\end{equation}
and for each $\lambda \in \rho(A_{\rm N}) \cap \rho(A_{\rm D})$.
From~\eqref{eq:spdieerste} we obtain with the help of Proposition~\ref{spprop}
the assertions~\eqref{eq:one} and~\eqref{eq:two}.
For $B \in \cB(L^2(\partial\Omega))$, Proposition~\ref{spprop2},
\eqref{eq:spdieerste} and \eqref{eq:spdiezwote} yield \eqref{eq:three}
and~\eqref{eq:four}.
\end{proof}

\begin{remark}\label{rem:Weak1}
Note that the statement of Theorem~\ref{thm:spelliptic}
can be refined if we replace the usual Schatten--von Neumann classes
$\sS_p$ by the weak Schatten--von Neumann classes $\sS_{p,\infty}$,
which are discussed in Remark~\ref{re:weakSvN}.
In this case one can allow $p$ to be equal to $2(n-1)/3$, $(n-1)/3$ or $(n-1)/2$,
respectively; cf.\ \cite[Section~4.2]{BLL13IEOT}	
and \cite[Section~3]{BLL13trace}.
\end{remark}

% *******************************************************************
\section{Schr\"odinger operators with $\delta$-interaction on hypersurfaces}
\label{sec:delta}
% *******************************************************************

\noindent
In this section we provide some applications of the results in
Sections~\ref{sec:closedoperators}, \ref{sec:consequences} and \ref{sec:suffconddecay}
to Schr\"odinger operators with $\delta$-interaction supported on a
smooth, not necessarily bounded hypersurface $\Sigma$ in $\R^n$.
To be more specific, we consider operators associated with the
formal differential expression
\[
  -\Delta-\alpha\langle\,\cdot,\delta_\Sigma\rangle\delta_\Sigma,
\]
where $\alpha$ is a complex constant or a complex-valued function on $\Sigma$,
the strength of the $\delta$-interaction.
The spectral theory of such operators is a prominent subject in
mathematical physics; see the review paper~\cite{E08}, the monograph \cite{EK15},
and the references therein.  The largest part of the existing literature
(see, e.g.\ \cite{BEKS94, EI01, EK03, ER16, EY02, LO16, MPS16})
is devoted to the case of a real interaction strength $\alpha$.
However, there has been recent interest in non-real $\alpha$;
see, e.g.\ \cite{Fr15, KK16}.

In what follows, let $\Omega_+$ be a uniformly regular, bounded or
unbounded domain in $\R^n$ (see Section~\ref{sec:Robin})
with boundary $\Sigma \defeq \partial\Omega_+$.
Furthermore, let $\Omega_- = \R^n\setminus(\Omega_+\cup\Sigma)$
be its complement in $\R^n$.  We write
$f = f_+ \oplus f_-$ for $f \in L^2(\R^n)$, where $f_\pm = f |_{\Omega_\pm}$.
By the same reason as in Section~\ref{sec:Robin}, the trace and the
normal derivative extend to continuous linear mappings
\[
  H^2(\Omega_\pm) \ni f_\pm \mapsto
  \biggl\{f_\pm|_\Sigma;\; \frac{\partial f_\pm}{\partial \nu_\pm}\Big|_\Sigma\biggr\}
  \in H^{3/2}(\Sigma) \times H^{1/2}(\Sigma).
\]
Both the above mappings are surjective onto $H^{3/2}(\Sigma) \times H^{1/2}(\Sigma)$.
Furthermore, we introduce an operator $T$ in $L^2(\R^n)$ by
\begin{equation}\label{eq:Tdelta}
  Tf = (-\Delta f_{\rm +}) \oplus (-\Delta f_{\rm -}),\qquad
  \dom T = H^2(\R^n\setminus\Sigma)\cap H^1(\R^n).
\end{equation}
On $\dom T$ we define boundary mappings $\Gamma_0$ and $\Gamma_1$ by
\begin{equation}\label{eq:Gammasdelta}
  \Gamma_0 f = \frac{\partial f_+}{\partial\nu_+}\Big|_\Sigma
  + \frac{\partial f_-}{\partial\nu_-}\Big|_\Sigma,
  \qquad
  \Gamma_1 f = f|_\Sigma \qquad \text{for}\;f \in \dom T;
\end{equation}
here $\frac{\partial f_\pm}{\partial\nu_\pm}\big|_\Sigma$ stand
for the normal derivatives of $f = f_+ \oplus f_-\in\dom T$ on two
opposite faces of $\Sigma$ with the normals pointing outwards $\Omega_\pm$;
note that the outer unit normal vector fields $\nu_-$ and $\nu_+$
of $\Omega_-$ and $\Omega_+$, respectively, satisfy $\nu_-(x) = - \nu_+(x)$
for all $x \in \Sigma$.  Moreover, consider the symmetric operator $S$
in $L^2(\R^n)$ defined as
\begin{equation}\label{eq:Sdelta}
  Sf = - \Delta f, \qquad
  \dom S = H^2(\R^n)\cap H^1_0(\R^n\setminus\Sigma).
\end{equation}

In the following proposition we state that $\{L^2(\Sigma), \Gamma_0,\Gamma_1\}$
is a quasi boundary triple for $T\subset S^*$ and we formulate
properties of this triple and of the associated $\gamma$-field
and Weyl function.  This proposition is analogous to
Proposition~\ref{prop:QBTelliptic} and can be proved in a similar way;
see the proofs of \cite[Propositions~3.1 and~3.2]{BLLR17}.
Note that in the case of a compact $\Sigma$, the statements and proofs
of the next proposition and further details can be found in \cite[\S3]{BLL13delta}
and \cite[\S3.1]{BLL13IEOT}.

\begin{proposition}\label{prop:QBTdelta}
The operator $S$ in~\eqref{eq:Sdelta}
is closed, symmetric and densely defined
with $S^* = \overline{T}$ for $T$ in~\eqref{eq:Tdelta}, and the
triple $\{L^2(\Sigma), \Gamma_0, \Gamma_1\}$
is a quasi boundary triple for $T\subset S^*$ with the following properties.
\begin{myenum}
\item % ----------
$\ran(\Gamma_0, \Gamma_1)^{\top} = H^{1/2}(\Sigma)\times H^{3/2}(\Sigma)$.
\item % ----------
$A_0$ is the \emph{free Laplace operator}
\[
  -\Delta_{\R^n}f = -\Delta f, \qquad \dom(-\Delta_{\R^n}) = H^2(\R^n),
\]
and $A_1$ is the orthogonal sum of the \emph{Dirichlet Laplacians} on $\Omega_+$
and $\Omega_-$, respectively,
\[
  -\Delta_{\rm D} f  = -\Delta f, 	
  \qquad
  \dom(-\Delta_{\rm D}) =  H^2(\R^n\setminus\Sigma)\cap H^1_0(\R^n\setminus\Sigma).
\]
Both operators, $-\Delta_{\dR^n}$ and $-\Delta_{\rm D}$,
are self-adjoint and non-negative in $L^2(\R^n)$.
\item % ----------
For all $\lambda \in \C\setminus\R_+$ the associated $\gamma$-field satisfies
\begin{equation}\label{eq:gammadelta}
  \gamma(\lambda)\left(\frac{\partial f_+}{\partial \nu_+} \Big|_\Sigma
  + \frac{\partial f_-}{\partial \nu_-}\Big|_\Sigma\right)
  = f \qquad \text{for all}\;\; f \in \ker(T-\lambda),
\end{equation}
and the associated Weyl function is given by:
\begin{equation}\label{eq:Weyldelta}
  M(\lambda)\left(\frac{\partial f_+}{\partial \nu_+} \Big|_\Sigma
  +	\frac{\partial f_-}{\partial \nu_-}\Big|_\Sigma\right)
  = f|_\Sigma \qquad \text{for all}\;\; f \in \ker(T-\lambda).
\end{equation}
Moreover, $M(\lambda)$ is a bounded, non-closed operator in $L^2(\Sigma)$
with domain $H^{1/2}(\Sigma)$ such that $\ran\ov{M(\lambda)} \subset H^1(\Sigma)$.
\end{myenum}
\end{proposition}

\medskip

\noindent
The following lemma ensures the decay of the Weyl function $M$ in~\eqref{eq:Weyldelta}.
For the definition of the exterior sector $\bbU_{w_0,\nu}$ we refer to~\eqref{defUw0nu}.

\begin{lemma}\label{lem:Weyldeltadecay}
Let $M$ denote the Weyl function in~\eqref{eq:Weyldelta}.
Then for all $w_0 < 0$, $\nu \in (0,\pi)$, and $\beta \in (0,\frac12)$
there exists a constant $C = C(\Sigma, \beta, w_0, \nu) > 0$ such that
\begin{equation}\label{powerdecaydelta}
  \big\|\ov{M(\lambda)}\big\| \le \frac{C}{\bigl(\dist(\lambda,\dR_+)\bigr)^{\beta}}
  \qquad \text{for all}\;\;\lambda \in \bbU_{w_0,\nu}.
\end{equation}
\end{lemma}

\begin{proof}
Let $\{L^2(\Sigma), \Gamma_0, \Gamma_1\}$ be the quasi boundary triple in
Proposition~\ref{prop:QBTdelta}. Recall that $A_0 = - \Delta_{\R^n}$;
in particular, $\sigma(A_0) = [0, \infty)$
and $\dom(A_0 + 1)^\frac{s}{2} = H^s(\R^n)$ for all $s > 0$ by the
definition of the Sobolev spaces.
Hence by the closed graph theorem, $(A_0 + 1)^{-s/2}$ is bounded as
an operator from $L^2(\R^n)$ to $H^s(\R^n)$ for each $s \geq 0$.
Since the trace map is bounded from $H^s(\R^n)$ to $L^2(\Sigma)$
for each $s \in (\frac{1}{2}, 1)$, it follows that
$f \mapsto ((A_0 + 1)^{-s/2} f) |_{\Sigma}$ is bounded
from $L^2(\R^n)$ to $L^2(\Sigma)$ for each $s \in (\frac{1}{2}, 1)$.
Therefore the operator
\[
  \Gamma_1(A_0 + 1)^{-\alpha}: L^2(\R^n)
  \supset \dom \bigl(\Gamma_1(A_0 + 1)^{-\alpha}\bigr) \to L^2(\Sigma)
\]
is bounded for each $\alpha \in (\frac{1}{4}, \frac{1}{2})$.
By Theorem~\ref{th:weyl_decay} it follows that for each $w_0<0$,
each $\nu\in(0,\pi)$ and each $\alpha \in (\frac{1}{4}, \frac{1}{2})$
there exists $C = C(\Sigma, \beta, w_0, \nu) > 0$ such that
\[
  \big\|\overline{M(\lambda)}\big\|
  \le \frac{C}{\bigl(\dist(\lambda, \R_+)\bigr)^{1-2\alpha}}
\]
holds for all $\lambda \in \bbU_{w_0, \nu}$.
From this the claim of the lemma follows.
\end{proof}

\begin{remark}\label{re:deltaMDecay}
It can be shown as in \cite[Proposition~3.2\,(iii)]{BLL13delta} that
\begin{equation}\label{MMplMmi}
  M(\lambda) = \bigl(M_+(\lambda)^{-1}+M_-(\lambda)^{-1}\bigr)^{-1}, \qquad
  \lambda\in\CC\setminus[0,\infty),
\end{equation}
where $M_+$ and $M_-$ are the Weyl functions from Section~\ref{sec:Robin}
for $-\Delta$ on $\Omega_+$ and $\Omega_-$, respectively.
Remark~\ref{re:NDdecay} implies that for each $\mu<0$ there
exist $C_\pm>0$ such that
\[
  \big\|\ov{M_\pm(\lambda)}\big\| \le \frac{C_\pm}{(\mu-\lambda)^{1/2}}\,,
  \qquad \lambda<\mu.
\]
Since $M_\pm(\lambda)\ge0$ for $\lambda\in(-\infty,0)$,
it follows from~\cite[Corollaries~I.2.4 and I.3.2]{Ando78} that
\[
  \bigl\|\ov{M(\lambda)}\bigr\|
  \le \frac14\Bigl(\bigl\|\ov{M_+(\lambda)}\bigr\| + \bigl\|\ov{M_-(\lambda)}\bigr\|\Bigr).
\]
Hence for each $\mu < 0$
there exists $C = C(\Sigma, \mu) > 0$ such that
\[
  \bigl\|\ov{M(\lambda)}\bigr\| \le \frac{C}{(\mu-\lambda)^{1/2}}
  \qquad \text{for all}\;\; \lambda < \mu.
\]
\end{remark}	

\medskip

From Lemma~\ref{lem:Weyldeltadecay}, Remark~\ref{re:deltaMDecay} and the
results of Section~\ref{sec:consequences} we obtain the following consequences
for Schr\"odinger operators with $\delta$-potentials supported on $\Sigma$;
cf.\ the proof of Theorem~\ref{thm:ABelliptic} and Corollary~\ref{cor:robin}.
Note that the assumptions of the next theorem allow certain classes
of unbounded functions $\alpha$; cf.\ Remark~\ref{re:multiplier}.

\begin{theorem}\label{thm:ABdelta}
Let $\alpha$ be a measurable complex-valued function such that
\begin{equation}\label{eq:alphaCond2}
  \alpha\varphi \in H^{1/2}(\Sigma) \qquad\text{for all}\;\; \varphi \in H^1(\Sigma),
\end{equation}
and that
\[
  b \defeq \sup(\Re\alpha) < \infty.
\]
Then the Schr\"odinger operator with $\delta$-interaction of strength $\alpha$
supported on $\Sigma$,
\begin{equation}\label{eq:ABdelta}
\begin{split}
  A_{[\alpha]} f & = (-\Delta f_+) \oplus (-\Delta f_-), \\
  \dom A_{[\alpha]} & = \left\{ f \in H^2(\R^n \setminus \Sigma)\cap H^1(\R^n) :
  \frac{\partial f_+}{\partial \nu_+} \Big|_\Sigma
  + \frac{\partial f_-}{\partial\nu_-}\Big|_\Sigma = \alpha f |_\Sigma \right\},
\end{split}
\end{equation}
in $L^2(\R^n)$
is m-sectorial, one has $\sigma(A_{[\alpha]})\subset\ov{W(A_{[\alpha]})}$, the resolvent formula
\begin{equation}\label{eq:KreinDelta}
  (A_{[\alpha]}-\lambda)^{-1} = (-\Delta_{\R^n}-\lambda)^{-1}
  + \gamma(\lambda)\big(I-\alpha M(\lambda)\big)^{-1}\alpha\gamma(\overline{\lambda})^*
\end{equation}
holds for all $\lambda\in\rho(A_{[\alpha]}) \setminus \dR_+$, and the following assertions are true.
\begin{myenum}
\item % ----------
$A_{[\ov\alpha]} = A_{[\alpha]}^*$.
\item % ----------
If $\alpha$ is real-valued, then $A_{[\alpha]}$ is self-adjoint and bounded from below.
If\, $\Im(\alpha(s)) \ge 0$ ($\le 0$, respectively) for almost all $s\in\Sigma$,
then $A_{[\alpha]}$ is maximal accumulative (maximal dissipative, respectively).
\end{myenum}
Moreover, the following spectral enclosures hold.
\begin{myenum}
\setcounter{counter_a}{2}
\item % ----------
If\, $b \le 0$, then $(-\infty, 0) \subset \rho(A_{[\alpha]})$.
\item % ----------
If\, $\Im\alpha$ is bounded and $b>0$, then for each $\mu < 0$ there exists $C > 0$ such that
for each $\xi < \mu-(Cb)^2$,
\[
  W(A_{[\alpha]})
  \subset \Biggl\{z\in\CC: \Re z \ge \mu-(Cb)^2,\;
  |\Im z| \le \frac{2C\|\Im\alpha\|_\infty}{1-\frac{C b}{(\mu-\xi)^{1/2}}}(\Re z-\xi)^{1/2}\Biggr\}.
\]
\item % ----------
If\, $\Im\alpha$ is bounded and $b\le0$, then for each $\mu < 0$ there exists $C > 0$ such that
\[
  W(A_{[\alpha]})
  \subset \biggl\{z\in\CC: \Re z \ge 0,\;
  |\Im z| \le \frac{2C\|\Im\alpha\|_\infty(\Re z-\mu)}{(\Re z-\mu)^{1/2}-Cb}\biggr\}.
\]
\item[\textup{(vi)}] % ----------
If\, $\alpha$ is bounded, then for each $w_0 < 0$, $\nu \in (0, \pi)$
and $\beta \in \bigl(0, \frac{1}{2}\bigr)$ there exists $C > 0$ such that
\[
  \sigma(A_{[\alpha]}) \cap \bbU_{w_0, \nu} \subset
  \Bigl\{z\in\bbU_{w_0,\nu}: \dist(z, \R_+) \le (C\|\alpha\|_\infty)^{1/\beta}\Bigr\},
\]
where $\bbU_{w_0,\nu}$ is defined in \eqref{defUw0nu}.
\end{myenum}
\end{theorem}

\medskip

\noindent
Let us illustrate the obtained spectral estimates in an example.

\begin{example}
Consider the case
\[
  \Omega_\pm = \R^n_\pm = \bigl\{x = (x', x_n)^\top \in \R^n :
  x' \in \R^{n - 1}, \pm x_n > 0 \bigr\},
\]
that is, $\Sigma = \{(x', 0)^\top : x' \in \R^{n-1}\}$, which we
identify with $\R^{n-1}$.
It follows from \eqref{MMplMmi} and \eqref{MLaplacian} that
\[
  \overline{M(\lambda)} = \frac{1}{2}\bigl(-\Delta_{\R^{n-1}}-\lambda\bigr)^{-1/2}, \qquad
  \lambda \in \C \setminus [0, \infty),
\]
and hence
\[
  \big\|\ov{M(\lambda)}\big\| = \frac{1}{2\sqrt{\dist(\lambda,\R_+)}\,}\,, \qquad
  \lambda \in \C \setminus [0, \infty).
\]
In particular, the estimate~\eqref{eq:estMpower} is satisfied with
$\mu = 0$ and $C = 1/2$.
In analogy to Example~\ref{ex:halfspace}, this observation can be used
to obtain several better enclosures for the spectrum and
numerical range of the operator $A_{[\alpha]}$.
Let $\alpha$ satisfy the conditions of Theorem~\ref{thm:ABdelta} and
let $\Im\alpha$ be bounded.
If $b > 0$, then for every $\xi < -b^2 / 4$ one has
\[
  W(A_{[\alpha]}) \subset \Biggl\{z\in\dC\colon \Re z \ge -\frac{b^2}{4}\,,\;
  |\Im z| \le \frac{\|\Im\alpha\|_\infty}{1-\frac{b}{2\sqrt{|\xi|}}}
  (\Re z - \xi)^{1/2}\Biggr\}.
\]
If $b \le 0$, then
\[
	W(A_{[\alpha]}) \subset \bigg\{z\in\dC\colon \Re z > 0,
	|\Im z| \le \frac{\|\Im\alpha\|_\infty\Re z}{(\Re z)^{1/2} - b/2}
	\bigg\}\cup\{0\}.
\]
If, in addition, $\alpha$ is  bounded, then by Proposition~\ref{prop:specEst}\,(a)
with $G = \dC\setminus\dR_+$ the spectrum of $A_{[\alpha]}$ satisfies the enclosure
\[
  \sigma(A_{[\alpha]}) \subset
  \biggl\{ z \in \C : \dist(z,\dR_+) \le \frac14\|\alpha\|_\infty^2 \biggr\}.
\]
\end{example}

We now have a closer look at the special case of a compact hypersurface $\Sigma$
and bounded $\alpha$.  For this case certain refined bounds for the
function $M$ from the recent work~\cite{GS15} are available and can be combined
with the results in the abstract part of this paper in order to obtain
the spectral bounds for $A_{[\alpha]}$ that are contained in the next theorem.
We remark that~\cite{GS15}
contains further bounds in space dimension two and in the special case
when $\Omega_+$ is a convex domain, which could be combined with our theorems;
however, we do not include this in the next theorem.

\begin{theorem}\label{thm:specEstdelta2}
Let $\Sigma$ be compact and let $\alpha\in L^\infty(\Sigma)$ be a complex-valued
function which satisfies \eqref{eq:alphaCond2}.
Then there exist constants $C_1,C_2 > 0$, which are independent of $\alpha$,
such that the spectrum of $A_{[\alpha]}$ satisfies
\begin{equation}\label{encl_galkowski}
  \sigma(A_{[\alpha]})
  \setminus\dR_+ \subset \dV_{\alpha,C_1} \cap \dW_{\alpha,C_2},
\end{equation}
where
\begin{align*}
  \dV_{\alpha,C_1} &\defeq
  \begin{cases}
    \Bigl\{z\in\C\setminus\{0\}:
    C_1\|\alpha\|_\infty\bigl(2+|z|\bigr)^{-\frac{1}{4}}\ln\bigl(2+|z|^{-1}\bigr)\ge 1\Bigr\},
    & n = 2,
    \\[2ex]
    \Bigl\{z\in\C:
    C_1\|\alpha\|_\infty\bigl(2+|z|\bigr)^{-\frac{1}{4}}\ln\bigl(2+|z|\bigr)\ge 1\Bigr\},
    & n\ge 3,
  \end{cases}
  \displaybreak[0]\\[1ex]
  \dW_{\alpha,C_2} &\defeq
  \begin{cases}
    \Bigl\{z\in\C\setminus\{0\}: C_2\|\alpha\|_\infty\bigl(2+|\Im\sqrt{z}|^2\bigr)^{-\frac{1}{2}}
    \ln\bigl(2 + |z|^{-1}\bigr) \ge 1 \Bigr\},  & n = 2,
    \\[2ex]
    \Bigl\{z\in\C:
    C_2\|\alpha\|_\infty\bigl(2+|\Im\sqrt{z}|^2\bigr)^{-\frac{1}{2}} \ge 1 \Bigr\}, & n\ge 3.
  \end{cases}
\end{align*}
\end{theorem}

\begin{figure}[ht]
\begin{center}
\begin{tabular}{cc}
\includegraphics[width=6cm]{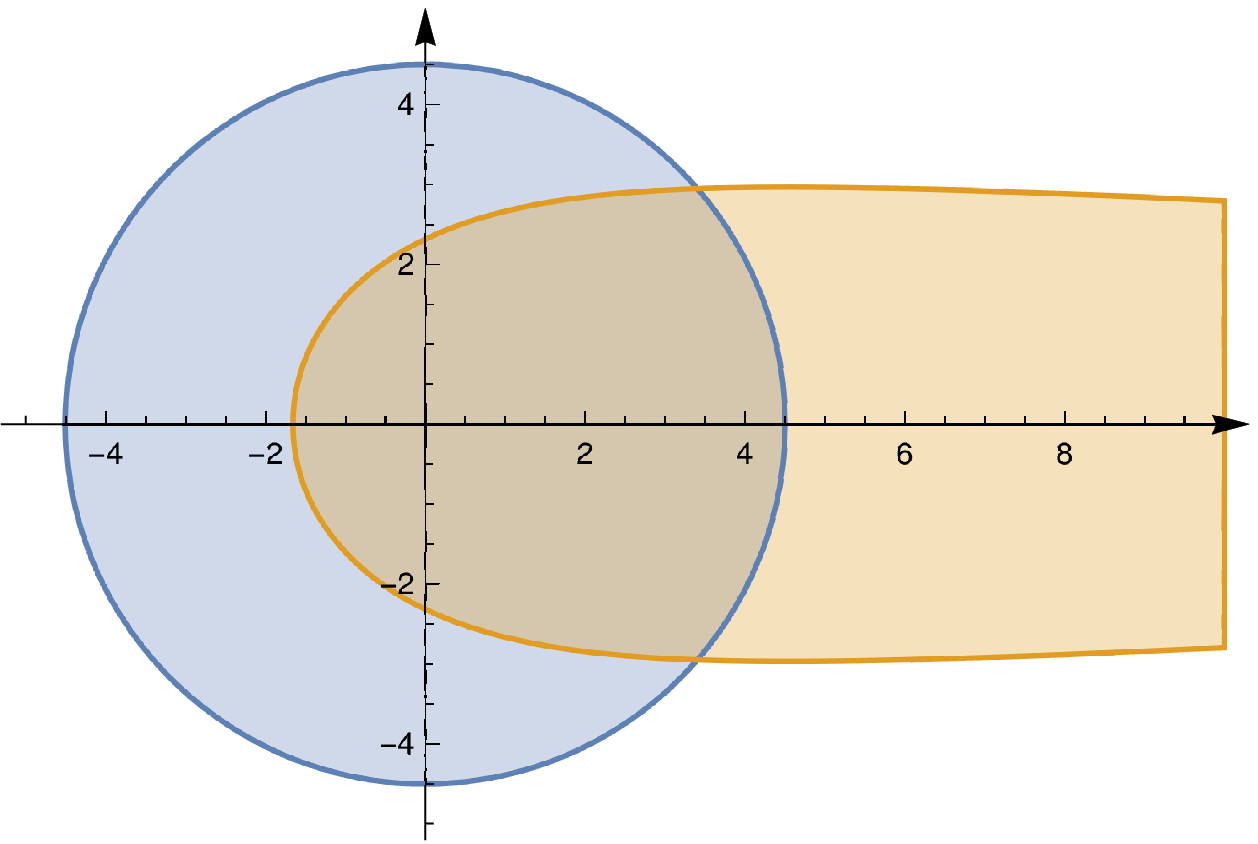}
&
\includegraphics[width=6cm]{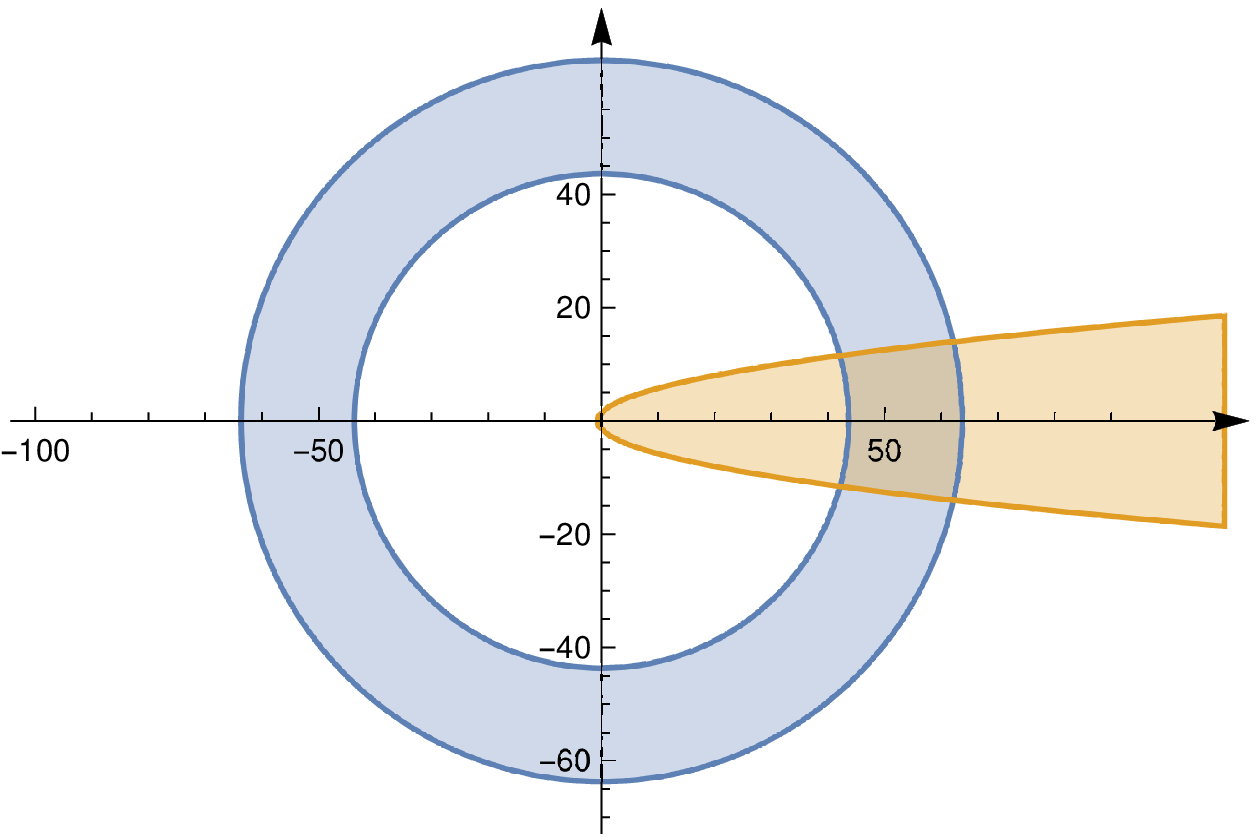}
\\[1ex]
(a) $n=2$ & (b) $n\ge3$
\end{tabular}
\end{center}
\caption{The sets $\dV_{\alpha,C_1}$ (blue) and $\dW_{\alpha,C_2}$ (yellow)
in Theorem~\ref{thm:specEstdelta2} for (a) $n = 2$ and (b) $n \geq 3$, respectively,
where $C_1\|\alpha\|_\infty=C_2\|\alpha\|_\infty=0.5$ in (a),
and $C_1\|\alpha\|_\infty=1.47$ and $C_2\|\alpha\|_\infty=0.6$ in (b).}
\label{fig:thm86}
\end{figure}

\begin{proof}
By \cite[Theorems~1.2 and 1.3]{GS15} there exist constants $C_1,C_2 >0$
(the constants here differ from the ones in \cite{GS15}
by a factor $\frac{1}{2}$) such that
\begin{align*}
  &\big\|\alpha\overline{M(\lambda)} \big\| \le
  \begin{cases}
    C_1\|\alpha\|_\infty\bigl(2+|\lambda|\bigr)^{-\frac{1}{4}}\ln\bigl(2+|\lambda|^{-1}\bigr),
    &\quad n = 2,
    \\[1ex]
    C_1\|\alpha\|_\infty\bigl(2+|\lambda|\bigr)^{-\frac{1}{4}}\ln\bigl(2+|\lambda|\bigr),
    &\quad n \ge 3,
  \end{cases}
  \displaybreak[0]\\[1ex]
  &\big\|\alpha \overline{M(\lambda)} \big\| \le
  \begin{cases}
    C_2\|\alpha\|_\infty\bigl(2+|\Im\sqrt{\lambda}|^2\bigr)^{-\frac{1}{2}}
    \ln\bigl(2+|\lambda|^{-1}\bigr),
    &\quad n = 2,
    \\[1ex]
    C_2\|\alpha\|_\infty\bigl(2+|\Im\sqrt{\lambda}|^2\bigr)^{-\frac{1}{2}},
    &\quad n \ge 3,
  \end{cases}
\end{align*}
hold for all $\lambda \in\dC\setminus \dR_+$.
Thanks to condition \eqref{eq:alphaCond2} we can view the multiplication
with $\alpha$ as an operator in $L^2(\Sigma)$ with domain $H^1(\Sigma)$
and range contained in $H^{1/2}(\Sigma)$.
Hence, by Theorem~\ref{mainthm}, any point $\lambda\in\dC\setminus\dR_+$
for which at least one of the above two upper bounds
on $\|\alpha\ov{M(\lambda)}\|$ is strictly less than one
belongs to the resolvent set of $A_{[\alpha]}$.
Thus, the enclosure in \eqref{encl_galkowski} follows.
\end{proof}

Furthermore, we obtain certain Schatten--von Neumann estimates for the
difference of the resolvents of $A_{[\alpha]}$ and the free Laplacian.
They are analogues of the first and the third estimates in
Theorem~\ref{thm:spelliptic}, and the proofs are analogous, where one uses the relations
\[
  \gamma(\lambda)^* \in \sS_p\bigl(L^2(\RR^n),L^2(\Sigma)\bigr)
  \qquad \text{for all}\;\; p > \frac{2(n-1)}{3}
\]
and
\[
  M(\lambda)^{-1} \gamma(\ov{\lambda})^* \in \sS_q\bigl(L^2(\RR^n),L^2(\Sigma)\bigr)
  \qquad \text{for all}\;\; q > 2(n-1).
\]

\begin{theorem}\label{thm:Spdelta}
Let all assumptions of Theorem~\ref{thm:ABdelta} be satisfied.
Moreover, assume that $\Sigma$ is compact.  Then
\begin{equation}\label{eq:one1}
  (A_{[\alpha]}-\lambda)^{-1} - (-\Delta_{\R^n}-\lambda)^{-1} \in \sS_p\bigl(L^2(\RR^n)\bigr)
  \qquad \text{for all}\;\; p > \frac{2(n-1)}{3}
\end{equation}
and all $\lambda \in \rho(A_{[\alpha]})$.
If, in addition, $\alpha$ is bounded, then
\begin{equation}\label{eq:three1}
  (A_{[\alpha]}-\lambda)^{-1} - (-\Delta_{\R^n}-\lambda)^{-1} \in \sS_p\bigl(L^2(\RR^n)\bigr)
  \qquad \text{for all}\;\; p > \frac{n-1}{3}
\end{equation}
and all $\lambda \in \rho(A_{[\alpha]})$.
\end{theorem}

\begin{remark}
In the same way as in Remark~\ref{rem:Weak1}, we can reformulate
Theorem~\ref{thm:Spdelta} for weak Schatten--von Neumann classes.
In this setting the endpoints for the intervals
of admissible values of $p$ can be included in
both \eqref{eq:one1} and \eqref{eq:three1}; cf.\ \cite[Section~4.2]{BLL13delta}.
\end{remark}

\begin{remark}			
In the case of a real, bounded coefficient $\alpha$, in space dimensions 2 and 3 the
previous theorem can be used in order to derive existence and completeness
of wave operators for the scattering pair $\{A_{[\alpha]}, -\Delta_{\dR^n}\}$.
In space dimension 2, the same is true for certain unbounded~$\alpha$;
cf.\ Example~\ref{ex:unboundedalpha}.
Let us also mention \cite{MPS16} where Schatten--von Neumann properties
were proved for certain $\delta$-interactions with unbounded real-valued coefficients.
\end{remark}

Finally, in the last theorem of this section we show that in two space dimensions
for $\|\alpha\|_\infty$ small enough the spectrum of $A_{[\alpha]}$
outside $[0,\infty)$ is contained in a disc with radius that
converges to $0$ exponentially as $\|\alpha\|_\infty\to0$ and that
in higher dimensions $A_{[\alpha]}$ has no spectrum
outside $[0,\infty)$ if $\|\alpha\|_\infty$ is small enough.
The result in two dimensions agrees well with the asymptotic expansion
in \cite{KL14} in the self-adjoint setting.
Related conditions for absence of non-real eigenvalues in higher dimensions
for Schr\"odinger operators with complex-valued regular potentials can be found in \cite{FKV17, F11}.
In the self-adjoint setting absence of negative eigenvalues for  $\|\alpha\|_\infty$
small enough is also a consequence of the Birman--Schwinger bounds in~\cite{BEKS94};
see also~\cite{EF09}.

\begin{theorem}
Let $\Sigma$ be compact and let $\alpha \in L^\infty(\Sigma)$ be a
complex-valued function that satisfies~\eqref{eq:alphaCond2}.
Then $\sess(A_{[\alpha]})=[0,\infty)$, and the following statements hold.
\begin{myenum}
\item % -----
Let $n=2$ and let $C_1 > 0$ be as in Theorem~\ref{thm:specEstdelta2}.
If\, $0 < \|\alpha\|_\infty\le \frac{1}{2C_1\ln 2}$, then
\[
  \sigma(A_{[\alpha]})\setminus \dR_+
  \subset \biggl\{z\in\CC: |z| \le 2\exp\biggl(-\frac{1}{C_1\|\alpha\|_\infty}\biggr)\biggr\}.
\]
\item % -----
Let $n\ge3$.  There exists $\eps = \eps(\Sigma) > 0$ such that
$\sigma(A_{[\alpha]}) = \sigma(-\Delta_{\dR^n}) = [0,\infty)$
if\, $\|\alpha\|_\infty<\eps$.
\end{myenum}
\end{theorem}

\begin{proof}
The statement about the essential spectrum follows
directly from Theorem~\ref{thm:Spdelta}.

(i)
Assume that $0 < \|\alpha\|_\infty\le \frac{1}{2C_1\ln 2}$ and
let $z\in\sigma(A_{[\alpha]})\setminus[0,\infty)$.
It follows from Theorem~\ref{thm:specEstdelta2} that $z\in\dV_{\alpha,C_1}$ and hence
\[
  C_1\|\alpha\|_\infty\ln\bigl(2+|z|^{-1}\bigr)
  \ge C_1\|\alpha\|_\infty\bigl(2+|z|\bigr)^{-1/4}\ln\bigl(2+|z|^{-1}\bigr)
  \ge 1,
\]
which implies that
\[
  |z| \le \frac{1}{\exp\bigl(\frac{1}{C_1\|\alpha\|_\infty}\bigr)-2}
  = \frac{\exp\bigl(-\frac{1}{C_1\|\alpha\|_\infty}\bigr)}{1-2\exp\bigl(-\frac{1}{C_1\|\alpha\|_\infty}\bigr)}
  \le 2\exp\biggl(-\frac{1}{C_1\|\alpha\|_\infty}\biggr).
\]

(ii)
Since the maximum of the function $g(t)=t^{-1/4}\ln t$, $t\in[2,\infty)$ is $\frac{4}{e}$
(attained at $t=e^4$), it follows that
\[
  \bigl(2+|z|\bigr)^{-1/4}\ln\bigl(2+|z|\bigr) \le \frac{4}{e}
  \qquad\text{for all}\;\; z\in\CC.
\]
If
\[
  \|\alpha\|_\infty< \eps \defeq \frac{e}{4C_1}\,,
\]
then
\[
  C_1\|\alpha\|_\infty\bigl(2+|z|\bigr)^{-1/4}\ln\bigl(2+|z|\bigr) < 1
\]
for every $z\in\CC$, and Theorem~\ref{thm:specEstdelta2}
implies that $\sigma(A_{[\alpha]})\setminus[0,\infty)=\emptyset$.
Together with the relation $\sess(A_{[\alpha]})=[0,\infty)$
this shows that $\sigma(A_{[\alpha]})=[0,\infty)$.
\end{proof}

% *******************************************************************
\section{Infinitely many point interactions on the real line}
\label{sec:pointinteractions}
% *******************************************************************

\noindent
In this section we provide applications of the results in
Section~\ref{sec:consequences}
to Hamiltonians with non-local, non-Hermitian
interactions supported on a discrete set of points $X=\{x_n:n\in\ZZ\}$,
where $(x_n)$ is a strictly increasing sequence of real numbers.
The investigation of such Hamiltonians has been initiated
almost a century ago in~\cite{KP31} for periodically distributed, local,
Hermitian point $\delta$-interactions.
Classical results are summarized in the monograph~\cite{AGHH};
see also the references therein and~\cite{K89, KM13}.
More recently, non-Hermitian interactions attracted attention (see~\cite{AFK02,AN07})
and also non-local interactions were studied; see~\cite{AN07, KZ16}.

Throughout this section we make the assumption
\begin{equation}\label{eq:d}
  d \defeq \inf_{n\in\dZ} (x_{n+1} - x_n) > 0;
\end{equation}
in particular, the sequence $(x_n)$ does not have a finite accumulation point.
We remark that this assumption can be avoided by using the methods of~\cite{AKM10, KM10},
but we do not focus on this here.

For each interval $I_n \defeq (x_n,x_{n+1})$ we denote by $H^2(I_n)$ the
usual Sobolev space on $I_n$ of second order.
Moreover, we set $f_n \defeq f|_{I_n}$ for $f\in L^2(\dR)$ and introduce
\[
  H^2(\R \setminus X) \defeq \bigg\{f\in L^2(\R):
  f_n\in H^2(I_n)\;\;\text{for all}\;\; n\in\Z,\;
  \sum_{n\in\Z} \|f_n\|_{H^2(I_n)}^2<\infty\bigg\},
\]
equipped with the norm
\begin{equation}\label{normformula}
  \|f\|_{H^2(\dR\setminus X)}^2
  \defeq \sum_{n\in\dZ} \|f_n\|_{H^2(I_n)}^2, \qquad f \in H^2(\R \setminus X).
\end{equation}

In order to construct a boundary triple which is suitable for the parameterization
of Hamiltonians with interactions supported on $X$,
we define operators $S$ and $T$ in $L^2(\R)$ by
\begin{equation}\label{eq:Sline}
  S f = -f'' \quad \text{on}\;\;\R \setminus X, \quad
  \dom S = \big\{f\in H^2(\dR): f(x_n) = 0 \;\;\text{for all}\;\; n\in\dZ \big\},
\end{equation}
and
\begin{equation}\label{eq:Tline}
  T f = - f'' \quad \text{on}\;\;\R \setminus X, \qquad
  \dom T = H^2(\R \setminus X) \cap H^1(\R),
\end{equation}
that is, $\dom T$ consists of all $f \in H^2(\R \setminus X)$
such that $f_{n-1}(x_n) = f_n(x_n)$ for all $n\in\dZ$.
Moreover, for $f \in \dom T$ we define
\begin{equation}\label{eq:Gamma01line}
  \Gamma_0 f = \bigl(f_n'(x_n)- f_{n - 1}'(x_n)\bigr)_{n\in\dZ} \qquad\text{and}\qquad
  \Gamma_1 f = \bigl(-f(x_n)\bigr)_{n\in\dZ}.
\end{equation}
In fact, $\Gamma_0$ and $\Gamma_1$ are boundary mappings for an
ordinary boundary triple, as the following proposition shows;
see also \cite[Proposition~7\,(i)]{K13} where a very similar boundary triple
was constructed.

\begin{proposition}\label{prop:BTline}
The operator $S$ in \eqref{eq:Sline} is closed, symmetric and densely defined
with $S^*=T$ for $T$ in \eqref{eq:Tline}, and the triple $\{\ell^2(\Z),\Gamma_0,\Gamma_1\}$
is an ordinary boundary triple for $S^*$ with the following properties.
\begin{myenum}
\item
  $A_0 = S^* \upharpoonright \ker \Gamma_0$ is given by
  \begin{equation}\label{eq:A0line}
    A_0 f = -f'', \quad \dom A_0 = H^2(\dR),
  \end{equation}
  and $A_1 = S^* \upharpoonright \ker \Gamma_1$ is given by
  \begin{equation}\label{eq:A1line}
	\begin{split}
    A_1f &= - f'' \quad \text{on}\;\;\dR \setminus X,
    \\[0.5ex]
    \dom A_1 &= \bigl\{ f \in H^2(\dR \setminus X) \cap H^1(\dR) :
    f(x_n) = 0\;\;\text{for all}\;\; n \in \Z \bigr\}.
	\end{split}
  \end{equation}
\item
  For $\lambda \in \C \setminus \R_+$ the associated $\gamma$-field
  acts as
  \begin{equation}\label{eq:gammaline}
    \bigl(\gamma(\lambda)\xi\bigr)(x) = \frac{-i}{2\sqrt{\lambda}\,}
    \sum_{n \in \Z} e^{i \sqrt{\lambda} |x_n - x|} \xi_n, \qquad
    x \in \R, \quad \xi = (\xi_n) \in \ell^2(\Z),
  \end{equation}
  and the associated Weyl function satisfies
  \begin{equation}\label{eq:Mline}
    M(\lambda)\xi = \left(\frac{i}{2\sqrt{\lambda}\,}
    \sum_{n \in \Z} e^{i\sqrt{\lambda}|x_n - x_m|} \xi_n \right)_{m \in \Z}, \qquad
    \xi = (\xi_n) \in \ell^2(\Z).
  \end{equation}
\end{myenum}
\end{proposition}

\begin{proof}
Let us first check that $\Gamma_0$ and $\Gamma_1$ are well-defined mappings
from $\dom T$ to $\ell^2(\Z)$.  For this we make use of the following estimate,
which can be found in, e.g.\ \cite[Lemma~8]{K04}:
if $[a, b]$ is a compact interval then for each $l \in (0, b - a]$ one has
\begin{equation}\label{eq:LSV}
  |f(a)|^2 \leq \frac{2}{l} \|f\|_{L^2(a,b)}^2 + l \|f'\|_{L^2(a,b)}^2 \quad
  \text{for all}\;\;f \in H^1(a,b).
\end{equation}
The same estimate holds for $|f(a)|^2$ replaced by $|f(b)|^2$.
From \eqref{eq:d} we obtain that $d \in (0, x_{n + 1} - x_n]$
for each $n \in \Z$, and \eqref{eq:LSV} yields
\[
  \sum_{n \in \Z} |f(x_n)|^2
  \le \frac{2}{d} \sum_{n \in \Z} \|f_n\|_{L^2(I_n)}^2
  + d \sum_{n \in \Z} \|f_n'\|_{L^2(I_n)}^2 < \infty
\]
for all $f \in \dom T \subset H^1(\R)$.
Hence $\Gamma_1 f \in \ell^2(\Z)$ for all $f \in \dom T$.
Similarly, using~\eqref{eq:LSV} for $f$ replaced by $f'$
we get $\Gamma_0 f \in \ell^2(\Z)$ for all $f \in \dom T$.

To show that $\{\ell^2(\Z), \Gamma_0, \Gamma_1\}$ is a boundary triple for $S^*$,
let us verify the conditions of Theorem~\ref{thm:ratetheorem}.
In fact, it is clear that $T \upharpoonright \ker \Gamma_0$ is given by
the operator $A_0$ in~\eqref{eq:A0line}, which is self-adjoint.
Moreover, for all $f, g \in \dom T$ we have
\begin{align*}
  & (Tf,g)_{L^2(\dR)} - (f,Tg)_{L^2(\dR)}
  = \sum_{n\in\dZ} \Bigl((-f_n'',g_n)_{L^2(I_n)} - (f_n,- g_n'')_{L^2(I_n)}\Bigr)
  \\[0.5ex]
  &= \sum_{n\in\dZ} \Big(f_n'(x_n) \ov{g(x_n)} - f_n'(x_{n+1}) \ov{g(x_{n+1})} \Big)
  \\[0.5ex]
  &\quad - \sum_{n\in\dZ} \Bigl(f(x_n) \ov{g_n'(x_{n})} - f(x_{n + 1}) \ov{g_n'(x_{n+1})}\Bigr)
  \displaybreak[0]\\[0.5ex]
  &= \sum_{n\in\dZ} \bigl(f_n'(x_n) - f_{n - 1}'(x_n)\bigr)\ov{g(x_n)}
  - \sum_{n\in\dZ} f(x_n) \ov{\bigl(g_n'(x_n) - g_{n - 1}'(x_n)\bigr)}
  \\[0.5ex]
  &= (\Gamma_1f,\Gamma_0g)_{\ell^2(\dZ)} - (\Gamma_0f,\Gamma_1g)_{\ell^2(\dZ)}.
\end{align*}
Furthermore, the pair of mappings
$(\Gamma_0, \Gamma_1)^\top : \dom T \to \ell^2(\Z) \times \ell^2(\Z)$ has a
dense range since it can be checked easily that all pairs of
unit sequences $\{e_j, e_k\}$, $j, k \in \Z$, belong to the range.
It follows from Theorem~\ref{thm:ratetheorem} that $S$ is closed
with $S^* = \overline T$ and that $\{\ell^2(\Z), \Gamma_0, \Gamma_1\}$
is a quasi boundary triple for $S^*$.

In order to conclude that $\{\ell^2(\Z), \Gamma_0, \Gamma_1\}$ is even an
ordinary boundary triple, let us verify that the operator $T$ is closed.
To this end define a mapping
\[
  K : H^2(\R \setminus X) \to \ell^2(\Z), \qquad
  f \mapsto \big(f_n(x_n) - f_{n - 1}(x_n) \big)_{n \in \Z}.
\]
For all $f \in H^2(\R \setminus X)$ we have
\begin{align*}
  & \|K f\|_{\ell^2(\Z)}^2
  \le 2\sum_{n \in \Z} \big( |f_n(x_n)|^2
  + |f_{n - 1}(x_n)|^2 \big)
  \\[0.5ex]
  &\le 2\sum_{n \in \Z} \biggl(\frac{2}{d} \|f_n\|_{L^2(I_n)}^2
  + d\|f_n'\|_{L^2(I_n)}^2
  + \frac{2}{d} \|f_{n-1}\|_{L^2(I_{n-1})}^2
  + d \|f_{n-1}'\|_{L^2(I_{n-1})}^2\biggr)
  \\[0.5ex]
  &\le 2 \max\biggl\{\frac{4}{d}, 2 d\biggr\} \sum_{n \in \Z} \|f_n\|_{H^2(I_n)}^2,
\end{align*}
where we have used~\eqref{eq:LSV} with $l = d$.  Therefore $K$ is a bounded operator and,
hence, its kernel, which equals $\dom T$, is closed in $H^2(\R \setminus X)$.
Equivalently, $\dom T$ equipped with the norm of $H^2(\R \setminus X)$ is complete.
It follows from \cite[Satz~6.24]{Wei}, its proof and \eqref{eq:d} that
for each $\eps > 0$ there exists $C(\eps) > 0$ such that for all $n \in \Z$
one has
\[
  \|f_n'\|_{L^2(I_n)}^2 \le \eps \|f_n''\|_{L^2(I_n)}^2 + C(\eps) \|f_n\|_{L^2(I_n)}^2, \qquad
  f_n \in H^2(I_n).
\]
This implies that $\dom T$ is also complete when equipped
with the graph norm of $T$, that is, $T$ is a closed operator.
Hence $\{\ell^2(\Z), \Gamma_0, \Gamma_1\}$ is an ordinary boundary triple for $S^*$.

The remaining assertion \eqref{eq:A1line} in (i) is obvious.
For the assertions in (ii) let $\lambda \in \C \setminus [0, \infty)$.
According to \cite[Satz 11.26]{Wei} or \cite[page~190]{T} we have
\[
  \bigl((A_0 - \lambda)^{-1}f\bigr)(y) = \frac{i}{2\sqrt{\lambda}\,}
  \int_{\dR} e^{i\sqrt{\lambda} |y - x|} f(x)\,\drm x, \qquad y \in \R,\; f \in L^2(\R).
\]
Hence for each compactly supported $f\in L^2(\dR)$ and
each $\xi = \{\xi_n\}_n \in \ell^2(\dZ)$ we obtain from \eqref{gammastar}
and the definition of $\Gamma_1$ that
\begin{align*}
  \bigl(f,\gamma(\lambda)\xi\bigr)_{L^2(\dR)}
  &= \bigl(\gamma(\lambda)^*f,\xi\bigr)_{\ell^2(\dZ)}
  = \bigl(\Gamma_1(A_0-\ov\lambda)^{-1}f,\xi\bigr)_{\ell^2(\dZ)}
  \\[0.5ex]
  &= \sum_{n\in\dZ} \bigg(-\frac{i}{2 \sqrt{\ov\lambda\,}\,}
  \int_{\dR} e^{i \sqrt{\overline{\lambda}\,}|x_n - x|} f(x)\,\drm x \bigg)\ov{\xi_n}
  \\[0.5ex]
  &= \int_{\dR} f(x) \ov{\bigg(\frac{-i}{2\sqrt{\lambda}\,}
  \sum_{n\in\dZ} e^{i \sqrt{\lambda} |x_n - x|} \xi_n \bigg)}\,\drm x,
\end{align*}
where we have used that $\overline{i \sqrt{\lambda}} = i \sqrt{\overline{\lambda}\,}$.
This proves \eqref{eq:gammaline}.
With the definition of $\Gamma_1$ also relation \eqref{eq:Mline} follows.
\end{proof}

Next we use the representation of the Weyl function in~\eqref{eq:Mline}
to estimate its norm.

\begin{lemma}\label{lem:WeylDecayLine}
The Weyl function associated with the boundary triple in
Proposition~\ref{prop:BTline} satisfies
\begin{equation}\label{eq:WeylEstLine}
  \|M(\lambda)\|\le \frac{\coth\bigl(\frac{d}{2}\Im\sqrt{\lambda}\,\bigr)}{2\sqrt{|\lambda|}}
\end{equation}
for all $\lambda \in \C \setminus [0, \infty)$.
In particular, the following estimates hold.
\begin{myenum}
\item
  For each $\mu < 0$,
  \[
    \|M(\lambda)\| \le \frac{\coth \bigl(\frac{d}{2}\sqrt{-\mu}\,\bigr)}{2(\mu-\lambda)^{1/2}}
    \qquad \text{for all}\;\; \lambda < \mu.
  \]
\item
  For each $w_0 < 0$ and each $\nu \in (0, \pi)$ we have
  \begin{equation}\label{eq:WeylEstLine2}
    \|M(\lambda)\| \le \frac{\coth(J_0)}{2\sqrt{|\lambda|}}
    \qquad \text{for~all}~\lambda \in \bbU_{w_0, \nu},
  \end{equation}
  where $J_0 = J_0(w_0, \nu) \defeq \frac{d}{2}\sqrt{|w_0|\sin\nu}\,\sin\bigl(\frac{\nu}{2}\bigr) > 0$
  and $\bbU_{w_0,\nu}$ is defined in \eqref{defUw0nu}.
\end{myenum}
\end{lemma}

\begin{proof}
Recall that for $\lambda \in \C \setminus [0, \infty)$ the
operator $M(\lambda)$ has the explicit representation~\eqref{eq:Mline}.
In order to estimate its norm, we make use of the Schur test;
see, e.g.\ \cite[Korollar~6.7]{Wei}.
For this note that $|x_n - x_m| \ge |n - m| d$ holds for all $n, m \in \Z$ and, thus,
\begin{align*}
  \sup_{m \in \Z} \sum_{n \in \Z} \left| e^{i\sqrt{\lambda} |x_n - x_m|} \right|
  &= \sup_{m \in \Z} \sum_{n \in \Z} e^{-\Im\sqrt{\lambda} |x_n - x_m|}
  \le \sup_{m \in \Z} \sum_{n \in \Z} e^{- \Im\sqrt{\lambda} |n - m| d} \\[0.5ex]
  &= \sum_{n \in \Z} e^{- \Im\sqrt{\lambda} |n| d}
  = \frac{1 + e^{- \imag \sqrt{\lambda} d}}{1 - e^{-\Im\sqrt{\lambda} d}}
  = \coth\biggl(\frac{d}{2}\Im\sqrt{\lambda}\,\biggr).
\end{align*}
Since the last term is finite and the same estimate holds by symmetry
when the roles of $m$ and $n$ are interchanged, the Schur test can be
applied and yields~\eqref{eq:WeylEstLine}.

The statement (i) is a direct consequence of the estimate in \eqref{eq:WeylEstLine}
and the monotonicity properties of the function $\coth$.
For the remaining statement (ii) we calculate
\begin{equation}\label{eq:J0}
  J_0 = J_0(w_0, \nu) \defeq \frac{d}{2} \min\bigl\{\imag\sqrt{\lambda}:
  \lambda \in \bbU_{w_0, \nu} \bigr\}.
\end{equation}
By symmetry it is clear that it suffices to consider
$\lambda \in \bbU_{w_0, \nu}$ with $\Im\lambda \geq 0$.
Since the function $\dC\setminus\{0\}\ni\lambda\mapsto \Im\sqrt{\lambda}$
has no local extremum, the minimum in \eqref{eq:J0} will be attained on the
boundary of~$\bbU_{w_0, \nu}$.  Let us first consider the
case when $\nu \in (0, \pi/2)$.  Writing $\lambda = x + i y$ with $x, y \in \R$,
for $\lambda \in \partial \bbU_{w_0, \nu}$ with $\Im\lambda \ge 0$ we have
\begin{equation}\label{eq:imCalc}
\begin{split}
  \imag \sqrt{\lambda} &= \imag \sqrt{x + i y}
  \ge \sqrt[4]{x^2 + y^2}\, \sin\Bigl(\frac{\nu}{2}\Bigr)
  \\[1ex]
  &= \sqrt[4]{x^2 + \tan^2 \nu \cdot (x - w_0)^2}\, \sin\Bigl(\frac{\nu}{2}\Bigr),
\end{split}
\end{equation}
and the right-hand side will be minimal if and only if
$x^2 + \tan^2 \nu \cdot (x - w_0)^2$ is minimal.
The latter happens for $x = (w_0 \tan^2 \nu) / (1 + \tan^2 \nu)$.
Plugging this into \eqref{eq:imCalc} and using elementary
trigonometric identities we obtain the claimed expression for $J_0$.
The case $\nu \in (\pi/2, \pi)$ can be treated analogously
with $\tan\nu$ replaced by $\tan(\pi-\nu)$, and for $\nu = \pi/2$ we have
\[
  \Im\sqrt{\lambda} \ge \sqrt[4]{w_0^2 + y^2} \sin\Bigl(\frac{\pi}{4}\Bigr)
  \ge \sqrt{|w_0|}\sin\Bigl(\frac{\pi}{4}\Bigr).
  \qedhere
\]
\end{proof}

\medskip

We are now able to formulate consequences of the results in
Section~\ref{sec:consequences}.  The assertions of the next theorem follow directly
from Lemma~\ref{lem:WeylDecayLine} in combination with~Corollary~\ref{obtcoraa},
Proposition~\ref{prop:specEst}\,(a), \cite[Proposition~1.4\,(i)]{DM95}
and the fact that $\{\ell^2(\Z),\Gamma_0,\Gamma_1\}$ is an ordinary boundary triple.

\begin{theorem}
Let $B$ be a closed operator in $\ell^2(\Z)$.
Then the operator $A_{[B]}$
\begin{equation}\label{eq:ABline}
\begin{split}
  A_{[B]} f &= - f'' \quad \text{on}~\R \setminus X, \\[0.5ex]
  \dom A_{[B]} &= \Big\{ f \in H^2(\R \setminus X) \cap H^1(\R) :
  \Gamma_0f =  B\Gamma_1f \Big\},
\end{split}
\end{equation}
in $L^2(\R)$ is closed, the resolvent formula
\[
  (A_{[B]} - \lambda)^{-1} = (A_0-\lambda)^{-1}
  + \gamma(\lambda)\big(I - BM(\lambda)\big)^{-1}B\gamma(\ov\lambda)^*
\]
holds for all $\lambda\in\rho(\AB)\cap\rho(A_0)$ and
the following assertions are true.
\begin{myenum}
\item
  If\, $B$ is self-adjoint, then $\AB$ is self-adjoint.
  If\, $B$ is maximal dissipative (maximal accumulative, respectively),
  then $\AB$ is maximal accumulative (maximal dissipative, respectively).
\item
  $A_{[B^*]} = \AB^*$.
\end{myenum}
Assume, additionally, that $B\in\cB(\ell^2(\dZ))$
and let $b\in\dR$ be such that
\[
  \Re(B\zeta,\zeta)_{\ell^2(\dZ)}\le b\|\zeta\|_{\ell^2(\dZ)}^2
  \qquad\text{for all}\;\; \zeta \in \ell^2(\dZ).
\]
Then the operator $\AB$ is m-sectorial;
in particular the inclusion $\sigma(\AB) \subset\ov{W(\AB)}$ holds,
and for any $\mu < 0$ and $C \defeq \frac12\coth(\frac{d}{2}\sqrt{-\mu})$
the following assertions are true.
\begin{myenuma}
\item % ----- (a)
If\, $b>0$, then for every $\xi < \mu-(Cb)^2$,
\[
  W(\AB) \subset \Bigl\{z \in \C : \Re z \ge \mu-(Cb)^2,\;
  |\Im z| \le K_\xi(\Re z-\xi)^{1/2}\Bigr\}, \hspace*{-3ex}
\]
where
\[
 K_\xi = \frac{2C\bigl\|\Im B\bigr\|}{1-\frac{Cb}{(\mu-\xi)^{1/2}}}\,.
\]
\item % ----- (b)
If\, $b\le0$, then
\[
  W(\AB) \subset \biggl\{z \in \C : \Re z \ge 0,\;
  |\Im z| \le \frac{2C\bigl\|\Im B\bigr\|(\Re z-\mu)}{(\Re z-\mu)^{1/2}-Cb}\biggr\}.
\]
\item
For any $w_0 < 0$
and each $\nu \in (0,\pi)$
\[
  \sigma(\AB) \cap \bbU_{w_0, \nu} \subset
  \Bigl\{z\in\bbU_{w_0,\nu}: \dist(z,\dR_+) \le \frac14\coth^2(J_0)\|B\|^2\Bigr\},
\]
where $\bbU_{w_0,\nu}$ is defined in \eqref{defUw0nu} and
$J_0 = \frac{d}{2} \sqrt{|w_0| \sin \nu}\, \sin(\frac{\nu}{2})$.
\end{myenuma}	
\end{theorem}

\medskip

\noindent
Finally, we remark that the class of Hamiltonians under consideration in this
section includes Schr\"odinger operators in $L^2(\RR)$ with local point
$\delta$-interactions supported on the set $X$, with possibly non-real coupling constants.
Such operators are obtained by choosing $B = \diag(\alpha_n)$
with $\alpha_n \in \dC$ for $n\in\dZ$.  The constant $\alpha_n$ can be viewed as
intensity (or strength) of the point $\delta$-interaction supported on $x_n$;
cf.\ \cite[Chapter~III.2]{AGHH}.

% *******************************************************************
\section{Quantum graphs with $\delta$-type vertex couplings}
\label{sec:graphs}
% *******************************************************************

\noindent
In this section we apply the results of the abstract part of this
paper to Laplacians on metric graphs. For a survey on this actively
developing field and references we refer the reader to the monograph~\cite{BK}
and the survey articles~\cite{B17,KS06,K04}.  In the present section we
consider the Laplacian on a finite, not necessarily compact metric graph,
equipped with $\delta$ or more general non-self-adjoint vertex couplings;
for further recent work on non-self-adjoint quantum graphs see~\cite{H14,HKS15,SSVW15}.
Furthermore, for the treatment of quantum graphs via boundary triples and
similar techniques we refer to, e.g.\ \cite{CKS17,EKK16,EK14,LSV14,P06,Post08}.

Let $G$ be a finite graph consisting of a finite set $V$ of vertices
and a finite set $E$ of edges, where we allow infinite edges,
i.e.\ edges `connecting a vertex to a point $\infty$'. Without loss of
generality we assume that there are no vertices of degree 0,
i.e.\ each vertex belongs to at least one edge, and that $G$ does not contain loops,
i.e.\ no edge connects a vertex to itself; this can always be achieved by
introducing additional vertices to the graph.
We equip each finite edge $e \in E$ with a length $L(e) > 0$ and
identify it with the interval $[0, L(e)]$.
Moreover, we identify each infinite edge with the interval $[0, \infty)$.
This identification gives rise to a natural metric on $G$ and to a natural $L^2$ space $L^2(G)$ on $G$.
For a vertex $v \in V$ and an edge $e \in E$ we write $v = o(e)$
or $v = t(e)$ if $e$ originates or terminates, respectively, at $v$,
and we occasionally simply write $v \sim e$ if one of these two properties holds.
For each vertex $v$ we denote by $\deg(v)$ the vertex degree,
that is, the number of edges which originate from or terminate at $v$.

In $\cH = L^2(G)$ we consider the Laplace differential expression
\[
  (- \Delta f)_e = - f_e'', \qquad e \in E,
\]
where $f_e$ denotes the restriction of $f$ to the edge $e \in E$.
In the following we write $\wt H^k(G) \defeq \bigoplus_{e \in E} H^k(0,L(e))$,
$k = 1,2,\dots$, for the orthogonal sum of the usual Sobolev spaces on the edges of $G$.
We say that a function $f \in \wt H^k(G)$ is continuous at a vertex $v$
whenever $v \sim e$ and $v \sim \hat e$ imply that the values of $f_e$
and $f_{\hat e}$ at $v$ coincide.  We define
\[
  H^1(G) \defeq \bigl\{f \in \wt H^1(G): f\;\;\text{is continuous at each}~v \in V \bigr\}.
\]
Note that for $f \in H^1(G)$ we can just write $f(v)$ for the evaluation of $f$
at a vertex~$v$.  For $f \in \widetilde H^2(G)$ and a vertex $v$ we write
\[
  \partial_\nu f(v) \defeq \sum_{t(e)=v} f_e'\bigl(L(e)\bigr) - \sum_{o(e)=v} f_e'(0).
\]
In order to construct an ordinary boundary triple let us consider the operators
\begin{equation}\label{eq:SGraph}
\begin{split}
 S f & = - \Delta f, \\
 \dom S & = \big\{ f \in H^1(G) \cap \widetilde H^2(G) : f(v) = \partial_\nu f(v) = 0~\text{for~all}~v \in V \big\},
\end{split}
\end{equation}
and
\begin{equation}\label{eq:TGraph}
  Tf = - \Delta f, \qquad \dom T = H^1(G) \cap \widetilde H^2(G),
\end{equation}
in $L^2(G)$.  Moreover, we choose an enumeration $V = \{v_1, \dots, v_{|V|}\}$
of the vertex set $V$ and define mappings
$\Gamma_0, \Gamma_1 : H^1(G) \cap \wt H^2(G) \to \C^{|V|}$ by
\[
  \begin{aligned}
    (\Gamma_0f)_j &= \partial_\nu f(v_j), \\[1ex]
    (\Gamma_1f)_j &= f(v_j),
  \end{aligned}
  \qquad j=1,\ldots,|V|,\;f\in\dom T.
\]

The mappings $\Gamma_0$ and $\Gamma_1$ give rise to an ordinary boundary triple
with finite-dimensional boundary space.  The following proposition is a
consequence of Theorem~\ref{thm:ratetheorem} and some elementary calculations.
It can also be derived from \cite[Lemma~2.14 and Theorem~2.16]{EK12}.
For the convenience of the reader we provide its proof below.

\begin{proposition}\label{prop:BTgraph}
The operator $S$ in \eqref{eq:SGraph} is closed, symmetric and densely defined
with $S^*=T$ for $T$ in \eqref{eq:TGraph}, and the triple $\{\CC^{|V|},\Gamma_0,\Gamma_1\}$
is an ordinary boundary triple for $S^*$ with the following properties.
\begin{myenum}
\item % ----------
$A_0 \defeq S^* \upharpoonright \ker \Gamma_0$ coincides with the standard
(or Kirchhoff) Laplacian
\begin{equation}\label{eq:standardLaplace}
\begin{split}
  - \Delta_G f & = - \Delta f, \\
  \dom(- \Delta_G) & = \Bigl\{ f \in H^1(G) \cap \widetilde H^2(G) :
  \partial_\nu f(v) = 0~\text{for all}~v \in V \Bigr\},
\end{split}
\end{equation}
and $A_1 \defeq S^* \upharpoonright \ker \Gamma_1$ coincides with the Dirichlet Laplacian
\begin{align*}
  - \Delta_{\rm D} f &= - \Delta f, \\
  \dom(- \Delta_{\rm D}) &= \Bigl\{ f \in H^1(G) \cap \widetilde H^2(G) :
  f(v) = 0~\text{for all}~v \in V \Bigr\}.
\end{align*}
In particular, $A_0$ and $A_1$ are both self-adjoint and non-negative operators in $L^2(G)$.
\item % ----------
For $\lambda \in \C \setminus \sigma(-\Delta_G)$, the corresponding $\gamma$-field is given by
\begin{equation}\label{eq:gammaGraph}
  \gamma(\lambda) \begin{pmatrix} \partial_\nu f(v_1) \\ \vdots \\ \partial_\nu f(v_{|V|}) \end{pmatrix}
  = f,
\end{equation}
where $f \in H^1(G) \cap \widetilde H^2(G)$ is any function that
satisfies $-\Delta f = \lambda f$, and the corresponding Weyl function is given by
\begin{equation}\label{eq:MGraph}
  M(\lambda) \begin{pmatrix} \partial_\nu f(v_1) \\ \vdots \\ \partial_\nu f(v_{|V|}) \end{pmatrix}
  = \begin{pmatrix} f(v_1) \\ \vdots \\ f(v_{|V|}) \end{pmatrix}.
\end{equation}
For each $\lambda \in \C \setminus\bigl(\sigma(-\Delta_G) \cup \sigma(-\Delta_{\rm D})\bigr)$
we have
\begin{equation}\label{eq:MInverseGraph}
  \bigl(M(\lambda)^{-1}\bigr)_{jk}
  = \begin{cases}
    \sqrt{\lambda}\!\! \sum\limits_{\substack{e \sim v_j \\[0.2ex] L(e) < \infty}}\!\!
    \cot\bigl(\sqrt{\lambda}L(e)\bigr) \\[4ex]
    \hspace*{10ex} - i\sqrt{\lambda}\,\big|\{e: o(e) = v_j, L(e) = \infty\}\big|,
    & j=k, \\[3ex]
    \sum\limits_{\substack{e \sim v_j, \\[0.2ex] e \sim v_k}}
    \dfrac{-\sqrt{\lambda}}{\sin\bigl(\sqrt{\lambda}L(e)\bigr)}\,, &
    j \ne k.
  \end{cases}
  \hspace*{-2ex}
\end{equation}
\end{myenum}
\end{proposition}

\begin{proof}
Let us verify the conditions of Theorem~\ref{thm:ratetheorem}.
Note first that $T \upharpoonright \ker \Gamma_0$ clearly equals the standard
Laplacian \eqref{eq:standardLaplace}, which is self-adjoint in $L^2(G)$.
Moreover, it can easily be seen by explicit construction that the
pair $(\Gamma_0, \Gamma_1)^{\top} : \dom T \to \C^{|V|} \times \C^{|V|}$ is surjective.
Finally, let us verify the abstract Green identity.
For $f, g \in \dom T$ integration by parts yields
\begin{align*}
  & (Tf,g)_{L^2(G)} - (f,Tg)_{L^2(G)}
  \\[1ex]
  &= \sum_{e \in E} \Biggl(\int_0^{L(e)}
  \bigl(-f_e''(x)\bigr)\ov{g_e(x)}\,\drm x
  - \int_0^{L(e)} f_e(x)\bigl(\ov{-g_e''(x)}\bigr)\,\drm x \Biggr)
  \\[1ex]
  &= \sum_{e \in E} \Biggl(\int_0^{L(e)} f_e'(x)\ov{g_e'(x)}\,\drm x
  - \int_0^{L(e)} f_e'(x)\ov{g_e'(x)}\,\drm x
  \\[1ex]
  &\quad + f_e'(0)\ov{g_e(0)} -  f_e'\bigl(L(e)\bigr)\ov{g_e\bigl(L(e)\bigr)}
  - f_e(0)\ov{g_e'(0)} + f_e\bigl(L(e)\bigr)\ov{g_e'\bigl(L(e)\bigr)}\Biggr)
  \displaybreak[0]\\[1ex]
  &= \sum_{j = 1}^{|V|} f(v_j) \ov{\biggl(\sum_{t(e) = v_{j}} g_e'\bigl(L(e)\bigr)
  - \sum_{o(e) = v_j} g_e'(0)\biggr)}
  \\[1ex]
  &\quad - \sum_{j = 1}^{|V|} \biggl(\sum_{t(e) = v_j} f_e'\bigl(L(e)\bigr)
  - \sum_{o(e) = v_j} f_e'(0)\biggr)\ov{g(v_j)}
  \\[1ex]
  &= (\Gamma_1 f, \Gamma_0 g)_{\CC^{|V}} - (\Gamma_0 f, \Gamma_1 g)_{\CC^{|V}}.
\end{align*}
From Theorem~\ref{thm:ratetheorem} it follows that $S$ is closed, densely defined
and symmetric with $S^* = T$ and that $\{\C^{|V|},\Gamma_0,\Gamma_1\}$
is an ordinary boundary triple for $T = S^*$.
Assertion (i) and the identities \eqref{eq:gammaGraph}, \eqref{eq:MGraph}
are obvious from the definition of the mappings $\Gamma_0, \Gamma_1$.

It remains to verify the representation of $M(\lambda)^{-1}$ in \eqref{eq:MInverseGraph}.
To this end fix $\lambda \in \C\setminus(\sigma(-\Delta_G) \cup \sigma(-\Delta_{\rm D}))$
and denote by $m^e(\lambda)$ the Dirichlet-to-Neumann map corresponding to the
equation $-f'' = \lambda f$ on the interval $[0, L(e)]$;
if $e$ is finite then $m^e(\lambda)$ is the matrix satisfying
\begin{equation}\label{eq:TWInterval}
\begin{split}
  \begin{pmatrix} f'(0) \\[1ex] - f'\bigl(L(e)\bigr) \end{pmatrix}
  &= \begin{pmatrix} m_{11}^e(\lambda) & m_{12}^e(\lambda) \\[1ex]
  m_{21}^e(\lambda) & m_{22}^e(\lambda) \end{pmatrix}
  \begin{pmatrix} f(0) \\[1ex] f\bigl(L(e)\bigr) \end{pmatrix}
  \\[1ex]
  &= \begin{pmatrix} m_{11}^e(\lambda)f(0) + m_{12}^e(\lambda)f\bigl(L(e)\bigr) \\[1ex]
  m_{21}^e(\lambda)f(0) + m_{22}^e(\lambda)f\bigl(L(e)\bigr) \end{pmatrix}
\end{split}
\end{equation}
for each $f \in H^2(0, L(e))$ with $-f'' = \lambda f$;
if $e$ is infinite then $m^e$ is the scalar function satisfying
\begin{equation}\label{eq:mScalar}
  f'(0) = m^e(\lambda)f(0)
\end{equation}
for each $f \in H^2(0,\infty)$ with $-f'' = \lambda f$.
Let us define the matrix $\Lambda(\lambda)$ by
\begin{equation}\label{eq:DNGraph}
  (\Lambda(\lambda))_{jk}
  = \begin{cases}
    \sum\limits_{\substack{o(e) = v_j \\[0.2ex] L(e) < \infty}} m_{11}^e(\lambda)
    + \sum\limits_{\substack{t(e) = v_j \\[0.2ex] L(e) < \infty}} m_{22}^e(\lambda)
    + \sum\limits_{\substack{o(e) = v_j \\[0.2ex] L(e) = \infty}} m^e(\lambda),
    & j = k,
    \\[5ex]
    \sum\limits_{\substack{o(e) = v_j \\[0.2ex] t(e) = v_k}} m_{12}^e(\lambda)
    + \sum\limits_{\substack{o(e) = v_k \\[0.2ex] t(e) = v_j}} m_{21}^e(\lambda),
    & j \ne k.
  \end{cases}
\end{equation}
We show that $\Lambda(\lambda) = - M(\lambda)^{-1}$.  Indeed, let $f \in \ker(T-\lambda)$.
Then for $j = 1, \dots, |V|$ we have
\begin{align*}
  \bigl(\Lambda(\lambda)\Gamma_1 f\bigr)_j
  &= \sum_{k=1}^{|V|} (\Lambda(\lambda))_{jk}f(v_k)
  \\
  &= \sum_{k \ne j} \bigg(\sum_{\substack{o(e) = v_j \\[0.2ex] t(e) = v_k}} m_{12}^e(\lambda)
  + \sum_{\substack{o(e) = v_k \\[0.2ex] t(e) = v_j}} m_{21}^e(\lambda)\bigg)f(v_k)
  \\
  &\quad + \bigg(\sum_{\substack{o(e) = v_j \\[0.2ex] L(e) < \infty}} m_{11}^e(\lambda)
  + \sum_{\substack{t(e) = v_j \\[0.2ex] L(e) < \infty}} m_{22}^e(\lambda)
  + \sum_{\substack{o(e) = v_j \\[0.2ex] L(e) = \infty}} m^e(\lambda)\bigg)f(v_j)
  \displaybreak[0]\\
  &= \sum_{k \ne j} \biggl(\sum_{\substack{o(e) = v_j \\[0.2ex] t(e) = v_k}}
  \Bigl(m_{11}^e(\lambda)f(v_j) + m_{12}^e(\lambda)f(v_k)\Bigr)
  \\
  &\quad + \sum_{\substack{o(e) = v_k \\[0.2ex] t(e) = v_j}}
  \Bigl( m_{21}^e(\lambda)f(v_k) + m_{22}^e(\lambda)f(v_j)\Bigr)\biggr)
  + \sum_{\substack{o(e) = v_j \\[0.2ex] L(e) = \infty}} m^e(\lambda)f(v_j),
\end{align*}
where we have used that $G$ does not contain loops.
Taking~\eqref{eq:TWInterval} and \eqref{eq:mScalar} into account we obtain that
\begin{align*}
  \bigl(\Lambda(\lambda)\Gamma_1 f\bigr)_j
  &= \sum_{k \ne j} \biggl(\sum_{\substack{o(e) = v_j \\[0.2ex] t(e) = v_k}}f_e'(0)
  - \sum_{\substack{o(e) = v_k \\[0.2ex] t(e) = v_j}}f_e'\bigl(L(e)\bigr)\biggr)
  + \sum_{\substack{o(e) = v_j \\[0.2ex] L(e) = \infty}}f_e'(0)
  \\
  &= \sum_{o(e) = v_j} f_e'(0) - \sum_{t(e) = v_j} f_e'\bigl(L(e)\bigr)
  = - (\Gamma_0 f)_j,
\end{align*}
which implies that $\Lambda(\lambda) = - M(\lambda)^{-1}$.
Note that $m^e$ can be calculated explicitly and is given by the expressions
\[
  m^e(\lambda)
  = \begin{cases}
    \dfrac{\sqrt{\lambda}}{\sin\bigl(\sqrt{\lambda}L(e)\bigr)}
    \begin{pmatrix}
      - \cos\bigl(\sqrt{\lambda}L(e)\bigr) & 1 \\[1ex]
      1 & - \cos\bigl(\sqrt{\lambda} L(e)\bigr)
    \end{pmatrix}
    & \text{if}~L(e) < \infty,
    \\[4ex]
    i \sqrt{\lambda} & \text{if}~L(e) = \infty.
  \end{cases}
\]
Plugging these representations into~\eqref{eq:DNGraph} we arrive
at~\eqref{eq:MInverseGraph}.
\end{proof}

The next lemma provides a decay property of the Weyl function.

\begin{lemma}\label{lem:WeylDecayGraph}
Let $M$ be the Weyl function corresponding to the boundary triple
in Proposition~\ref{prop:BTgraph}.  Then for each $w_0 < 0$
and $\nu\in(0,\pi)$ there exists $C=C(w_0,\nu)>0$ such that
\begin{equation}\label{eq:WeylAsymptGraphAllg}
  \|M(\lambda)\| \le \frac{C}{\sqrt{|\lambda|}\,}
  \qquad\text{for all}\;\;\lambda \in \bbU_{w_0,\nu},
\end{equation}
where $\bbU_{w_0,\nu}$ is defined in \eqref{defUw0nu}.
\end{lemma}

\begin{proof}
Let $w_0 < 0$ and $\nu\in(0,\pi)$.  If $|\lambda|\to\infty$
for $\lambda \in \bbU_{w_0,\nu}$, then $\sqrt{\lambda}\to\infty$ within
the sector $\{re^{i\phi}: r>0, \phi\in(\nu/2,\pi-\nu/2)\}$.
In particular, $\Im\sqrt{\lambda}$ tends to $+\infty$, and thus
\[
  -\cot\bigl(\sqrt{\lambda}L(e)\bigr) \to i \qquad \text{and} \qquad
  \frac{1}{\sin\bigl(\sqrt{\lambda}L(e)\bigr)} \to 0
\]
for all $e$ as $|\lambda| \to \infty$, and the convergence is uniform in $\bbU_{w_0, \nu}$.
Hence it follows from~\eqref{eq:MInverseGraph} that
\[
  M(\lambda)^{-1} \to - \sqrt{\lambda} \diag\bigl(\deg(v_1)i,\dots,\deg(v_{|V|})i\bigr)
\]
uniformly as $|\lambda| \to \infty$, $\lambda \in \bbU_{w_0,\nu}$.  It follows that
\begin{equation}\label{eq:convergence}
  M(\lambda)^{-1} \big(M(\lambda)^{-1} \big)^*
  \to |\lambda| \diag\bigl(\deg(v_1)^2,\dots,\deg(v_{|V|})^2\bigr)
\end{equation}
uniformly as $|\lambda| \to \infty$, $\lambda \in \bbU_{w_0,\nu}$.
Let $C_1 > 1$ be arbitrary.  Since the matrix $\diag(\deg(v_1)^2,\dots,\deg(v_{|V|})^2)$
is positive definite with smallest eigenvalue greater than or equal to 1,
it follows from \eqref{eq:convergence} that there exists $r_0 > 0$ such
that the smallest eigenvalue of $M(\lambda)^{-1} \big(M(\lambda)^{-1} \big)^*$
satisfies
\[
  \lambda_1 \Bigl(M(\lambda)^{-1}\bigl(M(\lambda)^{-1}\bigr)^*\Bigr)
  \ge \frac{|\lambda|}{C_1^2}
\]
for all $\lambda \in \bbU_{w_0, \nu}$ with $|\lambda| > r_0$.  Thus we obtain that
\begin{equation}\label{eq:firstEst}
  \|M(\lambda)\|
  = \frac{1}{\sqrt{\lambda_1 \bigl(M(\lambda)^{-1}\bigl(M(\lambda)^{-1}\bigr)^*\bigr)}\,}
   \le \frac{C_1}{\sqrt{|\lambda|}\,}
\end{equation}
for all $\lambda \in \bbU_{w_0, \nu}$ with $|\lambda| > r_0$.
On the other hand, since $\lambda \mapsto \sqrt{|\lambda|}\|M(\lambda)\|$ is
continuous on the compact set
\[
  \bbU_{w_0,\nu}^0 \defeq \bigl\{\lambda\in\bbU_{w_0, \nu}: |\lambda| \le r_0\bigr\},
\]
there exists $C_2 > 0$ with
\begin{equation}\label{eq:secondEst}
  \|M(\lambda)\| \leq \frac{C_2}{\sqrt{|\lambda|}\,}\,,
  \qquad \lambda \in \bbU_{w_0,\nu}^0.
\end{equation}
With $C \defeq \max \{ C_1, C_2\}$ the claim of the lemma follows
from the inequalities~\eqref{eq:firstEst} and~\eqref{eq:secondEst}.
\end{proof}

The assertions of the following theorem are direct consequences of
Proposition~\ref{prop:BTgraph}, Lemma~\ref{lem:WeylDecayGraph} and Corollary~\ref{obtcoraa}.
For characterizations of self-adjoint vertex conditions for Laplacians on
metric graphs we refer the reader to \cite{C98,KS99}.

\begin{theorem}\label{thm:graphs}
Let $B \in \C^{|V| \times |V|}$. Then the operator
\begin{equation}\label{eq:ABgraph}
\begin{split}
  A_{[B]} f & = - \Delta f, \\
 \dom A_{[B]} & = \left\{ f \in H^1(G) \cap \widetilde H^2(G):
 \begin{pmatrix} \partial_\nu f(v_1) \\ \vdots \\ \partial_\nu f(v_{|V|}) \end{pmatrix}
 = B \begin{pmatrix} f(v_1) \\ \vdots \\ f(v_{|V|}) \end{pmatrix} \right\},
\end{split}
\end{equation}
in $L^2(G)$ is m-sectorial, one has $\sigma(A_{[B]}) \subset \ov{W(A_{[B]})}$,
the resolvent formula
\[
  (A_{[B]} - \lambda)^{-1} = (-\Delta_G - \lambda)^{-1} + \gamma(\lambda)
  \bigl(I-BM(\lambda)\bigr)^{-1}B\gamma(\ov\lambda)^*
\]
holds for all $\lambda \in \rho(A_{[B]}) \cap \rho(-\Delta_G)$
and the following assertions are true.
\begin{myenum}
 \item $A_{[B]}$ is self-adjoint if and only if the matrix $B$ is Hermitian.
 Moreover, $A_{[B]}$ is maximal dissipative (maximal accumulative, respectively)
 if and only if $B$ is accumulative (dissipative, respectively).
 \item $A_{[B^*]} = A_{[B]}^*$.
\end{myenum}
Assume in addition that $b \in \R$ is chosen such that
\[
  \Re(B \xi, \xi) \le b |\xi|^2 \qquad \text{for all}\;\; \xi \in \C^{|V|}.
\]
Then the following spectral enclosures hold.
\begin{myenuma}
\item
If $b > 0$ then there exists $C > 0$ such that for each $\xi < - (C b)^2$
\begin{equation*}
  W(A_{[B]})
  \subset \Biggl\{z\in\CC: \Re z \ge \xi,\;
  |\Im z| \le \frac{2C\|\Im B\|}{1-\frac{Cb}{(-\xi)^{1/2}}}\bigl(\Re z - \xi\bigr)^{1/2}\Biggr\}.
\end{equation*}
\item
If $b \le 0$ then there exists $C > 0$ such that
\[
  W(A_{[B]})
  \subset \biggl\{z\in\CC: \Re z \ge 0,\;
  |\Im z| \le \frac{2C\|\Im B\|(\Re z)}{(\Re z)^{1/2}-Cb}\biggr\}.
\]
\item
For each $w_0 < \min\sigma(A_{\rm N})$ and $\nu \in (0, \pi)$
there exists $C > 0$ such that
\[
  \sigma(A_{[B]}) \cap \bbU_{w_0, \nu}
  \subset \bigl\{z \in \bbU_{w_0,\nu}: |z| \le (C\|B\|)^2\bigr\},
\]
where $\bbU_{w_0,\nu}$ is defined in \eqref{defUw0nu}.
\end{myenuma}
\end{theorem}

\begin{remark}
Note that the operator $A_{[B]}$ satisfies local matching conditions at all
vertices if and only if the matrix $B$ is diagonal, $B = \diag(b_1, \dots b_{|V|})$.
In this case $\dom A_{[B]}$ consists of all functions $f \in H^1(G) \cap \wt H^2(G)$
such that
\[
  \partial_\nu f(v_j) = b_j f(v_j)
\]
holds for $j = 1, \dots, |V|$.  These conditions describe $\delta$-couplings
of strengths $b_j$.  They have been studied extensively in the literature in
the self-adjoint case, i.e.\ for real $b_1, \dots, b_{|V|}$;
see, e.g.\ \cite{BK,EK12,EJ12,KKT16,K04}.
\end{remark}

\begin{remark}
In more specific situations the spectral estimates in Theorem~\ref{thm:graphs}
can be made more explicit.  Let, for instance, $G$ be combinatorially equal
to the complete graph $K_n$ with $n = |V| \geq 2$ vertices, that is,
each two vertices are connected by precisely one edge;
in particular, $\deg(v_j) = n - 1$ for $j = 1, \dots, |V|$.
Moreover, let $G$ be equilateral with $L(e) = 1$ for all $e \in E$.
It follows from \eqref{eq:MInverseGraph} that the Weyl function $M$ corresponding
to the boundary triple in Proposition~\ref{prop:BTgraph} satisfies
\[
  \bigl(M(\lambda)\bigr)^{-1} = \frac{\sqrt{\lambda}}{\sin\sqrt{\lambda}\,}
  \begin{pmatrix}
    (n-1)\cos\sqrt{\lambda} & -1 & \cdots & -1 \\[0.5ex]
    -1  & \ddots & \ddots & \vdots \\[0.5ex]
    \vdots & \ddots & & -1 \\[0.5ex]
    -1 & \dots & -1 & (n-1)\cos\sqrt{\lambda}
  \end{pmatrix}.
\]
A straightforward calculation yields that $M$ is given by
\[
  M(\lambda) = \frac{1}{\alpha(n, \lambda)}
  \begin{pmatrix}
    d(n, \lambda) & 1 & \cdots & 1 \\
    1  & \ddots & \ddots & \vdots \\
    \vdots & \ddots & & 1 \\
    1 & \dots & 1 & d(n, \lambda)
  \end{pmatrix},
\]
where
\begin{align*}
  \alpha(n,\lambda) &= \frac{\sqrt{\lambda}}{\sin\sqrt{\lambda}\,}
  \biggl[\biggl((n-1)\cos\sqrt{\lambda} - \frac{n-2}{2}\biggr)^2 - \frac{n^2}{4}\biggr],
  \\[1ex]
  d(n,\lambda) &= (n-1)\cos\sqrt{\lambda} - (n-2).
\end{align*}
Since in this case $M(\lambda)$ is a special case of a circulant matrix,
its norm can be calculated and estimated explicitly for $\lambda \in \dU_{w_0, \nu}$.
\end{remark}

The following example shows that the abstract spectral estimate in
Corollary~\ref{cor:ganzTolleWeylfunktion} cannot be improved in general.

\begin{example}
Let $G$ be a star graph consisting of $|E|$ infinite edges,
i.e.\ each edge of $G$ can be parameterized by the interval $[0, \infty)$
and there exists only one vertex $v$, which satisfies $o(e) = v$ for all $e \in E$.
Then for $B \in \C$ the functions in the domain of the operator $A_{[B]}$
in \eqref{eq:ABgraph} are continuous at $v$ and satisfy the condition
\[
  - \sum_{e \in E} f_e'(0) = Bf(v).
\]
If $B \notin \dR$ with $\Re B > 0$ then $A_{[B]}$ has $- B^2 / |E|^2$ as its
only non-real eigenvalue, as an explicit calculation shows.
On the other hand, by Proposition~\ref{prop:BTgraph}\,(ii) we
obtain that $M(\lambda) = i |E| / \sqrt{\lambda}$
for all $\lambda \in \C \setminus \R$, and
Corollary~\ref{cor:ganzTolleWeylfunktion} yields that
\[
  \sigma(A_{[B]}) \cap \big( \C \setminus [0, \infty) \big)
  \subset \biggl\{z \in \C \setminus [0, \infty): |z| \le \frac{|B|^2}{|E|^2}\biggr\}.
\]
This shows that Corollary~\ref{cor:ganzTolleWeylfunktion} is sharp.
\end{example}

\section*{Acknowledgements}
\noindent
JB, VL and JR gratefully acknowledge financial support by
the Austrian Science Fund (FWF), grant no.~P~25162-N26.
VL acknowledges the support of the Czech Science Foundation (GA\v{C}R),
grant no.~17-01706S
and also of the support of the Austria--Czech Republic Mobility Programme,
grant no.~7AMB17AT022.

% *******************************************************************


\begin{thebibliography}{109}
% *******************************************************************

\bibitem{AGW14}
H.~Abels, G.~Grubb and I.~Wood,
Extension theory and Kre\u{\i}n-type resolvent formulas for nonsmooth boundary value problems,
\textit{J.\ Funct.\ Anal.} \textbf{266} (2014), 4037--4100.

\bibitem{AAD01}
A.\,A.~Abramov, A.~Aslanyan and E.\,B.~Davies,
Bounds on complex eigenvalues and resonances,
\textit{J.\ Phys.\ A} \textbf{34} (2001), 57--72.

\bibitem{ABN17}
M.~Adler, M.~Bombieri and K.-J.~Nagel,
Perturbation of analytic semigroups and applications to partial differential equations,
\textit{J.\ Evol.\ Equ.} \textbf{17} (2017), 1183--1208.

\bibitem{A62}
S.~Agmon,
On the eigenfunctions and on the eigenvalues of general elliptic boundary value problems,
\textit{Commun.\ Pure Appl.\ Math.} \textbf{15} (1962), 119--147.

\bibitem{ABMN05}	
S.~Albeverio, J.\,F.~Brasche, M.\,M.~Malamud and H.~Neidhardt,
Inverse spectral theory for symmetric operators with several gaps: scalar-type Weyl functions,
\textit{J.\ Funct.\ Anal.} \textbf{228} (2005), 144--188.	

\bibitem{AFK02}
S.~Albeverio, S.-M.~Fei and P.~Kurasov,
Point interactions: $\cP\cT$-Hermiticity and reality of the spectrum,
\textit{Lett.\ Math.\ Phys.} \textbf{59} (2002),  227--242.

\bibitem{AGHH}
S.~Albeverio, F.~Gesztesy, R.~H{\o}egh-Krohn and H.~Holden,
\textit{Solvable Models in Quantum Mechanics. With an appendix by Pavel Exner.}
AMS Chelsea Publishing, 2005.

\bibitem{AKM10}
S.~Albeverio, A.~Kostenko and M.\,M.~Malamud,
Spectral theory of semibounded Sturm--Liouville operators with local
interactions on a discrete set,
\textit{J.\ Math.\ Phys.} \textbf{51} (2010), 102102, 24~pp.
	
\bibitem{AN07}
S.~Albeverio and L.~Nizhnik,
Schr\"{o}dinger operators with nonlocal point interactions,
\textit{J.\ Math.\ Anal.\ Appl.} \textbf{332} (2007),  884--895.

\bibitem{Ando78}
T.~Ando,
Topics on operator inequalities.
\textit{Division of Applied Mathematics, Research Institute of Applied Electricity,
Hokkaido University, Sapporo}, 1978.

\bibitem{AE11}
W.~Arendt and A.\,F.\,M.~ter Elst,
The Dirichlet-to-Neumann operator on rough domains,
\textit{J.\ Differential Equations} \textbf{251} (2011), 2100--2124.

\bibitem{AE15}
W.~Arendt and A.\,F.\,M.~ter Elst,
The Dirichlet-to-Neumann operator on exterior domains,
\textit{Potential Anal.} \textbf{43} (2015), 313--340.

\bibitem{AEKS14}
W.~Arendt, A.\,F.\,M.~ter Elst, J.\,B.~Kennedy and M.~Sauter,
The Dirichlet-to-Neumann operator via hidden compactness,
\textit{J.\ Funct.\ Anal.} \textbf{266} (2014), 1757--1786.

\bibitem{A00}
Yu.~Arlinskii,
Abstract boundary conditions for maximal sectorial extensions of sectorial operators,
\textit{Math.\ Nachr.} \textbf{209} (2000), 5--36.

\bibitem{A12}
Yu.~Arlinski\u{\i},
Boundary triplets and maximal accretive extensions of sectorial operators,
in: \textit{Operator Methods for Boundary Value Problems},
London Math.\ Soc.\ Lecture Note Series, vol.~404, 2012, pp.~35--72.

\bibitem{AKT12}
Yu.~Arlinski\u{\i}, Y.~Kovalev and E.~Tsekanovskii,
Accretive and sectorial extensions of nonnegative symmetric operators,
\textit{Complex Anal.\ Oper.\ Theory} \textbf{6} (2012), 677--718.

\bibitem{AP13}
Yu.~Arlinski\u{\i} and A.\,B.~Popov,
On m-accretive extensions of a sectorial operator. [Russian]
\textit{Mat.\ Sb.} \textbf{204} (2013), 3--40;
English translation: \textit{Sb.\ Math.} \textbf{204} (2013), 1085--1121.

\bibitem{AP17}
Yu.~Arlinski\u{\i} and A.~Popov,
On m-sectorial extensions of sectorial operators,
\textit{Zh.\ Mat.\ Fiz.\ Anal.\ Geom.} \textbf{13} (2017), 205--241.
%(Journal of Mathematical Physics, Analysis, Geometry \textbf{13} (2017), 205--241)

\bibitem{BF62}
W.~Bade and R.~Freeman,
Closed extensions of the Laplace operator determined by a general class of boundary conditions,
\textit{Pacific J.\ Math.} \textbf{12} (1962), 395--410.

\bibitem{B65}
R.~Beals,
Non-local boundary value problems for elliptic operators,
\textit{Amer.\ J.\ Math.} \textbf{87} (1965), 315--362.

\bibitem{BGLL15}
J.~Behrndt, G.~Grubb, M.~Langer and V.~Lotoreichik,
Spectral asymptotics for resolvent differences of elliptic operators with $\delta$
and $\delta'$-interactions on hypersurfaces,
\textit{J.\ Spectr.\ Theory} \textbf{5} (2015), 697--729.

\bibitem{BL07}
J.~Behrndt and M.~Langer,
Boundary value problems for elliptic partial differential operators on bounded domains,
\textit{J.\ Funct.\ Anal.} \textbf{243} (2007), 536--565.

\bibitem{BL10}
J.~Behrndt and M.~Langer,
On the adjoint of a symmetric operator,
\textit{J.\ London\ Math.\ Soc.\ (2)} \textbf{82} (2010), 563--580.

\bibitem{BL12}
J.~Behrndt and M.~Langer,
Elliptic operators, Dirichlet-to-Neumann maps and quasi boundary triples,
in: \textit{Operator Methods for Boundary Value Problems},
London Math.\ Soc.\ Lecture Note Series, vol.~404, 2012, pp.~121--160.

\bibitem{BLLLP10}
J.~Behrndt, M.~Langer, I.~Lobanov, V.~Lotoreichik and I.\,Yu.~Popov,
A remark on Schatten--von Neumann properties of resolvent differences of
generalized Robin Laplacians on bounded domains,
\textit{J.\ Math.\ Anal.\ Appl.} \textbf{371} (2010), 750--758.

\bibitem{BLL13delta}
J.~Behrndt, M.~Langer and V.~Lotoreichik,
Schr\"odinger operators with $\delta$ and $\delta'$-potentials
supported on hypersurfaces,
\textit{Ann.\ Henri Poincar\'e} \textbf{14} (2013), 385--423.

\bibitem{BLL13IEOT}
J.~Behrndt, M.~Langer and  V.~Lotoreichik,
Spectral estimates for resolvent differences of self-adjoint elliptic operators,
\textit{Integral Equations Operator Theory} \textbf{77} (2013), 1--37.

\bibitem{BLL13trace}
J.~Behrndt, M.~Langer and V.~Lotoreichik,
Trace formulae and singular values of resolvent power differences of
self-adjoint elliptic operators,
\textit{J.\ London Math.\ Soc.\ (2)} \textbf{88} (2013), 319--337.

\bibitem{BLLR17}
J.~Behrndt, M.~Langer, V.~Lotoreichik and J.~Rohleder,
Quasi boundary triples and semibounded self-adjoint extensions,
\textit{Proc.\ Roy.\ Soc.\ Edinburgh Sect.\ A} \textbf{147} (2017), 895--916.

\bibitem{BR15}
J.~Behrndt and J.~Rohleder,
Spectral analysis of selfadjoint elliptic differential operators, Dirichlet-to-Neumann maps,
and abstract Weyl functions,
\textit{Adv.\ Math.} \textbf{285} (2015), 1301--1338.

\bibitem{B17}
G.~Berkolaiko,
An elementary introduction to quantum graphs,
in: \textit{Geometric and Computational Spectral Theory},
Contemp.\ Math., vol.~700, Amer.\ Math.\ Soc.,
Providence, RI, 2017, pp.~41--72.

\bibitem{BK}
G.~Berkolaiko and P.~Kuchment,
\textit{Introduction to Quantum Graphs},
Mathematical Surveys and Monographs, vol.~186,
American Mathematical Society, Providence, RI, 2013.

\bibitem{BS80}
M.\,Sh.~Birman and M.\,Z.~Solomjak,
Asymptotic behavior of the spectrum of variational problems on solutions of
elliptic equations in unbounded domains,
\textit{Funktsional.\ Anal.\ i Prilozhen.} \textbf{14} (1980), 27--35 [Russian];
English translation: \textit{Funct.\ Anal.\ Appl.} \textbf{14} (1981), 267--274.

\bibitem{BK08}
D.~Borisov and D.~Krej\v{c}i\v{r}\'{\i}k,
$\mathcal{P}\mathcal{T}$-symmetric waveguides,
\textit{Integral Equations Operator Theory} \textbf{62} (2008), 489--515.

\bibitem{BK12}
D.~Borisov and D.~Krej\v{c}i\v{r}\'{\i}k,
The effective Hamiltonian for thin layers with non-Hermitian Robin-type boundary conditions,
\textit{Asymptot.\ Anal.} \textbf{76} (2012), 49--59.

\bibitem{BZ17}
D.~Borisov and M.~Znojil,
On eigenvalues of a $\mathcal{P}\mathcal{T}$-symmetric operator in a thin layer [Russian],
\textit{Mat.\ Sb.} \textbf{208} (2017), 3--30;
English translation: \textit{Sb.\ Math.} \textbf{208} (2017), 173--199.

\bibitem{BEKS94}
J.\,F.~Brasche, P.~Exner, Yu.\,A.~Kuperin and P.~\v{S}eba,
Schr\"odinger operators with singular interactions,
\textit{J.\ Math.\ Anal.\ Appl.} \textbf{184} (1994), 112--139.

\bibitem{B59}
F.\,E.~Browder,
Estimates and existence theorems for elliptic boundary value problems,
\textit{Proc.\ Nat.\ Acad.\ Sci.\ U.S.A.} \textbf{45} (1959), 365--372.

\bibitem{B61}
F.\,E.~Browder,
On the spectral theory of elliptic differential operators. I,
\textit{Math.\ Ann.} \textbf{142} (1961), 22--130.

\bibitem{BGW09}
B.\,M.~Brown, G.~Grubb and I.\,G.~Wood,
$M$-functions for closed extensions of adjoint pairs of
operators with applications to elliptic boundary problems,
\textit{Math.\ Nachr.} \textbf{282} (2009), 314--347.

\bibitem{BMNW08}
B.\,M.~Brown, M.~Marletta, S.~Naboko and I.\,G.~Wood,
Boundary triplets and $M$-functions for non-selfadjoint operators,
with applications to elliptic PDEs and block operator matrices,
\textit{J.\ London Math.\ Soc.\ (2)} \textbf{77} (2008), 700--718.

\bibitem{BMNW17}
B.\,M.~Brown, M.~Marletta, S.~Naboko and I.~Wood,
Inverse problems for boundary triples with applications,
\textit{Studia Math.} \textbf{237} (2017), 241--275.

\bibitem{B76}
V.\,M.~Bruk,
A certain class of boundary value problems with a spectral parameter in the boundary condition,
\textit{Mat.\ Sb.\ (N.S.)} \textbf{100 (142)} (1976), 210--216.
% correct citation: (?)
% On a class of boundary value problems with spectral parameter in the boundary condition
%\textit{Math.\ USSR Sb.} \textbf{29} (1976), 186--192.

\bibitem{BGP08}
J.~Br\"uning, V.~Geyler and K.~Pankrashkin,
Spectra of self-adjoint extensions and applications to solvable Schr\"odinger operators,
\textit{Rev.\ Math.\ Phys.} \textbf{20} (2008), 1--70.

\bibitem{C98}
R.~Carlson,
Adjoint and self-adjoint differential operators on graphs,
\textit{Electron.\ J.\ Differential Equations} \textbf{1998}, no.~6, 10~pp.

%\bibitem{Ca02}
%G.~Carron,
%Determinant relatif et la fonction Xi.
%\textit{Amer.\ J.\ Math.} \textbf{124} (2002), 307--352.

\bibitem{CKS17}
K.\,D.~Cherednichenko, A.\,V.~Kiselev and L.\,O.~Silva,
Functional model for extensions of symmetric operators and applications to scattering theory,
to appear in \textit{Netw.\ Heterog.\ Media},
arXiv:1703.06220.

\bibitem{ConwayBuch}
J.\,B.~Conway,
\textit{Functions of One Complex Variable I}, Second edition.
Graduate Texts in Mathematics, vol.~11, Springer, 1978.

\bibitem{C17}
J.-C.~Cuenin,
Eigenvalue bounds for Dirac and fractional Schr\"{o}dinger operators with complex potentials,
\textit{J.\ Funct.\ Anal.} \textbf{272} (2017), 2987--3018.

\bibitem{CT16}
J.-C.~Cuenin and C.~Tretter,
Non-symmetric perturbations of self-adjoint operators,
\textit{J.\ Math.\ Anal.\ Appl.} \textbf{441} (2016), 235--258.

%\bibitem{DL91}
%A.\,A.~Danielyan and B.\,M.~Levitan,
%Asymptotic behavior of the Weyl--Titchmarsh m-function [Russian],
%\textit{Izv.\ Akad.\ Nauk SSSR Ser.\ Mat.} \textbf{54} (1990), 469--479;
%English translation: \textit{Math.\ USSR-Izv.} \textbf{36} (1991), 487-–496.

\bibitem{D02}
E.\,B.~Davies,
Non-self-adjoint differential operators,
\textit{Bull.\ London Math.\ Soc.} \textbf{34} (2002), 513--532.

\bibitem{DHK09}
M.~Demuth, M.~Hansmann and G.~Katriel,
On the discrete spectrum of non-selfadjoint operators,
\textit{J.\ Funct.\ Anal.} \textbf{257}  (2009), 2742--2759.

\bibitem{DHMS06}
V.\,A.~Derkach, S.~Hassi, M.\,M.~Malamud and H.~de~Snoo,
Boundary relations and their Weyl families,
\textit{Trans.\ Amer.\ Math.\ Soc.} \textbf{358} (2006), 5351--5400.

\bibitem{DHMS09}
V.\,A.~Derkach, S.~Hassi, M.\,M.~Malamud and H.~de~Snoo,
Boundary relations and generalized resolvents of symmetric operators,
\textit{Russ.\ J.\ Math.\ Phys.} \textbf{16} (2009), 17--60.

\bibitem{DHMS12}
V.\,A.~Derkach, S.~Hassi, M.\,M.~Malamud and H.~de~Snoo,
Boundary triplets and Weyl functions. Recent developments,
in: \textit{Operator Methods for Boundary Value Problems},
London Math.\ Soc.\ Lecture Note Series, vol.~404, 2012, pp.~161--220.

\bibitem{DM91}
V.\,A.~Derkach and M.\,M.~Malamud,
Generalized resolvents and the boundary value problems for Hermitian operators with gaps,
\textit{J.\ Funct.\ Anal.} \textbf{95} (1991), 1--95.

\bibitem{DM95}
V.\,A.~Derkach and M.\,M.~Malamud,
The extension theory of Hermitian operators and the moment problem,
\textit{J.\ Math.\ Sci.} \textbf{73} (1995), 141--242.

\bibitem{EE87}
D.\,E.~Edmunds and W.\,D.~Evans,
\textit{Spectral Theory and Differential Operators}.
Oxford Mathematical Monographs. The Clarendon Press, Oxford University Press,
New York, 1987.

\bibitem{ET96}
D.\,E.~Edmunds and H.~Triebel,
\textit{Function Spaces, Entropy Numbers, Differential Operators.}
Cambridge Tracts in Mathematics, vol.~120,
Cambridge University Press, Cambridge, 1996.

\bibitem{EKK16}
Y.~Ershova, I.~Karpenko and A.~Kiselev,
Isospectrality for graph Laplacians under the change of coupling at graph vertices,
\textit{J.\ Spectr.\ Theory} \textbf{6} (2016), 43--66.

\bibitem{EK12}
Y.~Ershova and A.~Kiselev,
Trace formulae for graph Laplacians with applications to recovering matching conditions,
\textit{Methods Funct.\ Anal.\ Topology} \textbf{18} (2012), 343--359.

\bibitem{EK14}
Y.~Ershova and A.~Kiselev,
Trace formulae for Schr\"odinger operators on metric graphs with
applications to recovering matching conditions,
\textit{Methods Funct.\ Anal.\ Topology} \textbf{20} (2014), 134--148.

\bibitem{E08}
P.~Exner,
Leaky quantum graphs: a review,
in: \textit{Analysis on Graphs and its Applications}.
Selected papers based on the Isaac Newton Institute for Mathematical Sciences programme,
Cambridge, UK, 2007, \textit{Proc.\ Symp.\ Pure Math.} \textbf{77} (2008), 523--564.

\bibitem{EF09}
P.~Exner and M.~Fraas,
On geometric perturbations of critical Schr\"odinger operators with
a surface interaction,
\emph{J.\ Math.\ Phys.} \textbf{50} (2009), 112101, 12~pp.
	
\bibitem{EI01}
P.~Exner and T.~Ichinose,
Geometrically induced spectrum in curved leaky wires,
\textit{J.\ Phys.\ A} \textbf{34} (2001), 1439--1450.

\bibitem{EJ12}
P.~Exner and M.~Jex,
On the ground state of quantum graphs with attractive $\delta$-coupling,
\textit{Phys.\ Lett.\ A} \textbf{376} (2012), 713--717.

\bibitem{EK03}
P.~Exner and S.~Kondej,
Bound states due to a strong $\delta$-interaction supported by a curved surface,
\textit{J.\ Phys.\ A} \textbf{36} (2003), 443--457.	

\bibitem{EK15}
P.~Exner and H.~Kova{\v{r}}{\'{\i}}k,
\textit{Quantum Waveguides},
Springer, Heidelberg, 2015.

\bibitem{ER16}
P.~Exner and J.~Rohleder,
Generalized interactions supported on hypersurfaces,
\textit{J.\ Math.\ Phys.} \textbf{57} (2016), 041507, 23~pp.

\bibitem{EY02}
P.~Exner and K.~Yoshitomi,
Asymptotics of eigenvalues of the Schr\"odinger operator with a
strong $\delta$-interaction on a loop,
\textit{J.\ Geom.\ Phys.} \textbf{41} (2002), 344--358.

\bibitem{FKV17}
L.~Fanelli, D.~Krej\v{c}i\v{r}\'{\i}k and L.~Vega,
Spectral stability of Schr\"odinger operators with subordinated complex potentials,
\textit{J.\ Spectr.\ Theory}, to appear,
arXiv:1506.01617.

\bibitem{F11}
R.\,L.~Frank,
Eigenvalue bounds for Schr\"odinger operators with complex potentials,
\textit{Bull.\ Lond.\ Math.\ Soc.} \textbf{43} (2011), 745--750.

\bibitem{Fr15}
R.\,L.~Frank,
Eigenvalues of Schr\" odinger operators with complex surface potentials,
in: \textit{Functional Analysis and Operator Theory for Quantum Physics},
%J.~Dittrich, et al. (eds.),
Europ.\ Math.\ Soc., Z\"urich, 2017, pp.~245--259.
%preprint: arXiv:1512.09067.

\bibitem{F62}
R.~Freeman,
Closed extensions of the Laplace operator determined by a
general class of boundary conditions, for unbounded regions,
\textit{Pacific\ J.\ Math.} \textbf{12} (1962), 121--135.

\bibitem{F67}
R.~Freeman,
Closed operators and their adjoints associated with
elliptic differential operators,
\textit{Pacific J.\ Math.} \textbf{22} (1967), 71--97.

\bibitem{GS15}
J.~Galkowski and H.\,F.~Smith,
Restriction bounds for the free resolvent and resonances in lossy scattering,
\textit{Int.\ Math.\ Res.\ Notices} \textbf{16} (2015), 7473--7509.

\bibitem{GLMZ05}
F.~Gesztesy, Y.~Latushkin, M.~Mitrea and M.~Zinchenko,
Nonselfadjoint operators, infinite determinants, and some applications,
\textit{Russ.\ J.\ Math.\ Phys.} \textbf{12} (2005), 443--471.

\bibitem{GM08}
F.~Gesztesy and M.~Mitrea,
Generalized Robin boundary conditions, Robin-to-Dirichlet maps, Kre\u{\i}n-type
resolvent formulas for Schr\"{o}dinger operators on bounded Lipschitz domains,
in: \textit{Perspectives in Partial Differential Equations, Harmonic Analysis and Applications},
Proc.\ Sympos.\ Pure Math., vol.~79, Amer.\ Math.\ Soc., Providence, RI, 2008,
pp.~105--173.

\bibitem{GM11}
F.~Gesztesy and M.~Mitrea,
A description of all self-adjoint extensions of the
Laplacian and Kre\u{\i}n-type resolvent formulas on non-smooth domains,
\textit{J.\ Anal.\ Math.} \textbf{113} (2011), 53--172.

\bibitem{GMZ07}
F.~Gesztesy, M.~Mitrea and M.~Zinchenko,
Variations on a theme of Jost and Pais,
\textit{J.\ Funct.\ Anal.} \textbf{253} (2007), 399--448.

\bibitem{GMZ09}
F.~Gesztesy, M.~Mitrea and M.~Zinchenko,
On Dirichlet-to-Neumann maps and some applications to modified Fredholm determinants,
in: \textit{Methods of Spectral Analysis in Mathematical Physics},
Oper.\ Theory Adv.\ Appl., vol.~186, Birkh\"{a}user Verlag, Basel, 2009,
pp.~191--215.

\bibitem{GK69}
I.\,C.~Gohberg and M.\,G.~Kre\u\i{}n,
\textit{Introduction to the Theory of Linear Nonselfadjoint Operators}.
Transl.\ Math.\ Monogr., vol.~18.,
Amer.\ Math.\ Soc., Providence, RI, 1969.

%\bibitem{GMO14}
%L.~Golinskii, M.~Malamud and L.~Oridoroga,
%Schoenberg matrices of radial positive definite functions and
%Riesz sequences of translates in $L^2(\dR^n)$,
%\textit{J.\ Fourier Anal.\ Appl.} \textbf{21} (2015), 915--960.

\bibitem{GG91}
V.\,I.~Gorbachuk and M.\,L.~Gorbachuk,
\textit{Boundary Value Problems for Operator Differential Equations}.
Kluwer Academic Publ., Dordrecht, 1991.

\bibitem{GK14}
A.~Grod and S.~Kuzhel,
Schr\"{o}dinger operators with non-symmetric zero-range potentials,
\textit{Methods Funct.\ Anal.\ Topology} \textbf{20} (2014), 34--49.

\bibitem{G68}
G.~Grubb,
A characterization of the non-local boundary value problems associated with an
elliptic operator,
\textit{Ann.\ Scuola Norm.\ Sup.\ Pisa (3)} \textbf{22} (1968), 425--513.

\bibitem{G84}
G.~Grubb,
Remarks on trace estimates for exterior boundary problems,
\textit{Comm.\ Partial Differential Equations} \textbf{9} (1984), 231--270.

\bibitem{G08}
G.~Grubb,
Krein resolvent formulas for elliptic boundary problems in nonsmooth domains,
\textit{Rend.\ Semin.\ Mat.\ Univ.\ Politec.\ Torino}~\textbf{66} (2008), 271--297.

\bibitem{G09}
G.~Grubb,
\textit{Distributions and Operators}.
Graduate Texts in Mathematics, vol.~252.
Springer, New York, 2009.

\bibitem{G11}
G.~Grubb,
Spectral asymptotics for Robin problems with a discontinuous coefficient,
\textit{J.\ Spectr.\ Theory} \textbf{1} (2011), 155--177.

\bibitem{H14}
A.~Hussein,
Maximal quasi-accretive Laplacians on finite metric graphs,
\textit{J.\ Evol.\ Equ.} \textbf{14} (2014), 477--497.

\bibitem{HKS15}
A.~Hussein, D.~Krej\v{c}i\v{r}\'{\i}k and P.~Siegl,
Non-self-adjoint graphs,
\textit{Trans.\ Amer.\ Math.\ Soc.} \textbf{367} (2015), 2921--2957.

\bibitem{ILP15}
A.~Ibort, F.~Lled\'{o} and J.\,M.~P\'{e}rez-Pardo,
Self-adjoint extensions of the Laplace--Beltrami operator and unitaries at the boundary,
\textit{J.\ Funct.\ Anal.} \textbf{268} (2015), 634--670.

\bibitem{KK68}
I.\,S.~Kac and M.\,G.~Krein,
On the spectral function of a string,
in: F.\,V.~Atkinson,
\textit{Discrete and Continuous Boundary Problems} (Russian translation),
Addition II, Mir, Moscow, pp.~648--737 (1968);
English translation: \textit{Amer.\ Math.\ Soc.\ Transl.} \textbf{2}\,(103) (1974), 19--102.

\bibitem{KK74}
I.\,S.~Kac and M.\,G.~Krein,
R-functions---analytic functions mapping the upper halfplane into itself,
\textit{Amer.\ Math.\ Soc.\ Transl.\ (2)} \textbf{103} (1974), 1--18.

\bibitem{KKT16}
G.~Karreskog, P.~Kurasov and I.~Trygg Kupersmidt,
Schr\"odinger operators on graphs: symmetrization and Eulerian cycles,
\textit{Proc.\ Amer.\ Math.\ Soc.} \textbf{144} (2016), 1197--1207.

\bibitem{K95}
T.~Kato,
\textit{Perturbation Theory for Linear Operators.}
Reprint of the 1980 edition. Classics in Mathematics.
Springer-Verlag, Berlin, 1995.

\bibitem{K75}
A.\,N.~Kochubei,
%A.\,N.~Ko\v{c}ube\u{\i},
Extensions of symmetric operators and symmetric binary relations [Russian],
\textit{Mat.\ Zametki} \textbf{17} (1975), 41--48;
English translation: \textit{Math.\ Notes} \textbf{17} (1975), 25--28.

\bibitem{K89}
A.\,N.~Kochubei,
One-dimensional point interactions [Russian],
\textit{Ukr.\ Mat.\ Zh.} \textbf{41} (1989), 1391--1395;
English translation: \textit{Ukr.\ Math.\ J.} \textbf{41} (1989), 1198--1201.

\bibitem{KK16}
S.~Kondej and D.~Krej\v{c}i\v{r}\'{\i}k,
Asymptotic spectral analysis in colliding leaky quantum layers,
\textit{J.\ Math.\ Anal.\ Appl.} \textbf{446} (2017), 1328--1355.
%\emph{arXiv:1606.07273}.

\bibitem{KL14}
S.~Kondej and V.~Lotoreichik,
Weakly coupled bound state of 2-D Schr\"odinger operator with potential-measure,
\textit{J.\ Math.\ Anal.\ Appl.} \textbf{420} (2014), 1416--1438.

\bibitem{K13}
Yu.\,G.~Kovalev,
1D nonnegative Schr\"odinger operators with point interactions,
\textit{Mat.\ Stud.} \textbf{39} (2013), 150--163.

\bibitem{KM10}
A.~Kostenko and M.~Malamud,
1-D Schr\"odinger operators with local point interactions on a discrete set,
\textit{J.\ Differential Equations} \textbf{249} (2010), 253–-304.

\bibitem{KM13}
A.~Kostenko and M.~Malamud,
1-D  Schr\"odinger operators with local point interactions: a review,
in: \textit{Spectral Analysis, Differential Equations and Mathematical Physics:
A Festschrift in Honor of Fritz Gesztesy's 60th Birthday},
Proc.\ Symp.\ Pure Math., vol.~87,
Amer.\ Math.\ Soc., Providence, RI, 2013, pp.~235--262.

\bibitem{KS99}
V.~Kostrykin and R.~Schrader,
Kirchhoff's rule for quantum wires,
\textit{J.\ Phys.\ A} \textbf{32} (1999), 595--630.

\bibitem{KS06}
V.~Kostrykin and R.~Schrader,
Laplacians on metric graphs: eigenvalues, resolvents and semigroups,
in: \textit{Quantum Graphs and their Applications},
Contemp.\ Math., vol.~415,
Amer.\ Math.\ Soc., Providence, RI, 2006, pp.~201--225.

\bibitem{KP31}
R.~de~L.~Kronig and W.\,G.~Penney,
Quantum mechanics of electrons in crystal lattices,
\textit{Proc.\ Roy.\ Soc.\ Lond.\ A} \textbf{130} (1931), 499--513.

\bibitem{K04}
P.~Kuchment,
Quantum graphs. I. Some basic structures. %Special section on quantum graphs.
\textit{Waves Random Media} \textbf{14} (2004), S107--S128.

\bibitem{KZ16}
S.~Kuzhel and M.~Znojil,
Non-self-adjoint Schr\"{o}dinger operators with nonlocal one-point interactions,
\textit{Banach J.\ Math.\ Anal.} \textbf{11} (2017), 923--944.
%Quantum solvable models with nonlocal one point interactions,
%\emph{arXiv:1607.00350}.

\bibitem{LS09}
A.~Laptev and O.~Safronov,
Eigenvalue estimates for Schr\"{o}dinger operators with complex potentials,
\textit{Commun.\ Math.\ Phys.} \textbf{292} (2009), 29--54.

\bibitem{LSV14}
D.~Lenz, C.~Schubert and I.~Veseli\'c,
Unbounded quantum graphs with unbounded boundary conditions,
\textit{Math.\ Nachr.} \textbf{287} (2014), 962--979.

\bibitem{LM72}
J.~Lions and E.~Magenes,
\textit{Non-Homogeneous Boundary Value Problems and Applications I}.
Springer-Verlag, Berlin--Heidelberg--New York, 1972.

\bibitem{LO16}
V.~Lotoreichik and T.~Ourmi{\`e}res-Bonafos,
On the bound states of Schr\"{o}dinger operators with $\delta$-interactions
on conical surfaces,
\textit{Comm.\ Partial Differential Equations} \textbf{41} (2016), 999--1028.

\bibitem{LR12}
V.~Lotoreichik and J.~Rohleder,
Schatten--von Neumann estimates for resolvent differences of Robin Laplacians on a half-space,
in: \textit{Spectral Theory, Mathematical System Theory, Evolution Equations,
Differential and Difference Equations},
Oper.\ Theory Adv.\ Appl., vol.~221, Birkh\"auser/Springer Basel AG, Basel, 2012,
pp.~453--468.

\bibitem{LS17}
V.~Lotoreichik and P.~Siegl,
Spectra of definite type in waveguide models,
\textit{Proc.\ Amer.\ Math.\ Soc.} \textbf{145} (2017), 1231--1246.

\bibitem{M06}
M.\,M.~Malamud,
Operator holes and extensions of sectorial operators and dual pairs of contractions,
\textit{Math.\ Nachr.} \textbf{279} (2006), 625--655.

\bibitem{M10}
M.\,M.~Malamud,
Spectral theory of elliptic operators in exterior domains,
\textit{Russ.\ J.\ Math.\ Phys.} \textbf{17} (2010), 96--125.

\bibitem{MM03}
M.\,M.~Malamud and S.\,M.~Malamud,
Spectral theory of operator measures in a Hilbert space [Russian],
\textit{Algebra i Analiz} \textbf{15} (2003), 1--77;
English translation: \textit{St.\ Petersburg Math.\ J.} \textbf{15} (2004), 323--373.

\bibitem{MN11}
M.\,M.~Malamud and H.~Neidhardt,
On the unitary equivalence of absolutely continuous parts of self-adjoint extensions,
\textit{J.\ Funct.\ Anal.} \textbf{260} (2011), 613--638.

\bibitem{MPS16}
A.~Mantile, A.~Posilicano and M.~Sini,
Self-adjoint elliptic operators with boundary conditions on not closed hypersurfaces,
\textit{J.\ Differential Equations} \textbf{261} (2016), 1--55.

\bibitem{MS09}
V.\,G.~Maz'ya and T.\,O.~Shaposhnikova,
\textit{Theory of Sobolev Multipliers}.
Springer-Verlag, Berlin Heidelberg, 2009.

\bibitem{P06}
K.~Pankrashkin,
Spectra of Schr\"odinger operators on equilateral quantum graphs,
\textit{Lett.\ Math.\ Phys.} \textbf{77} (2006), 139--154.

\bibitem{P87}
A.~Pietsch,
\textit{Eigenvalues and $s$-Numbers}.
Cambridge Studies in Advanced Mathematics, vol.~13.
Cambridge University Press, Cambridge, 1987.

\bibitem{P08}
A.~Posilicano,
Self-adjoint extensions of restrictions,
\textit{Oper.\ Matrices} \textbf{2} (2008), 1--24.

\bibitem{Post08}
O.~Post,
Equilateral quantum graphs and boundary triples,
in: \textit{Analysis on Graphs and its Applications},
Proc.\ Sympos.\ Pure Math., vol.~77,
Amer.\ Math.\ Soc., Providence, RI, 2008, pp.~469--490.

\bibitem{P16}
O.~Post,
Boundary pairs associated with quadratic forms,
\textit{Math.\ Nachr.} \textbf{289} (2016), 1052--1099.

\bibitem{S12}
K.~Schm\"{u}dgen,
\textit{Unbounded Self-Adjoint Operators on Hilbert Space}.
Springer, Dordrecht, 2012.

\bibitem{SSVW15}
C.~Schubert, C.~Seifert, J.~Voigt and M.~Waurick,
Boundary systems and (skew-)self-adjoint operators on infinite metric graphs,
\textit{Math.\ Nachr.} \textbf{288} (2015), 1776--1785.

\bibitem{T}
G.~Teschl,
\textit{Mathematical Methods in Quantum Mechanics.
With Applications to Schr\"odinger Operators}.
American Mathematical Society, Providence, 2009.

\bibitem{V52}
M.\,I.~Vishik,
On general boundary problems for elliptic differential equations [Russian],
\textit{Trudy Moskov.\ Mat.\ Ob\v s\v c.} \textbf{1} (1952), 187--246.

\bibitem{Wei}
J.~Weidmann,
\textit{Lineare Operatoren in Hilbertr\"aumen. Teil 1. Grundlagen} [German],
Mathematische Leitf\"aden, B.\,G.~Teubner, Stuttgart, 2000.

\bibitem{W00}
H.~Winkler,
Spectral estimations for canonical systems,
\textit{Math.\ Nachr.} \textbf{220} (2000), 115--141.

\bibitem{W10}
C.~Wyss,
Riesz bases for $p$-subordinate perturbations of normal operators,	
\textit{J.\ Funct.\ Anal.} \textbf{258} (2010),  208--240.

%\bibitem{Y10}
%D.\,R.~Yafaev,
%\textit{Mathematical Scattering Theory. Analytic Theory}.
%Mathematical Surveys and Monographs, vol.~158,
%American Mathematical Society, Providence, RI, 2010.


\end{thebibliography}
\end{document}